\documentclass[reqno, intlimits, sumlimits]{amsart} 

\usepackage{amsmath, amstext, amssymb, amsthm, bbm,lmodern, enumerate, esint}
\usepackage[left=3cm,right=3cm, top=4cm, bottom=3cm]{geometry}

\usepackage[hidelinks]{hyperref}
\usepackage[toc = false, symbols,nogroupskip,sort=def]{glossaries-extra}

\allowdisplaybreaks

    \renewcommand{\phi}{\varphi}
    \renewcommand{\epsilon}{\varepsilon}
    \renewcommand{\P}{\mathsf{P}}
 
    \newcommand	{\eins}	{\mathbbm{1}}   
	\newcommand	{\norm}[1]	{\left\lVert#1\right\rVert}
	\newcommand	{\abs}[1]		{\left\lvert#1\right\rvert}
	\newcommand {\cost}	{\mathsf{cost}}
	\newcommand {\ent}	{\mathsf{Ent}}

	\newcommand {\stat}[1]		{(#1)^{{stat}}}
\newcommand{\Cpl}{ \mathsf{Cpl}}
\newcommand{\cpl}{ \mathsf{cpl}}
\newcommand{\Q}{ \mathsf{Q}}
\newcommand{\R}{ \mathsf{R}}
\newcommand{\sfQ}{ \mathsf{Q}}
\newcommand{\q}{ \mathsf{q}}
\newcommand{\C}{\mathsf{C}}
\newcommand{\W}{\mathsf{W}}
\newcommand{\Poi}{\mathsf{Poi}}
\newcommand{\Leb}{\mathsf{Leb}}
\newcommand{\proj}{\mathsf{proj}}
\newcommand{\pr}{\mathrm{pr}}

\newcommand{\refl}{\mathsf{ref}}
\newcommand{\Semi}{\mathsf{S}}
\newcommand{\id}{\mathsf{id}}
\newcommand{\F}{\mathcal{F}}

\newcommand\numberthis{\addtocounter{equation}{1}\tag{\theequation}}

	\DeclareMathOperator	{\IE}			{\mathbb{E}} 
	\DeclareMathOperator	{\IN}			{\mathbb{N}}
	\DeclareMathOperator	{\IP}			{\mathbb{P}}
	
	\DeclareMathOperator	{\IR}			{\mathbb{R}}
	\DeclareMathOperator	{\IZ}			{\mathbb{Z}}
	
	\DeclareMathOperator	{\supp}			{supp}

        \DeclareMathOperator	{\spp1}			{\mathcal{P}_s(\Gamma)}

	\theoremstyle{plain}
\newtheorem{thm}			{Theorem}[section]
\newtheorem{lem}	[thm]	{Lemma}
\newtheorem{cor}	[thm]	{Corollary}
\newtheorem{prop}	[thm]	{Proposition}

\theoremstyle{definition}
\newtheorem{defi}	[thm]	{Definition}
\newtheorem{ex}	    [thm]	{Example}
\newtheorem{rem}	[thm]	{Remark}

\numberwithin{equation}{section}

\makenoidxglossaries
\glsxtrRevertMarks 

\glsxtrnewsymbol[description={
Borel sets of $\IR^d$
}]{B}{\ensuremath{\mathcal B (\IR^d)}}
\glsxtrnewsymbol[description={The box $[-n/2,n/2]^d$}]{Lambda}{\ensuremath{\Lambda_n}} 
\glsxtrnewsymbol[description={Compactly supported continuous functions on some space}]{C_c}{\ensuremath{\mathcal C_c(\cdot)}}
 \glsxtrnewsymbol[description={Bounded continuous functions on some space}]{C_b}{\ensuremath{\mathcal C_b(\cdot)}}
\glsxtrnewsymbol[description={Radon measures, also $\mathcal M$}]{M}{\ensuremath{\mathcal M (\IR^d)}}
\glsxtrnewsymbol[description={Configuration space, i.e. locally finite subsets and/or counting measures in $\IR^d$}]{Gamma}{\ensuremath{\Gamma}}
\glsxtrnewsymbol[description={
Label map, mostly via lexicographic ordering
}]{l}{\ensuremath{\ell}}
\glsxtrnewsymbol[description={Random measure, i.e.~a random variable in $\mathcal M(\IR^d)$, also $\xi$}]{xi}{\ensuremath{\xi^\bullet}}
\glsxtrnewsymbol[description={Probability distributions on $\mathcal M(\IR^d)$, also $\mathcal P$}]{P(M)}{\ensuremath{\mathcal P(\mathcal M (\IR^d))}}
\glsxtrnewsymbol[description={Push-forward of a measure $\mu$ under a map $T$}]{push}{\ensuremath{T_{\#}\mu}}
\glsxtrnewsymbol[description={Shift operator on $\mathcal M$ by $z\in\IR^d$, i.e.~$(\theta_z\xi)(A)=\xi(A+z)$ or $\theta_z\delta_z=\delta_0$}]{theta}{\ensuremath{\theta_z}}\glsxtrnewsymbol[description={Flow on the underlying probability space $(\Omega,\mathcal F, \mathbb P)$}]{thetaO}{\ensuremath{\theta^\Omega}}
\glsxtrnewsymbol[description={Stationary distributions on $\mathcal M(\IR^d)$, i.e. $\P$ with $(\theta_z)_\#\P= \P$}]{Ps}{\ensuremath{\mathcal P_s}}
\glsxtrnewsymbol[description={The homogeneous Poisson point process on $\IR^d$}]{Poi}{\ensuremath{\mathsf{Poi}}}
\glsxtrnewsymbol[description={Stationary couplings $\Q$ of $\P_0$ and $\P_1$}]{Cpl_s}{\ensuremath{\Cpl_s(\P_0,\P_1)}}
\glsxtrnewsymbol[description={Transport cost between distributions $\P_0,\P_1\in\mathcal P_s$, also $\C_p$ for $\vartheta(x)=|x|^p$}]{C}{\ensuremath{\C (\P_0,\P_1)}}
\glsxtrnewsymbol[description={
Transport cost between distributions on $\Lambda_n$ having $k$ points
}]{Ckn}{\ensuremath{\C_{\Lambda_n}^k}}
\glsxtrnewsymbol[description={Couplings $\q$ between measures $\xi,\eta\in\mathcal M(\IR^d)$}]{cpl}{\ensuremath{\cpl(\xi,\eta)}}
\glsxtrnewsymbol[description={Equivariant couplings $\q^\bullet$ (also $\q$) between jointly equivariant random measures}]{cple}{\ensuremath{\cpl_e(\xi^{\bullet},\eta^{\bullet})}}
\glsxtrnewsymbol[description={Transport cost per volume between measures $\xi,\eta\in\mathcal M(\IR^d)$, see \eqref{eq:costMeasures}}]{c}{\ensuremath{c (\xi,\eta)}}
\glsxtrnewsymbol[description={Expected transport cost per volume between random measures $\xi^\bullet,\eta^\bullet$, see \eqref{eq:costRandomMeasures}}]{frakc}{\ensuremath{\cost(\xi^\bullet,\eta^\bullet)}}
\glsxtrnewsymbol[description={Wasserstein distance $\C_p^{1/p}$ on $\mathcal P_s$}]{Wp}{\ensuremath{\W_p}}
\glsxtrnewsymbol[description={
Wiener measure on $\mathcal C\big([0,\infty),\IR^d\big)$}]{Wiener}{\ensuremath{\mathbb W}}
\glsxtrnewsymbol[description={Relative entropy, i.e. $\mathsf{Ent}(\mu|\nu)=\int\log(\frac{d\mu}{d\nu})d\mu$ if $\mu\ll \nu$}]{Ent}{\ensuremath{\mathsf{Ent}}}
\glsxtrnewsymbol[description={Specific relative entropy with respect to $\Poi$, see \eqref{eq:specific_rel_ent}}]{specE}{\ensuremath{\mathcal E}}
\glsxtrnewsymbol[description={Relative Fisher information, i.e. $I(\mu|\nu)=\int\lvert\nabla\sqrt{{d\mu}/{d\nu}}\rvert^2d\mu$ if $\mu\ll \nu$}]{I}{\ensuremath{I}}
\glsxtrnewsymbol[description={Specific relative Fisher information with respect to $\Poi$, see \eqref{eq:specific_rel_ent}}]{specI}{\ensuremath{\mathcal I}}
\glsxtrnewsymbol[description={Restriction of $\xi\in\mathcal M(\IR^d)$ to a set $A\in\mathcal B (\IR^d)$, i.e.~$\mathrm{pr}_A(\xi)=\xi(A\cap \cdot)$}]{pr_A}{\ensuremath{\mathrm{pr}_A}}
\glsxtrnewsymbol[description={Restriction of $\P\in\mathcal P_s$ to a set $A\in\mathcal B (\IR^d)$, i.e.~$(\mathrm{pr}_A)_{\#}\P$}]{P_A}{\ensuremath{\P_A}}
\glsxtrnewsymbol[description={Geodesic in $t\in[0,1]$ between distributions $\P^0,\P^1\in\mathcal P_s$ with finite cost $\C_p$}]{Pt}{\ensuremath{\P_t}}
\glsxtrnewsymbol[description={The distribution of a tiled point process}]{Ptil}{\ensuremath{\P^{til}}}
 \glsxtrnewsymbol[description={The distribution of a stationarized point process}]{Pstat}{\ensuremath{\P^{stat}}}
 \glsxtrnewsymbol[description={
 The distribution of a point process conditioned on having $k$ points
 }]{Pk}{\ensuremath{\P^k}}
 \glsxtrnewsymbol[description={
Modified process obtained from Theorem \ref{thm:Modification} for $\Lambda_n$,  similarly $\tilde{\Q}_{\Lambda_n},\tilde{\mathsf R}_{\Lambda_n}$
}]{tildeP}{\ensuremath{\tilde{\P}_{\Lambda_n}}}
 \glsxtrnewsymbol[description={
The optimal pair of distribution- and process-coupling, minimizing $\C$
}]{Qq}{\ensuremath{(\Q,\q^\bullet)}}
\glsxtrnewsymbol[description={Semigroup of attaching Brownian motion to points, see also \eqref{def_semi_box} for $\Semi^{\Lambda_n}_t$}]{Semi}{\ensuremath{\Semi_t}}
\glsxtrnewsymbol[description={Gradient flow of a stationary distribution $\P$ at time $t\ge 0$}]{SemiP}{\ensuremath{\Semi_t\P}}
\glsxtrnewsymbol[description={The counting map, mostly on $\Lambda_n$, that is $\pi:\Gamma\to\IN_0, \xi \mapsto \xi(\IR^d)$}]{pi}{\ensuremath{\pi}} 

\begin{document}

\author{Matthias Erbar}
\address{Matthias Erbar: Faculty of Mathematics, University of Bielefeld, Universit\"atsstrasse 25, 33615 Bielefeld, Germany}
\email{erbar@math.uni-bielefeld.de}

\author{Martin Huesmann}
\address{Martin Huesmann: Institute for Mathematical Stochastics,
University of M\"unster
Orl\'eans-Ring 10,
48149 M\"unster, Germany}
\email{martin.huesmann@uni-muenster.de}

\author{Jonas Jalowy}
\address{Jonas Jalowy: Institute for Mathematical Stochastics,
University of M\"unster
Orl\'eans-Ring 10,
48149 M\"unster, Germany}
\email{jjalowy@uni-muenster.de}

\author{Bastian M\"uller}
\address{Bastian M\"uller: Institute for Mathematical Stochastics,
University of M\"unster
Orl\'eans-Ring 10,
48149 M\"unster, Germany}
\email{bastian.mueller@uni-muenster.de}
\thanks{ME, MH, JJ are supported by the Deutsche Forschungsgemeinschaft (DFG, German Research Foundation) through the SPP 2265 {\it Random Geometric Systems}. MH and BM have been funded by the Deutsche Forschungsgemeinschaft (DFG, German Research Foundation) under Germany's Excellence Strategy EXC 2044 -390685587, Mathematics M\"unster: Dynamics--Geometry--Structure . }

\title[Optimal transport of stationary point processes]{Optimal transport of stationary point processes: Metric structure, gradient flow and convexity of the specific entropy}


\begin{abstract}
We develop a theory of optimal transport for stationary random measures with a focus on stationary point processes and construct a family of distances on the set of stationary random measures. These induce a natural notion of interpolation between two stationary random measures along a shortest curve connecting them.
In the setting of stationary point processes we leverage this transport distance to give a geometric interpretation for the evolution of infinite particle systems with stationary distribution. Namely, we characterise the evolution of infinitely many Brownian motions as the gradient flow of the specific relative entropy w.r.t.~the Poisson point process. Further, we establish displacement convexity of the specific relative entropy along optimal interpolations of point processes and establish an stationary analogue of the HWI inequality, relating specific entropy, transport distance, and a specific relative Fisher information.
\end{abstract}

\maketitle

\tableofcontents

\section{Introduction}

Optimal transport has proven to be a powerful tool in the analysis of interacting particle systems and the associated PDEs. In this work, we develop a counterpart to the rich theory of optimal transport in the setting of stationary random measures with a particular focus on stationary point processes, i.e. stationary random infinite point configurations.

We construct and analyse a class of transportation distances between stationary random measures. Given random measures on $\IR^d$ whose distribution is stationary under shifts, these distances are obtained by minimising the expectation of suitable cost functions on the space of locally finite measures on $\IR^d$ over all possible joint distributions of the random measures under the constraint that also the joint distribution be stationary. The transport distances we obtain turn out the be geodesic distances and lead to a natural notion of interpolation between stationary random measures, in analogy to the famous displacement interpolation of McCann \cite{McCann} in classical optimal transport.

Our motivation for considering this novel geometry is twofold. On the one hand we want to investigate convexity properties of functionals of stationary point processes along displacement interpolations. Such properties have already been leveraged successfully characterising the sine$_\beta$ process as the unique minimiser of the renormalised free energy \cite{EHL21}, very much in the spirit of McCann's original work. We also expect displacement convexity to provide a general strategy to obtain functional inequalities for stationary point processes, which are available only in very special cases to date.
On the other hand we aim to provide the tools to analyse dynamics of infinite interacting particle systems by interpreting them geometrically as steepest descents of suitable free energy functionals and extending the powerful Otto calculus \cite{Ott01} to this setting. 

Indeed, in the present work we characterise the evolution of infinitely many Brownian motions starting from a stationary initial condition as the gradient flow of the specific relative entropy w.r.t.~the Poisson process. This provides an analogue of the celebrated results of Jordan-Kinderlehrer-Otto \cite{JKO98} for stationary point processes. Moreover, we establish displacement convexity of the specific relative entropy along optimal interpolations w.r.t.~one of our new transport distances taking over the role of the Wasserstein distance in the setting of stationary random measures.

Let us describe the setting we consider and the results we obtain in some more detail.

\subsection{Optimal transport for stationary random measures}

In this paper, a random measure $\xi^\bullet$ will be a random variable with values in the space $\gls{M}$ of locally finite measures on $\IR^d$.  We say that $\xi^\bullet$ is a
point process if $\xi^\bullet(A)\in\IN_0$ almost surely for all compact
$A\in\mathcal B(\IR^d)$.
The distribution of a random measure is an element of $\mathcal P\big(\mathcal M(\IR^d)\big)$ the set of probability measures over $\mathcal M(\IR^d)$. 
Note that $\IR^d$ naturally acts on $\mathcal M(\IR^d)$ by shift of the support, namely for $x\in \IR^d$ and $\xi\in \mathcal M(\IR^d)$ we define $\glsdisp{theta}{\theta_x}\xi\in \mathcal M(\IR^d)$ by $\theta_x\xi(A)=\xi(A+x)$ for all measurable sets $A$. We say that $\P\in \mathcal P\big(\mathcal M(\IR^d)\big)$ is \emph{stationary} if $\P\circ \theta_x^{-1}=\P$ for all $x\in \IR^d$. We denote by $\gls{Ps}\big(\mathcal M(\IR^d)\big)$ the set of stationary distributions. A random measure is called stationary if its distribution is.

Stationarity of the distribution of a random measure is implied by the following stronger property.  A random measure $\xi^\bullet:(\Omega,\F,\IP)\to\mathcal M(\IR^d)$ is called \emph{invariant} if the probability space $(\Omega,\F,\IP)$ admits a
measurable flow, i.e.~a
family of measurable mappings $\IR^d\times\Omega\ni(x,\omega)\mapsto \theta_x\omega\in\Omega$ with $\theta_0=\id$ and $\theta_x\circ\theta_y=\theta_{x+y}$ for all $x,y\in\IR^d$, such that $\IP$ is invariant under $\theta$ and for all $x\in\IR^d$,
$\omega\in\Omega$, and $A\in\mathcal B(\IR^d)$ it  holds
\[ \xi^{\omega}(A)=\xi^{\theta_x\omega}(A-x)\;.\]

We will now construct a transport distance on the space $\mathcal P_s\big(\mathcal M(\IR^d)\big)$. We first consider a suitable cost function $c:\mathcal M(\IR^d)\times\mathcal M(\IR^d)\to [0,\infty]$. This cost is defined in terms of a transport problem itself. However, since the measures involved are infinite, the key idea is to consider the transport cost per volume. Denoting by $\Lambda_n=[-n/2,n/2]^d$ the box of side length $n$ centred at the origin, we set for $p\geq 1$ and $\xi,\eta\in \mathcal M(\IR^d)$

\begin{align}\label{eq:intro-ground-cost}
\gls{c}=\inf_{\mathsf q\in\cpl(\xi,\eta)}  \limsup_{n\to\infty}\frac{1}{n^d} \int_{\Lambda_n\times \IR^d}|x-y|^p \mathsf q(dx,dy) \;,
\end{align}
where $\cpl(\xi,\eta)$ denotes the set of all couplings $\q\in \mathcal M(\IR^d\times \IR^d)$ of $\xi$ and $\eta$.

Consider now  $\P_0,\P_1\in \mathcal P_s\big(\mathcal M(\IR^d)\big)$ two distributions of stationary random measures. Denote by $\Cpl_s(\P_0,\P_1)$ the set of all couplings $\Q\in \mathcal P\big(\mathcal M(\IR^d)\times\mathcal M(\IR^d)\big)$ between $\P_0$ and $\P_1$ that are stationary in the sense that $\Q\circ(\theta_x,\theta_x)^{-1}=\Q$ for all $x\in \IR^d$.
 
 Then, we put

\begin{align}\label{eq:stationaryOT}
  \gls{C}=\inf_{\mathsf Q\in \Cpl_s(\P_0,\P_1)} \int c(\xi,\eta) \mathsf Q(d\xi,d\eta).
\end{align}
Equivalently, we can write this optimisation problem as
\begin{align*}
\mathsf C(\P_0,\P_1)=\inf_{(\xi^\bullet,\eta^\bullet)}\IE\big[c(\xi^\bullet,\eta^\bullet)\big]\;,
\end{align*}
where the infimum is taken over all jointly invariant random measures $(\xi^\bullet,\eta^\bullet)$ defined on a common probability space with distribution $\P_0$ and $\P_1$ respectively, since their joint laws are precisely the stationary couplings of $\P_0$ and $\P_1$.

Note that \eqref{eq:stationaryOT} is a two layer optimisation problem. The first layer of optimisation is on the level of the coupling of the distributions $\P_0,\P_1$. The second layer is on the level of the coupling of the realisations of the random measures $\xi^\bullet\sim\P_0$ and $\eta^\bullet\sim\P_1$ in the transport problem defining $c$. Moreover, observe, that \eqref{eq:stationaryOT} is an optimal transport problem with an additional probabilistic constraint, namely stationarity (see Lemma \ref{lem:charStat}). A direct consequence is compactness of the set of all stationary couplings such that existence of an optimal stationary coupling $\Q$ follows by the direct method of the calculus of variations once l.s.c.\ of $c$ is established. Moreover, one can derive a duality result (see Proposition \ref{prop:duality}) for instance by adapting the minmax  argument of \cite{BeHLPe13}.
In the following we fix $p\geq 1$ in \eqref{eq:intro-ground-cost} and put $\mathsf W_p:=\mathsf C^{\frac1p}$.

\begin{thm}
$\mathsf W_p$ defines a geodesic extended distance on the space of stationary distributions $\mathcal P_s(\mathcal M(\IR^d))$ with unit intensity. 
\end{thm}

Here, extended distance means that $\W_p$ might attain the value $+\infty$. This result is more subtle than it looks at first, since for instance it is not clear whether the cost function $c(\xi,\eta)$ is symmetric. Hence, it is not obvious that $\mathsf W_p$ is symmetric. The solution  lies in another representation formula for $\mathsf W_p$ (resp.\ $\mathsf C$). To this end, we return to the cost function \eqref{eq:intro-ground-cost} which has been analysed for jointly invariant random measures with fixed joint distribution by Sturm and the second author in \cite{HS13,H16}. For two jointly invariant random measures $\xi^\bullet,\eta^\bullet$ defined on the same probability space $(\Omega,\F,\IP)$, we say that a random measure $\mathsf q^\bullet:(\Omega,\F,\IP)\to \mathcal M(\IR^d\times\IR^d)$ is an invariant/equivariant coupling of $\xi^\bullet$ and
$\eta^\bullet$ iff for all $A\in\mathcal B(\IR^d)$
\begin{align*}
    \mathsf q^{\omega}(A\times \IR^d)=\xi^{\omega}(A)  \qquad \text{ and }\qquad \mathsf q^{\omega}(\IR^d\times A)=\eta^\omega(A)
\end{align*}
and for all $x\in\IR^d$ and $A,B\in\mathcal B(\IR^d)$ we have
$\mathsf q^{\omega}(A\times B)=\mathsf q^{\theta_x\omega}(A-x\times B-x)$. 
Denote by $\cpl_e(\xi^\bullet,\eta^\bullet)$ the set of all such couplings. Note that a random coupling $\q^\bullet\in\cpl_e(\xi^\bullet,\eta^\bullet)$ induces a coupling $\Q\in\Cpl_s(\P_0,\P_1)$ between the distributions of $\xi^\bullet,\eta^\bullet$ by considering the joint law of the marginals $(\xi^\bullet,\eta^\bullet)$. We call $\Q$ a \emph{distribution coupling} and $\q^\bullet$ a \emph{process coupling}.

If we put
\begin{align*}
\cost(\xi^\bullet,\eta^\bullet) := \inf_{\q^\bullet\in \cpl_{e}(\xi^\bullet,\eta^\bullet)}\IE\Big[\int_{ \Lambda_1\times \IR^d} |x-y|^p\q^\bullet(dx,dy)\Big]\;,
 \end{align*}
then the following representation holds.

\begin{prop}
\begin{align}\label{eq:2ndrep}
     \mathsf C (\P_0,\P_1)=\inf_{(\xi^\bullet,\eta^\bullet)}\cost(\xi^\bullet,\eta^\bullet),
 \end{align} 
 where the infimum runs over all jointly invariant random measures $(\xi^\bullet,\eta^\bullet)$ such that $\xi^\bullet\sim\P_0, \eta^\bullet\sim\P_1$. Moreover, if $\mathsf C(\P_0,
 P_1)<\infty$, there exists an optimal pair $(\Q,\q^\bullet)$ of a process coupling $\q^\bullet\in \cpl_e(\xi,\eta)$ and the induced distribution coupling $\Q\in \Cpl_s(\P_0,\P_1)$ such that $\C(\P_0,\P_1)=\IE[c(\xi^\bullet,\eta^\bullet)]$ and $\q^\bullet$ attains $c(\xi^\bullet,\eta^\bullet)$ almost surely.
\end{prop}

Note that the right hand side of \eqref{eq:2ndrep} can be interpreted as a Palm expectation. Hence, our optimization problem can be rephrased as an optimal transport problem between Palm versions of $\P_0$ and $\P_1$. 

We want to stress that the existence of an optimal \emph{invariant} $\q^\bullet$ attaining $c(\xi^\bullet,\eta^\bullet)$ a.s.~is the most difficult part of the statement. This follows using techniques developed in \cite{HS13,H16} once an optimal distribution coupling $\Q$ is fixed. To see why \eqref{eq:2ndrep} holds, observe that for any $\q^\bullet\in \cpl_{e}(\xi^\bullet,\eta^\bullet)$ we have by invariance  
$$\IE\Big[\int_{ \Lambda_1\times \IR^d} |x-y|^p\q^\bullet(dx,dy)\Big]=\limsup_{n\to\infty}\frac{1}{n^d}\IE\Big[\int_{ \Lambda_n\times \IR^d} |x-y|^p\q^\bullet(dx,dy)\Big],$$
so that the difference between the two sides in \eqref{eq:2ndrep} is essentially whether the $\limsup$ is inside or outside the expectation (plus the choice of an invariant $\q^\bullet$). The main tool to show equality is the ergodic theorem together with approximation techniques for the  optimal $\q^\bullet$ from \cite{HS13, H16}. 

A main advantage of \eqref{eq:2ndrep} is that the symmetry of $\mathsf W_p$ is now a direct consequence of the invariance of an optimal process coupling $\q^\bullet$ together with the mass transport principle. To show that $\mathsf W_p$ defines a geodesic metric we need to construct an interpolation $(\mathsf P_t)_{0\leq t \leq 1}$ between any two distributions of random measures $\mathsf P_0$ and $\mathsf P_1$ such that for any $0\leq s\leq  t\leq 1$
\begin{align*}
    \mathsf W_p(\mathsf P_s,\mathsf P_t)=(t-s)\mathsf W_p(\mathsf P_0,\mathsf P_1).
\end{align*}
To this end, fix $\P_0$ and $\P_1$ such that $\W_p(\P_0,\P_1)<\infty.$ Let $(\Q,\q^\bullet)$ be an optimal pair. For $t\in[0,1]$ define the map $\mathsf{geo}_t:\IR^d\times \IR^d\to \IR^d$ by $\mathsf{geo}_t(x,y)=x+t(y-x)$. Then, the push-forward $(\mathsf{geo}_t)_\#\q^\bullet\in\mathcal M(\IR^d)$ is the process interpolating the points of the support of $\q^\bullet$ according to the Euclidean geodesics. Let
\begin{align}\label{eq:congeodesic-intro}
\mathsf P_t= {\sf law} \big((\mathsf{geo}_t)_\#\q^\bullet\big)
\end{align}
be its distribution.
Then, $t\mapsto\mathsf P_t$ is the desired constant speed geodesic.
\bigskip

We will be particularly interested in point processes, i.e. random measures taking values in the set of locally finite counting measures. A realization of a point process can be identified with a random configuration of countably many points in $\IR^d$. It is interesting to note that the class of stationary point processes forms a geodesically convex subset, i.e.\ the geodesic interpolation of two stationary point processes is again a stationary point process, see Proposition \ref{prop:closed_sPP}.

\subsection{Infinite particle systems, displacement convexity and gradient flow of the specific entropy}

In the second part of this paper we apply the novel geometry induced by the transport distance $\mathsf W_2$ to study functionals on stationary point processes and infinite particle dynamics. 

Let us recall the seminal observation of McCann \cite{McCann} that the Boltzmann entropy is displacement convex, i.e.\ convex along geodesics in the Wasserstein distance. Here, the Boltzmann entropy of a probability distribution $\mu$ on $\IR^d$ is the relative entropy with respect to Lebesgue, that is $$\gls{Ent}(\mu|\Leb)=\int\rho\log\rho d\Leb$$ provided $\mu=\rho\Leb$ and $+\infty$ otherwise.

In the context of stationary point processes, we consider the specific entropy of a stationary point process. Let $\P\in\mathcal P_s(\Gamma)$ be the distribution of a stationary point process and $\Poi$ the distribution of a Poisson point process. Then, the \emph{specific relative entropy of} $\P$ is defined as
\[\gls{specE}(\P):=\sup_{n\geq 0} \frac1{n^d}\mathsf {Ent}(\P_{\Lambda_n}|\Poi_{\Lambda_n})=\lim_{n\to\infty} \frac1{n^d}\mathsf {Ent}(\P_{\Lambda_n}|\Poi_{\Lambda_n})\;, \]
where $\P_{\Lambda_n}$ denotes the restriction of $\P$ to $\Lambda_n$. $\mathcal E$ is a well studied and natural object for stationary point process, e.g.\ see \cite{RAS,LeSe17,Georgii,Der19,EHL21}. 
One of the main results of this paper is the convexity of the specific entropy along geodesics with respect to the transport distance we construct.

\begin{thm}\label{cor:convexity}
The specific entropy $\mathcal E$ is convex along $\mathsf W_2$-geodesics, i.e.~for any $\W_2$-geodesic $(\P_t)_{t\in[0,1]}$ we have
\[\mathcal E(\P_t) \leq (1-t)\mathcal E(\P_0) + t\mathcal E(\P_1)\;.\]
\end{thm}

This can be seen as the natural analogue of McCann's observation \cite{McCann} in the context of stationary point processes. Hence, we will also say that the specific entropy is displacement convex. Note that for the usual linear interpolation, the specific relative entropy is affine.
\medskip

In a similar spirit, we obtain an analogue for stationary point processes of the celebrated HWI inequality by Otto and Villani \cite{OV00} which relates the Boltzmann entropy $H=\mathsf{Ent}(\cdot|\Leb)$, the Wasserstein distance $W_2$ and the Fisher information $I$. More precisely, the Fisher information of the probability distribution $\mu\in \mathcal P(\IR^d)$ is defined by $I(\mu|\Leb)=\int |\nabla\sqrt{\rho}|^2d\Leb$, provided $\mu=\rho\Leb$ and $+\infty$ else. The HWI inequality then states that for $\mu_0,\mu_1\in\mathcal P(\IR^d)$
\[H(\mu_0)-H(\mu_1)\leq W_2(\mu_0,\mu_1)\sqrt{I(\mu_0)}\;.\]

Thus, it is natural to define the \emph{specific (relative) Fisher information} of a stationary point process $\P$ with respect to $\Poi$ by
\begin{align}
\gls{specI} (\P):=
\limsup_{n\to\infty}\frac 1 {n^d}
 I(\P_{\Lambda_n}|\Poi_{\Lambda_n})\;.
\end{align}
i.e.~as the large volume limit of the normalized relative Fisher information of $\P$ on boxes. The latter is given for a subset $A\subset \IR^d$ by
\[\gls{I}\big(\P_A\big\vert\Poi_A\big):=  \int_{\Gamma_A}|\nabla \sqrt{\rho}|^2d \Poi_A=4 E\big(\sqrt{\rho}\big)\:,\]
provided $\P_A = \rho \Poi_A$ is absolutely continuous w.r.t.\ $\Poi_A$ and $\sqrt\rho\in D(E)$.
Here, $E$ is the natural Dirichlet form on the configuration space introduced by Albeverio-Kondratiev-R\"ockner in \cite{AKR98}, see Section \ref{sec:HWI} for more details. Then we have the following

\begin{thm}\label{thm:HWI2-intro}
Let $\P_0,\P_1\in\spp1$ have finite cost $\W_2(\P_0,\mathsf \P_1)<\infty$, finite entropies $\mathcal E(\P_0),\mathcal E(\P_1)<\infty$ and finite Fisher information $\mathcal I(\P_0)<\infty$. Then, the following HWI-inequality holds
\begin{align}\label{eq:HWI2-intro}
    \mathcal E(\P_0)-\mathcal E(\P_1)\le \W_2(\P_0,\P_1)\sqrt{\mathcal I (\P_0)}.
\end{align}
\end{thm}

We believe that this result is a first step in developing a general framework to establish functional inequalities for stationary point processes via optimal transport, keeping in mind that under a Bakry-Emery condition for a measure $\mathsf{m}=e^{-V}\Leb$, a strengthened version of the HWI inequality of Otto-Villani for the relative entropy and Fisher information w.r.t.~$\mathsf{m}$ gives rise to the logarithmic Sobolev inequality for $\mathsf{m}$.
\medskip

In fact, the convexity of the specific entropy in Theorem \ref{cor:convexity} will be a consequence of our analysis of the gradient flow of $\mathcal E$ w.r.t.~to the geometry induced by $\W_2$. Namely, we will show that this gradient flow coincides with the evolution of a stationary point process induced by adding independent Brownian motions to each point of its realisation.
This can be seen as the natural analogue in the context of infinitely many particles of the seminal results by Jordan, Kinderlehrer and Otto \cite{JKO98} that the heat flow, or the evolution of the law of a single Brownian motion is the Wasserstein gradient flow of the Boltzmann entropy.

To be more precise, let us denote by $(\Semi_t)_t$ the semigroup acting on point processes given by evolving each point by an independent Brownian motion. We then have the following gradient flow characterisation of this semigroup in terms of an \textit{Evolution Variational Inequality} (EVI).

\begin{thm}\label{thm:PPEVI}
Let $\P,\mathsf R\in\mathcal P_s(\Gamma)$ be two distributions of stationary point processes with unit intensity. Then, $\Semi_t\P$ satisfies the EVI
\begin{align*}
\mathsf W^2_2(\Semi_t\P,\mathsf R)-\W^2_2(\P, \mathsf R) \leq 2t\big[ \mathcal E(\mathsf R)- \mathcal E(\Semi_t\P)\big]\;.
\end{align*}
\end{thm}

Among several possible characterisations of gradient flows in metric spaces, Evolution Variational Inequalities are among the strongest. We refer e.g.~to \cite{Daneri_2008} for a general discussion of EVI flows.
The EVI encodes and entails a number of strong consequences for the semigroup. In particular, it yields the geodesic convexity of the specific entropy. Furthermore, we derive from it the following results on the long-time behaviour of the semi-group.

\begin{cor}Let $\P,\R\in \overline{D(\mathcal E)}\subseteq{\mathcal P(\Gamma)}$ the closure of the proper domain of $\mathcal E$ w.r.t.\ $\mathsf W_2$. Then, 
\begin{enumerate}
    \item for all $t\geq 0$
    \[\W_2(\Semi_t\P,\Semi_t \R)\leq \W_2(\P,\R)\; ;\]
    \item  if additionally  $\mathsf W_2(\P,\mathsf{Poi})<\infty$, then the function $(0,\infty)\ni t\mapsto \mathcal{E}(\Semi_t\P)$ is decreasing and 
    \[
    \mathcal{E}(\Semi_t\P)\xrightarrow{t\to\infty}0.
    \] 
    In particular, we have $\Semi_t\P\to\Poi$ weakly, w.r.t.~the vague topology.
\end{enumerate}
\end{cor}

The convergence statement in the second item is known since at least \cite{St68}, however the convergence in entropy is stronger than weak convergence due to Pinsker's inequality. In fact, we are not aware of a stronger convergence statement in the literature. However, the appealing and interesting part is certainly the interpretation as a gradient flow which we expect to extend also to infinite systems of interacting particles. Note that such a characterisation for stationary point processes was out of reach with the existing results \cite{ErHu15, DSSu21} which only cover point processes absolutely continuous to $\Poi$. The current results are complementary to this setting, as the only stationary process absolutely continuous to $\Poi$ is $\Poi$ itself. An example satisfying the assumptions in the second item of the Corollary is the shifted lattice in $d\geq 3$ where each point is perturbed by an iid r.v.\ which is uniform on a small ball.
\medskip

Finally, we want to comment on the proof of Theorem \ref{thm:PPEVI}. Since both, the specific entropy $\mathcal E$ as well as the metric $\mathsf W_2$ are defined using restrictions to boxes $\Lambda_n$ and taking limits, it is natural to first prove an EVI for $\P$ and $\mathsf R$ restricted to $\Lambda_n$ and then to pass to the limit. This strategy works very well under the assumption that the number statistics of $\P$ and $\mathsf R$ coincide, i.e. that for any $n\in\IN$ and any $k\in\IN_0$ the probabilities of finding $k$ points in the box $\Lambda_n$ are equal under $\P$ and $\mathsf R$. This is a necessary condition to be able to couple $\P_{\Lambda_n}$ and $\mathsf R_{\Lambda_n}$ using a Wasserstein distance on $(\IR^d)^k$. However, there is no reason why this should be the case.

To overcome this problem we construct modifications $\tilde\P_{\Lambda_n}$ and $ \tilde{\mathsf{ R}}_{\Lambda_n}$ of $\P$ and $\R$, respectively, as follows: Let $(\Q,\q^\bullet)$ be an optimal pair for $\P$ and $\mathsf R$ and assume that $\q^\bullet$ is a matching. Recall that the construction of geodesics from \eqref{eq:congeodesic-intro} transports points along straight lines connecting matched points $(x,y)\in\supp\q^\bullet$. If this transport happens to leave (or enter) $\Lambda_n$, then we replace that corresponding exterior point with an interior point nearby its exit position.
More precisely, we cover $\Lambda_n\setminus\Lambda_{n-1}$ by disjoint boxes $(K_i)_{i=1}^N$ of side length $1/2$. For each $(x,y)\in\supp(\q^\bullet)$ such that $x\in\Lambda_{n-1}, y\notin \Lambda_{n-1}$ (or switched), we choose the unique box $K_j$ such that the straight line $\mathsf{geo}_t(x,y)$ connecting $x$ and $y$ leaves $\Lambda_{n-1}$ via $K_j$. For each such pair, we add a uniformly distributed random variable $U\sim\mathcal U(K_j)$ to the configuration $\xi|_{\Lambda_{n-1}}$. The distributions of the resulting 'modified' point processes are the auxiliary measures $\tilde\P_{\Lambda_n}, \tilde {\mathsf{R}}_{\Lambda_n}$.

The most important features of this construction are that the number statistics of $\tilde\P_{\Lambda_n}$ and $ \tilde {\mathsf{R}}_{\Lambda_n}$ coincide, that we leave the interior $\Lambda_{n-1}$ unmodified and that we add points in a diffuse manner. With these properties at hand, we can prove an EVI on $\Lambda_n$ for our auxiliary measures. To be able to use that EVI to show the desired inequality on the full space, we need to
\begin{itemize}
    \item connect the transport cost on $\Lambda_n$ between $\tilde\P_{\Lambda_n}$ and $ \tilde {\mathsf{R}}_{\Lambda_n}$ with $\mathsf W_2(\P,\mathsf R)$
\item connect $\mathsf{Ent}(\tilde\P_{\Lambda_n}|\Poi_{\Lambda_n})$ with $\mathcal E(\P)$ and similarly for $\mathsf R$
\item show that the EVI together with the induced semigroup on $\Lambda_n$ converge to the corresponding objects on the full space.
\end{itemize}
The first item follows rather quickly from the construction since we implicitly shorten the transport distance. For the second item we need to express $\frac{d\tilde\P_{\Lambda_n}}{d\Poi_{\Lambda_n}}$ in terms of $\frac{d\P_{\Lambda_n}}{d\Poi_{\Lambda_n}}$ (similarly for $\mathsf R$), which is one of the technical parts where we make use of the explicit uniform distribution of the points we add. The final item follows from exit time estimates of Brownian motion from a box, carefully coupling the semigroups on $\Lambda_n$ and the full space, and coupling $\P$ and $\mathsf R$ with its auxiliary processes $\tilde\P_{\Lambda_n}$ and $ \tilde {\mathsf{R}}_{\Lambda_n}$.

We finally want to mention that the construction has similarities to the screening procedure used for Coulomb and Riesz gases for instance in \cite{EHL21, LeSe17, petrache2017next}.

\subsection{Connection to the literature}

The transport problem \eqref{eq:stationaryOT} considered here stands in line with a number of other recent developments involving transport problems with additional probabilistic constraints, such as martingale optimal transport or causal optimal transport, see e.g.~\cite{BeJu16, BeCoHu17, BaBeLiZa17}. 

Stochastic dynamics of infinite interacting particle systems on the configuration space $\Gamma$, the space of locally finite point configurations, have been investigated using Dirichlet form techniques. For diffusive dynamics with Ruelle type interactions we refer to \cite{AKR98, AKR98b}, for singular interactions we mention the works of Osada \cite{Os12, Os13, Os13b}. In the case of the Poisson process as volume measure on $\Gamma$, the resulting process can be describe as an infinite system of independent Brownian motions as considered here.

In \cite{ErHu15}, the first two authors showed that for the case of no interaction the evolution of the distribution of the process can be identified with the gradient flow of the Boltzmann entropy w.r.t.\ the Poisson process $\mathsf {Ent}( \cdot | \mathsf {Poi})$ in the sense of the EVI where the metric on $\mathcal P_s(\Gamma)$ is given as the $L^2$-Wasserstein metric w.r.t.\ non-normalised $L^2$-Wasserstein metric $d_\Gamma$ on $\Gamma.$ This, as well as an extensive analysis of the relation between the Dirichlet form and distance $d_\Gamma$ was largely generalised to much more abstract base spaces in \cite{DSSu21, DSSu22}.

As already mentioned,  these results are complementary to the present work since the allowed starting distributions for the EVI are very different. Note for instance that $\mathsf {Ent}( \cdot | \mathsf {Poi})$ is infinite for any stationary point process different from $\Poi$ and that the Wasserstein distance w.r.t.\ the non-normalised transport cost $d_\Gamma$ is typically infinite for stationary distributions. In contrast, the evolution of infinitely many free Brownian motions is well defined for any starting configuration $\xi\in\Gamma$ which does not clump too much, see \cite{KLR08, ErHu15} and in the language of point processes \cite{St68, FiKe80}.

In the setting of 1D Coulomb gases, the first two authors together with Lebl\'e \cite{EHL21} have shown that the sine$_\beta$ process is the unique minimiser of the renormalised free energy introduced in \cite{LeSe17} by exploiting strict convexity properties of this free energy along interpolations by optimal transport between finite volume approximations. We expect the  displacement interpolation w.r.t.~$\W_2$ constructed in the present work to have similar convexity properties.
Recently, Suzuki \cite{Su23} has established  related estimates of Bakry-\'Emery type for the Dirichlet form associated with the sine$_\beta$ process, however yet without capturing strict convexity properties.

In a different direction, one can consider birth-death dynamics on the configuration space that leave the distribution of the Poisson process invariant and are generated by locally adding and removing points, see e.g.~\cite{Su83, Su84, KoLy05}. Recently, Dello Schiavo, Herry, and Suzuki \cite{DSHeSu23} proposed a gradient flow interpretation of this  dynamic 
in terms of the Boltzmann entropy w.r.t.~the Poisson measure and a transportation metric that is inspired by similar constructions developed for jump processes, see \cite{Maa11, Mie11, Er14}. They also established a Talagrand and an HWI inequality linking the Boltzmann entropy w.r.t.\ the Poisson measure with a Fisher information and the metric.  Several functional inequalities for the Poisson process have been established exploiting the powerful Malliavin calculus for Poisson processes, e.g.\ \cite{Last, LaPeSc16} for Poincaré inequalities or \cite{Wu00} for logarithmic Sobolev type inequalities.

Stationary transports of point processes have received considerable attention in the literature due to their connections with shift-couplings of a point process with its Palm version \cite{HoPe05}: Any {\it allocation}, i.e.\ an invariant random map transporting the Lebesgue measure to a point process $\xi$ induces a shift coupling of the point process with its Palm version via $\theta_T\xi.$ The search for explicit constructions lead to a variety of interesting results on allocations, e.g.\ \cite{HoHoPe06, ChPePeRo10, MaTi16, HS13, NaSoVo07}, or {\it matchings} which requires the initial measure to be a point process as well instead of the Lebesgue measure, see e.g.~\cite{HPPS09, HoJaWa22}. These developments have inspired Last and Thorisson to investigate stationary transports between  invariant random measures  in a very general framework  in \cite{LaTh09}. A particular challenging part in these constructions is the quest of constructing factor allocations and matchings, i.e.\ invariant transports between two invariant random measures $(\xi,\eta)$ that measurably only depend on $(\xi,\eta)$. Interestingly, the invariance poses severe limitations on the pair $(\xi,\eta)$. For instance, Last and Thorisson construct an example of two jointly invariant random measures $(\xi,\eta)$ on $\IR^d$ where $\xi$ is concentrated on a $d-1$ dimensional set such there is no invariant factor transport map between $\xi$ and $\eta$. Recently, it was shown independently in \cite{HuMu23, KhMe23} that if $\xi$ does not charge $d-1$ rectifiable sets than one can always construct an invariant transport map which is a factor. We stress that in all these results the distributional coupling of the two random measures is fixed.

Allocations and matchings are closely related to the optimal matching problem which received considerable attention in the last years due to the new approach put forward by \cite{CaLuPaSi14}. This has led to a couple of refined results on the level of transport cost, e.g.\ \cite{AmStTr16, Le17, Ja23, Wa21, HuMaTr23}, and transport maps, e.g.\ \cite{AmGlTr19, ClMa23}, and convergence to Gaussian fields \cite{GoHu22}. However, it is still open to show convergence of solutions to the optimal matching problem to a unique stationary transport, see \cite{HS13, H16, GHO} for first results in this direction.

\subsection{Structure of the paper}
In Section 2 we develop the general theory of optimal transport for stationary random measures, discussing different cost functions and their relation, the existence of optimisers and geodesic interpolations.
In Section 3 we discuss the modification of given pairs of stationary point processes by stationary point processes with equal number of points in a given box. Moreover, we give the construction of the infinite particle semigroup and its approximations. In Section 4 we derive the EVI by lifting it from the evolution of finitely many points in a box. Consequences of the EVI are presented in Section 5, e.g.~convexity of the entropy and the HWI inequality.

A list of frequently used symbols can be found at the end of the article.

\section{Theory of optimal transport for stationary random measures}

\subsection{Stationarity} Let $\mathcal M=\gls{M}$ be the space of Radon measures 
on $\IR^d$ equipped with the topology of vague convergence generated by the maps $\xi\mapsto\xi(f)=\int fd\xi$ for continuous, compactly supported $f\in\glsdisp{C_c}{\mathcal C_c(\IR^d)}$. The group $\IR^d$ acts on $\mathcal M(\IR^d)$ by the natural shift operation $\theta$, i.e.\ $(\gls{theta}\xi)(A)=\xi(A+z)$ for any point $z\in\IR^d$, measure $\xi\in\mathcal M(\IR^d)$ and any Borel set $ A\in\gls{B}$.

We will be interested in probability distributions $\P\in\gls{P(M)}$, whose corresponding random variables are \emph{random measures} denoted by $\gls{xi}:\Omega\to\mathcal M (\IR^d)$. Here and in the following we emphasize the randomness with the notation $\xi^\bullet$, however we may drop $^\bullet$ whenever there is no ambiguity. A distribution $\P$ is called \emph{stationary} if it stays invariant under the push-forward $ \glsdisp{push}{(\theta_z)_\#\P}= \P$ for all $z\in\IR^d$ and we shall denote by $\gls{Ps}$ the family of stationary distributions on $\mathcal M (\IR^d)$.
The intensity measure of a stationary random measure can only be the Lebesgue measure $\Leb$ on $\IR^d$ up to a multiplicative constant.
Without loss of generality we shall assume that this constant is equal to one, i.e. $\IE_\P(\xi^\bullet(A))=\Leb(A)$ for all stationary distributions $\P\in\mathcal P_s (\mathcal M(\IR^d))$.
If the random measure takes values in the space of locally finite counting measures, then it is called \emph{point process}. 
For any $A\in\mathcal B (\IR^d)$, let $\Gamma_A$ be the configuration space on $A$, i.e. $\Gamma_A$ consists of all locally finite counting measures on $A$,  equipped with the topology of vague convergence. We write $\Gamma=\Gamma_{\IR^d}$.
As usual, we shall identify the configuration space $\gls{Gamma}$ with the set of all locally finite subsets of $\IR^d$. Let $\spp1$ be the set of all stationary point processes on $\IR^d$ with intensity one. That is, \[
\spp1=\{\P\in \gls{Ps}(\mathcal{M}(\IR^d))\mid \P(\Gamma)=1\}.
\] 

An important example is given by the \emph{Poisson point process} (homogeneous, intensity one) $\xi^\bullet\sim\gls{Poi}$, satisfying that for every $A\in\mathcal B (\IR^d)$ the point statistic $\xi^\bullet(A)$ is a Poisson distributed random variable with parameter $\Leb(A)$ and for disjoint bounded $A_1,\dots,A_n\in\mathcal B(\IR^d)$ the family $(\xi^\bullet(A_i))_{i=1,\dots,n}$ is independent. We refer to \cite{Last} for a detailed overview. Due to the complete independence, the Poisson point process can be seen as the most random stationary point process, whereas the "most deterministic" one can be thought of the stationarised grid, see Example \ref{ex:grid} below. An example in between complete randomness and periodicity is the Ginibre point process (or, more generally the interpolating family of Coulomb gases).

For any two distributions $\P_0, \P_1\in\mathcal P_s(\mathcal M (\IR^d))$, define the set $ \gls{Cpl_s}$ of all stationary couplings $ \Q$, i.e. $ (\theta_z, \theta_z)_{\#}\Q = \Q$ for all $z\in\IR^d$ and $\Q$ is a coupling of $\P_0$ and $\P_1$, short $\Q\in\Cpl(\P_0,\P_1).$\\

For a given  cost function, i.e.\ a measurable function $ c:\mathcal M (\IR^d)\times\mathcal M(\IR^d)\to\IR$, we consider the minimisation problem
 \begin{align}\label{eq:Problem}
     \gls{C}:=\inf_{ \Q\in \Cpl_s(\P_0,\P_1)}\IE_{ \Q}(c)=\inf_{ \Q\in \Cpl_s(\P_0,\P_1)}\int c(\xi,\eta) \Q (d\xi,d\eta).
 \end{align} 

\begin{lem}\label{lem:charStat}
A coupling $ \Q$ is stationary iff for all $n\in\IN$, bounded functions $G\in\glsdisp{C_b}{\mathcal C_b(\IR ^n)}$, $F_i\in \mathcal C_c(\IR^{d}\times \IR^d)$, $i=1,\dots, n$ and all $z\in\IR^d$ it holds 
\begin{align}\label{eq:charStat}
&\int G\left[\Big(\int F_i(x-z,y-z)\xi(dx)\eta(dy)\Big)_{i=1,\dots,n}\right] \Q(d\xi,d\eta)\nonumber\\
&=\int G\left[\Big(\int F_i(x,y)\xi(dx)\eta(dy)\Big)_{i=1,\dots,n}\right] \Q(d\xi,d\eta).
\end{align}
\end{lem}

For a compact space $X$, a similar class of test functions is convergence determining for the weak topology on $\mathcal P(\mathcal P(X))$ by Stone-Weierstrass approximation, see e.g. \cite[Lemma 7.3]{Beigl}, where it is also shown that it is sufficient to consider $n=1$. However, since $\mathcal M(\IR^d)^2$ is not compact, we will prove it differently here.

\begin{proof}
If $ \Q$ is stationary, then $\int\phi d \Q=\int \phi d (  (\theta_z, \theta_z)_{\#}\Q )$ for all $\phi\in\mathcal C_b(\mathcal M^2 )$ and hence for the test function $\phi(\xi,\eta)=G\big[\big(\int F_i d(\theta_{-z}\xi) d(\theta_{-z}\eta)\big)_{i\le n}\big]$ we have
\begin{align*} 
&\int G\left[\Big(\int F_i(x-z,y-z)\xi(dx)\eta(dy)\Big)_{i=1,\dots,n}\right] \Q(d\xi,d\eta)\\
&=\int G\left[\Big(\int F_i(x,y)(\theta_{-z}\xi)(dx)(\theta_{-z}\eta)(dy)\Big)_{i=1,\dots,n}\right]( (\theta_z^2)_{\#}\Q )(d\xi,d\eta)\\
&=\int G\left[\Big(\int F_i(x,y)\xi(dx)\eta(dy)\Big)_{i=1,\dots,n}\right] \Q(d\xi,d\eta).
\end{align*}
On the other hand suppose \eqref{eq:charStat} holds, then in order to show stationarity we basically need to argue that functions of the above form already generate the weak topology on $\mathcal P (\mathcal M^2)$. As the topology on $\mathcal M^2$ is the weak (initial) topology generated by $(\xi,\eta)\mapsto \int F d\xi d\eta$ for $F\in\mathcal C_c(\IR^d\times\IR ^d)$, each basis element $A$ of the topology on $\mathcal M^2$ can be represented as a finite intersection of preimages of the form
\[
A=\bigcap_{i=1}^n\Big\{(\xi,\eta)\in\mathcal M^2:\int F_id\xi d\eta\in U_i\Big\}
\]
for some $F_i\in\mathcal C_c(\IR ^d\times\IR ^d)$ and some open sets $U_i\subseteq \IR$. Since $\IR^d$ is a separable metric space, so is $\mathcal M(\IR^d)$ and in turn $\mathcal P(\mathcal M^2)$ is a separable metric space, hence the basis of the topology generates the Borel sigma algebra $\mathcal B(\mathcal M^2)$. Thus $ \Q$ is uniquely determined on the intersection-stable basis of the topology and we need to show $ ((\theta_z, \theta_z)_{\#}\Q) (A)= \Q(A)$ for all basis elements $A$. To this end let $(G_k)_{k\in\IN}$ be a sequence of $\mathcal C_b(\IR ^n)$ functions pointwise converging to $\prod_{i\le n} \mathbbm 1_{U_i}$. By dominated convergence and our assumption it follows
\begin{align*}
     \Q(A)=&\int \prod_{i=1}^n\mathbbm 1 _{U_i}\Big(\int F_i(x,y)\xi(dx)\eta(dy)\Big)d \Q(\xi,\eta)\\
    =&\lim_{k\to\infty}\int G_k\left[\Big(\int F_i(x,y)\xi(dx)\eta(dy)\Big)_{i=1,\dots,n}\right] \Q(d\xi,d\eta)\\
   =& \lim_{k\to\infty}\int G_k\left[\Big(\int F_i(x-z,y-z)\xi(dx)\eta(dy)\Big)_{i=1,\dots,n}\right] \Q(d\xi,d\eta)\nonumber\\
     =&\lim_{k\to\infty}\int G_k\left[\Big(\int F_i(x,y)\xi(dx)\eta(dy)\Big)_{i=1,\dots,n}\right]( (\theta^2_{-z})_{\#}\Q)(d\xi,d\eta)=( (\theta^2_{-z})_{\#}\Q)(A)
\end{align*}
for all $z\in\IR ^d$.
\end{proof}

Throughout the paper, we will repeatedly make use of the above family of test functions.

\begin{lem}\label{lem:compact}
$ \Cpl_s(\P_0,\P_1)$ is non empty, compact and convex w.r.t.\ linear interpolation. 
\end{lem}

\begin{proof}
The proof requires only minor adjustment of classical arguments, like tightness, Prokhorov's Theorem \cite[Theorem 14.3]{Kallenberg} and Lemma \ref{lem:charStat}.
\end{proof}

For completeness, let us also mention the following dual formulation of the optimisation problem \eqref{eq:Problem}. Since we will not make explicit use of neither Lemma \ref{lem:compact} nor the duality, we omit the details of their proofs.

\begin{prop}\label{prop:duality}
If $\mathsf C(\mathsf P_0,\mathsf P_1)<\infty$, then $\mathsf C(\mathsf P_0,\mathsf P_1)= \mathsf D (\mathsf P_0,\mathsf P_1)$ where 
\begin{align*}
\mathsf D (\mathsf P_0,&\mathsf P_1):=
\sup\Bigg\{\int \phi d\mathsf P_0+\int\psi d\mathsf P_1 : \\& \phi\in L^1(\mathsf P_0), \psi \in L^1(\mathsf P_1), \exists \tilde G\in\mathcal F\text{ s.t. } \phi(\xi)+\psi(\eta)+ \tilde G(\xi,\eta)\le c(\xi,\eta)\ \forall\ \xi,\eta \Bigg\},
\end{align*}
where $\mathcal F\subseteq \mathcal C_b(\mathcal M^2)$ 
the family of functions given as
\begin{align*}
    G\left[\Big(\int F_i(x-z,y-z)\xi(dx)\eta(dy)\Big)_{i\le n}\right]-G\left[\Big(\int F_i(x,y)\xi(dx)\eta(dy)\Big)_{i \le n}\right]
\end{align*}
for some $n\in\IN, z\in\IR^d, G\in\mathcal C_b(\IR ^n)$ and $ F_i\in \mathcal C_c(\IR^{d}\times \IR^d), i=1,\dots, n $.
\end{prop}
\begin{proof}
    This can be shown via a minmax argument exactly as for the duality result in \cite{BeHLPe13}.
\end{proof}

 \subsection{Transport costs $c$ on $\mathcal M$} 
Let us introduce a family of cost functions  $c$ which will be important for this article. We say $\q\in\mathcal M (\IR^d\times\IR^d)$ is a coupling of two measures $\xi, \eta$, if $\q(A\times\IR^d)=\xi(A)$ and  $\q(\IR^d\times A)=\eta(A)$ for all $A\in\mathcal B (\IR^d)$. We denote the set of all couplings of $\xi,\eta\in\mathcal M (\IR^d)$ by $ \gls{cpl}$.
Similar to the cost functions on $\mathcal M (\IR^d)$ studied in \cite{HS13,H16}, we define the transport cost per volume
 \begin{align}\label{eq:costMeasures}
\gls{c}:= \inf_{\q\in \cpl(\xi,\eta)} \limsup_{n\to \infty} \frac{1}{n^d}\int_{\Lambda_n \times\IR^d} \vartheta(x-y)\q(dx,dy),
 \end{align}
where $\gls{Lambda}=\Lambda_n=[-n/2,n/2]^d$ is the box of volume $n^d$ and $\vartheta:\IR ^ d\to[0,\infty]$ is some continuous radial unbounded function, that is radially increasing and satisfies $\vartheta(0)=0$. In Section \ref{sec:geodesics} and below we shall consider the special case $\vartheta(x)=|x|^p$ for some $p\ge 1$.

While the cost function $c$ is very natural from an optimal transport point of view, it is difficult to work with due to the normalisation factor.
In the following we will show that there is an equivalent formulation of the optimisation problem \eqref{eq:Problem} using stationarity. To this end, we will need to introduce some notation and terminology:

Let $(\Omega,\mathcal{F},\IP)$ be a probability space  equipped with a measurable flow $\gls{thetaO}$.
That is, the mapping $(x,\omega)\mapsto \theta^\Omega_x\omega$ is measurable, $\theta^\Omega_y\circ \theta^\Omega_x=\theta^\Omega_{x+y}$ for all $x,y\in \IR^d$ and $\theta^\Omega_0$ is the identity. We assume that  $\IP$  is stationary  w.r.t.\ the flow $\theta^\Omega$, i.e. $\IP(A)=\IP(\theta^\Omega_x(A))$ for all $x\in \IR^d$.
On the probability space $(\Omega,\mathcal{F},\IP)$ we consider a pair of equivariant random measures, i.e.\ random variables $\xi^\bullet,\eta^\bullet:\Omega\to\mathcal M$ such that $\xi^{\theta^\Omega_x\omega}(A)=\xi^{\omega}(A-x)$ for all $x\in\IR^d, \omega\in\Omega.$

We are interested in random couplings $\q$ with $\q^\omega\in \cpl(\xi^\omega,\eta^\omega)$ for $\IP$-almost all $\omega$. We denote the set of all these couplings by $\gls{cple}$. We call $\q$ equivariant if $\q^{\theta^\Omega_x\omega}(A\times B)=\q^{ \omega}((A-x)\times (B-x))$ for all $x\in\IR ^d, \omega\in\Omega$. Denote the set of such (random) couplings by $ \cpl_{e}(\xi^\bullet,\eta^\bullet)$. Moreover, we define the expected transport cost per volume

\begin{align}\label{eq:costRandomMeasures}
\gls{frakc}:= \inf_{\q^\bullet\in \cpl_{e}(\xi^\bullet,\eta^\bullet)}\IE\Big[\int_{ \Lambda_1\times \IR^d} \vartheta(x-y)\q^\bullet(dx,dy)\Big].
 \end{align}
Note that under our assumptions the joint distribution $(\xi^\bullet,\eta^\bullet)_\#\IP= \Q$ belongs to $ \Cpl_s(\P_0, \P_1)$, where $\P_0$ and $\P_1$ are the distributions of $\xi^\bullet$ and $\eta^\bullet$
respectively. We will write
$$(\xi^\bullet,\eta^\bullet)\sim \Q$$
for any pair of jointly equivariant random measure on some probability space as above whose joint distribution is given by $\Q$. Moreover, by equivariance, the domain of integration in \eqref{eq:costRandomMeasures} could be changed to any bounded Borel set $A$ of positive Lebesgue measure upon normalising the integral with its Lebesgue measure $\Leb(A)$.
 
The following proposition shows that \eqref{eq:costRandomMeasures} induces an equivalent optimisation problem to \eqref{eq:Problem} with cost function given by \eqref{eq:costMeasures}.
\begin{prop}\label{prop:Problem2}
Let $ \Q$ be a stationary coupling of $\P_0$ and $\P_1$ such that $\IE_{ \Q}[c]<\infty$, then
    \begin{align}\label{eq:costequal}
\IE_{ \Q}[c]=\inf_{(\xi^\bullet,\eta^\bullet)\sim \Q}\cost(\xi^\bullet,\eta^\bullet).
\end{align}
Minimizing over $ \Q\in \Cpl_s$  implies that (cf.\ \eqref{eq:Problem})
\begin{align}\label{eq:Problem2}
     \C (\P_0,\P_1)=\inf_{(\xi^\bullet,\eta^\bullet)}\cost(\xi^\bullet,\eta^\bullet),
 \end{align} 
where the infimum runs over all jointly equivariant random measures $(\xi^\bullet,\eta^\bullet)$ with distribution $\P_0,\P_1$ (and all probability spaces supporting the random measures $(\xi^\bullet,\eta^\bullet)$). 
\end{prop}

For fixed joint distribution, the right hand side of \eqref{eq:Problem2} has been studied  in \cite{H16} and for $\xi^\bullet\sim\Poi,\eta^\bullet\equiv \Leb$ in much more detail in \cite{HS13}. Thus, the essential new feature of the right hand side of \eqref{eq:Problem2} is the additional minimization over all joint distributions, which are stationary couplings. On the other hand, the left hand side gives rise to the classical theory of optimal transport on the abstract space $\mathcal M (\IR^d)$. In the sequel we will make use of both viewpoints.

Let us also comment on the idea of proof, revealing the heuristic behind equation \eqref{eq:Problem2}. For ergodic point processes, we expect the transport cost per volume in \eqref{eq:costMeasures} to be close to the average $\IE_\Q[\int_{\Lambda_1\times\IR^d}\vartheta(x-y) d \q]$. 
Minimizing this average over $\q$ (as in \eqref{eq:costMeasures}) leads to \eqref{eq:costRandomMeasures}. Thus, we expect T is precisely the statement Proposition \ref{prop:Problem2}. The proof will follow this idea and is split into two lemmata providing the upper and lower bound.

\begin{lem}\label{lem:ineq_equiv_costfct}
Let $ \Q$ be a stationary coupling of $\P_0$ and $\P_1$ such that $
\inf_{(\xi^\bullet,\eta^\bullet)\sim \Q}\cost(\xi^\bullet,\eta^\bullet)<\infty$, then \[
\IE_{ \Q}[c]\le\inf_{(\xi^\bullet,\eta^\bullet)\sim \Q}\cost(\xi^\bullet,\eta^\bullet).
\]
\end{lem}
\begin{proof}
Fix a probability space $(\Omega,\mathcal F,\IP)$ supporting jointly stationary random measures $(\xi^\bullet,\eta^\bullet)$ with distribution $\sfQ$. Let us first suppose that $\IP$ is ergodic.
Then for all $\q^\bullet\in \cpl_{e}(\xi^\bullet,\eta^\bullet)$, $(\xi^\bullet,\eta^\bullet)\sim \Q$, by the ergodic theorem \cite[Theorem 9.9]{Kallenberg}
\begin{align*}
    \lim_{n\to\infty} \frac{1}{n^d} \int_{\Lambda_n\times\IR ^d} \vartheta(x-y) d\q^\bullet(x,y)=\IE_{ \IP} \int_{\Lambda_1\times\IR^d}\vartheta(x-y) d\q^\bullet(x,y) \quad  \IP-a.s..
\end{align*}
 Taking expectations, we obtain
\begin{align}\label{eq:ergodic_expectations}
   \IE_{ \IP} \limsup_{n\to\infty} \frac{1}{n^d} \int_{\Lambda_n\times\IR ^d} \vartheta(x-y) d\q^\bullet(x,y)=\IE_{ \IP} \int_{\Lambda_1\times\IR^d}\vartheta(x-y) d\q^\bullet(x,y).
\end{align}

In the case that $ \IP$ is not ergodic, but only stationary, we use an ergodic decomposition \cite[Theorem 9.12]{Kallenberg}: There exists a probability measure $m$
so that $ \IP=\int  \IP_\alpha dm(\alpha)$, where $ \IP_\alpha$ are ergodic. Integrating \eqref{eq:ergodic_expectations} with respect to $m$ shows that \eqref{eq:ergodic_expectations} still holds for non-ergodic $ \IP$.

For any $\q^\bullet\in \cpl_{e}(\xi^\bullet,\eta^\bullet)$ with $(\xi^\bullet,\eta^\bullet)\sim \Q$ we have moreover 
\begin{align*}
\IE_{ \Q}[c]= \IE_{ \Q} \inf_{\q\in  \cpl(\xi,\eta)}\limsup_{n\to\infty} \frac{1}{n^d} \int_{\Lambda_n\times\IR ^d} \vartheta(x-y) d\q(x,y)\le \IE_{ \IP} \limsup_{n\to\infty} \frac{1}{n^d} \int_{\Lambda_n\times\IR ^d} \vartheta(x-y) d\q^\bullet(x,y).
\end{align*}
Combining this with \eqref{eq:ergodic_expectations} and 
first minimizing over $q^\bullet\in \cpl_{e}(\xi^\bullet,\eta^\bullet)$ and then over supporting probability spaces implies the claim.

\end{proof}

  \begin{lem}\label{lem:Cn}
Let $ \Q$ be a stationary coupling of $\P_0$ and $\P_1$ such that $\IE_{ \Q}[c]<\infty$. Then, for any $n\in\IN$ we have
\[ \C_n:=\inf_{\substack{\q^\bullet\in  \cpl(\xi^\bullet,\eta^\bullet)\\ (\xi^\bullet,\eta^\bullet)\sim  \Q} }\frac{1}{n^d}\IE\Big[\int_{\Lambda_n\times \IR^d}\vartheta(x-y) d\q^\bullet \Big]= \frac{1}{n^d}\IE_\Q\Big[\inf_{\q\in  \cpl(\xi,\eta)}\int_{\Lambda_n\times \IR^d}\vartheta(x-y) d\q\Big]\leq \IE_{ \Q}[c]<\infty. \]
  \end{lem}
\begin{proof}
Fix $n\in\IN$ and put
\[X^n_z(\xi,\eta):=\inf_{\q\in  \cpl(\xi,\eta)}\frac{1}{n^d}\int_{(z+\Lambda_n)times\IR^d}\vartheta(x-y)\q(dx,dy)=X^n_0(\theta_z\xi,\theta_z\eta).\]
It then follows that for any $\q\in  \cpl(\xi,\eta)$ and $m\in\IN$
\begin{align}\label{eq:1}
\frac{1}{m^d} \sum_{z\in\Lambda_{nm}\cap (n\IZ)^d} X_z^n(\xi,\eta) \leq \frac{1}{(mn)^d}\int_{\Lambda_{mn}\times \IR^d}\vartheta(x-y)d\q.
\end{align}
Since $ \Q$ is in particular $(n\IZ)^d$-stationary, denoting by $\mathcal I^n$ the $\sigma$-algebra of $(n\IZ)^d$-invariant events, the left hand side of \eqref{eq:1} converges by \cite[Theorem 9.6]{Kallenberg} to $\IE_{ \Q}[X^n_0|\mathcal I^n]$ as $m\to\infty$.
The right hand side is bounded from above by 
\[\limsup_{m\to\infty} \frac{1}{(mn)^d}\int_{\Lambda_{mn}\times \IR^d}\vartheta(x-y)d\q\leq \limsup_{m\to\infty}  \frac{1}{m^d}\int_{\Lambda_{m}\times \IR^d}\vartheta(x-y)d\q\]
for any $\q\in  \cpl(\xi,\eta)$ so that we obtain
\[\IE_{ \Q}[X^n_0|\mathcal I^n]\leq \inf_{\q\in  \cpl(\xi,\eta)}\limsup_{m\to\infty} \frac{1}{m^d}\int_{\Lambda_{m}\times \IR^d}\vartheta(x-y)d\q = c(\xi,\eta).\]
Taking the expectation w.r.t.\ $ \Q$ yields the claim.
\end{proof}
\begin{proof}[Proof of Proposition \ref{prop:Problem2}]
Lemma \ref{lem:Cn} implies that $\C_\infty:=\limsup_{n\to\infty} \C_n\leq \IE_{ \Q}[c]$. By \cite[Remark 6.6]{H16}, it follows that
$$\C_\infty=\inf_{\substack{\q^\bullet\in  \cpl_e(\xi^\bullet,\eta^\bullet)\\ (\xi^\bullet,\eta^\bullet)\sim  \Q} }\IE_{ \Q}\Big[\int_{\Lambda_1\times \IR^d} \vartheta(x-y) d\q^\bullet(x,y)\Big] =\inf_{(\xi^\bullet,\eta^\bullet)\sim \Q} \cost(\xi^\bullet,\eta^\bullet).$$
In particular this value is finite. The other inequality follows from Lemma \ref{lem:ineq_equiv_costfct}.

 \end{proof}
Finally, \cite[Corollary 6.5]{H16} yields another equivalent formulation of the cost.
\begin{cor}\label{cor:equiv_cost}
Let $\P_0,\P_1\in \mathcal P_s(\mathcal M(\IR^d))$. Then 
  \begin{align*}
  \C(\P_0,\P_1)&=\inf_{\substack{\Q\in \mathsf{Cpl}_s(\P_0,\P_1)\\ (\xi^\bullet,\eta^\bullet)\sim  \Q} }  
  \inf_{\q^{\bullet}\in \cpl(\xi^{\bullet},\eta^{\bullet})}
  \liminf_{n\to \infty}n^{-d}\IE\left[
  \int_{\Lambda_n\times \IR^d}\vartheta(x-y)\q(dx,dy)\right]\\
  &=\inf_{\substack{\Q\in \mathsf{Cpl}_s(\P_0,\P_1)\\ (\xi^\bullet,\eta^\bullet)\sim  \Q} }  
  \liminf_{n\to \infty}
  \inf_{\q^{\bullet}\in \cpl(\xi^{\bullet},\eta^{\bullet})}
  n^{-d}\IE\left[
  \int_{\Lambda_n\times \IR^d}\vartheta(x-y)\q(dx,dy)\right].
\end{align*}  
\end{cor}
\begin{proof}
This follows from applying \cite[Corollary 6.5]{H16} to the following formulation of the cost \[
\C(\P_0,\P_1)=\inf_{\substack{\Q\in \mathsf{Cpl}_s(\P_0,\P_1)\\ (\xi^\bullet,\eta^\bullet)\sim  \Q} } \cost(\xi^{\bullet},\eta^{\bullet}).
\]
\end{proof}

  \subsection{Existence of optimal pairs}

As explained above, we have two layers of optimization problems. One on the level of distributions of random measures, that is \eqref{eq:Problem}, and one on the level of random measures itself, that is \eqref{eq:costRandomMeasures}. In this section we will verify the existence of optimal couplings on both layers.
\begin{defi}\label{def:optimalpair}
Let $\xi^\bullet\sim\P_0$, $\eta^\bullet\sim\P_1$ be jointly equivariant random Radon measures.
\begin{enumerate}[a)]
    \item We call $\Q\in\Cpl(\P_0,\P_1)$ an (optimal) distribution-coupling of $\P_0$ and $\P_1$, if it minimizes \eqref{eq:Problem}, i.e. $$\C(\P_0,\P_1)=\IE_\Q(c).$$
    \item We call $\q^\bullet\in\cpl_e(\xi^\bullet,\eta^\bullet)$ an (optimal) spatial/process-coupling of $\xi^\bullet$ and $\eta^\bullet$, if it minimizes \eqref{eq:costRandomMeasures}, i.e. $$\cost(\xi^\bullet,\eta^\bullet)=\IE\Big[\int_{ \Lambda_1\times \IR^d} \vartheta(x-y)\q^\bullet(dx,dy)\Big].$$
    \item We call \gls{Qq} an (optimal) pair, if $\Q$ is an (optimal) distribution-coupling and $\q^\bullet$ is an (optimal) process coupling of $(\xi^\bullet,\eta^\bullet)\sim\Q$.
    \item We call $\q^\bullet\in\cpl_e(\xi^\bullet,\eta^\bullet)$ a matching if a.s. $\q^\bullet\in \Gamma_{\IR^{2d}}$.
\end{enumerate}
\end{defi}

\begin{rem}\label{rem:q_implies_Q}
   Every process coupling $\q$ of $\P_0$ and $\P_1$ induces a distribution coupling $\Q$ of  $\P_0$ and $\P_1$, defined by \[
\Q=\mathsf{law}((\pr_1)_{\#}\q,(\pr_2)_{\#}\q),
\]
 where $\pr_i:\IR^d\times \IR^d\to \IR^d$ is the projection on the $i$-th coordinate. 
 Furthermore,  applying the disintegration theorem, see \cite[Theorem 5.3.1]{AGS08},  to the probability space $(\mathcal M(\IR^d\times \IR^d),\mathsf{law}(\q))$ and the map $(\pr_1,\pr_2):\mathcal M(\IR^d\times \IR^d)\to \mathcal M(\IR^d)^2$ yields a family of probability measures $(\mu_{(\xi,\eta)})_{(\xi,\eta)\in \mathcal M(\IR^d)^2}$ on $\mathcal M(\IR^d)^2$ such that\begin{enumerate}
     \item  For all measurable $A\subset \mathcal M(\IR^d\times \IR^d)$ the map $(\xi,\eta)\mapsto \mu_{(\xi,\eta)}(A)$ is measurable
     \item For $\Q$-almost all $(\xi,\eta)$ the measure $\mu_{(\xi,\eta)}$ is concentrated on the set $\{q\in \mathcal M(\IR^d\times \IR^d)\mid (\pr_1,\pr_2)(q)=(\xi,\eta)\}$
     \item For every measurable $f:\mathcal M(\IR^d\times \IR^d)\to [0,\infty]$ \[
     \IE[f(\q)]=\int \Q(d(\xi,\eta)) \int \mu_{(\xi,\eta)}(dq)f(q).
     \]
     \end{enumerate}
     Define the random coupling 
     $\tilde{\q}^{\bullet}:(\mathcal M(\IR^d)^2,\Q)\to \mathcal M(\IR^d\times \IR^d)$ by \[
     \tilde{\q}^{(\xi,\eta)}(A)=\int q(A)\mu_{(\xi,\eta)}(dq).
     \]\begin{align*}
    \IE_{\Q}[f(\tilde{\q})]&=\int \Q(d(\xi,\eta))\int \mu_{(\xi,\eta)}(dq)f(q)=\IE[f(\q)].
     \end{align*}
    This implies that $\tilde{\q}$ is a process coupling of $\P_0$ and $\P_1$ on the canonical setup induced by $\Q$. Furthermore, $\tilde{\q}$ yields the same cost as $\q$, i.e.
    \[
    \IE\left[\int_{\Lambda_1\times \IR^d}\vartheta(x-y)\q(dx,dy)\right]
    =\IE_{\Q}\left[\int_{\Lambda_1\times \IR^d}\vartheta(x-y)\tilde{\q}(dx,dy)\right].
    \] 
\end{rem}

\begin{prop}\label{prop:lsc_costfct+existence_min}
The cost function $\C:\mathcal P_s(\mathcal M(\IR^d))^2\to [0,\infty)$ is lower semicontinuous.
    For $\P,\mathsf R\in \mathcal P_s(\mathcal M(\IR^d))$ with finite cost $\C(\P,\R)<\infty$, there exist jointly equivariant random measures $\xi^{\bullet}\sim \P$ and $\eta^{\bullet}\sim \mathsf R$ and a process-coupling $\q^\bullet\in \mathsf{cpl}_e(\xi^{\bullet},\eta^{\bullet})$ such that \[
    \C(\P,\mathsf R)=\IE\Big[\int_{ \Lambda_1\times \IR^d} \vartheta(x-y)\q^\bullet(dx,dy)\Big]\;.
    \]
\end{prop}

\begin{proof}
Let $\P_n\to \P$ and $\mathsf R_n\to \mathsf R$ weakly as $n\to \infty$. 
    We want to use \cite[Proposition 8.5]{H16}. To this end, we need to rewrite our setup a bit.
    For each $n\in \IN$, by Proposition \ref{prop:Problem2} there exist 
    jointly equivariant random measures $\xi_n^{\bullet}\sim \P_n$ and $\eta_n^{\bullet}\sim \mathsf R_n$ and  process-couplings $\q_n^\bullet\in \mathsf{cpl}_e(\xi_n^{\bullet},\eta_n^{\bullet})$ such that \begin{align}\label{eq:existesnce_cost}
     \IE\Big[\int_{ \Lambda_1\times \IR^d} \vartheta(x-y)\q^\bullet_n(dx,dy)\Big]\leq 
     \C(\P_n,\mathsf R_n)+\frac{1}{n}.
    \end{align}  By \cite[Proposition 3.18]{H16}
    we can assume that $\q_n^\bullet\in \mathsf{cpl}_e(\xi_n^{\bullet},\eta_n^{\bullet})$ is an optimal process-coupling.
    We can assume  the probability spaces, on which the measures $\q_n^\bullet$ are defined, to be compact. By considering the product space, which is compact by Tychonoff's theorem, with the product action, we can assume that they are defined on the same compact probability space $\Omega$. By considering subsequences, we can assume that  \[
    \lim_{n\to \infty}\C(\P_{n},\mathsf R_{n})=\liminf_{n\to \infty}\C(\P_n,\mathsf R_n).
    \]On $\IR^d\times \Omega$ consider the sequence of Campbell measures $(\xi^{\omega}_n(dx)\mathbb{P}(d\omega))_n$ and $(\eta^{\omega}_n(dx)\mathbb{P}(d\omega))_n$ which we denote by $\xi_n^{\bullet}\mathbb{P}$ and $\eta_n^{\bullet}\mathbb{P}$. We claim that both sequences are vaguely relatively compact in $\mathcal M(\IR^d\times \Omega)$. We prove this only for the first sequence, since the proof is exactly the same for the other sequence. By \cite[Theorem A2.3]{Kallenberg} we have to show that for $f\in \mathcal C_c(\IR^d\times \Omega)$ we have
    \[
    \sup_{n\in \IN}\int f d\xi^\bullet_n\mathbb{P}<\infty.
    \] 
Let $f$ be such a function and let $A\subset \IR^d$ be compact such that $\supp{f}\subset A\times \Omega$ (recall that $\Omega$ was chosen to be compact but this is not necessary).  Then 
    \begin{align*}
    \int f(x,\omega)\xi_n^{\omega}(dx)\mathbb{P}(d\omega)
    &\leq \lVert f\rVert_{\infty} \int \eins_{A}(x)\xi_n^{\omega}(dx)\mathbb{P}(d\omega).
    \end{align*}
    Since the intensities of the random measures $\xi_n^{\bullet}$ are uniformly bounded (they are all equal to one), we obtain that  the sequences $(\xi_n^{\bullet}\mathbb{P})_n$ and $(\eta_n^{\bullet}\mathbb{P})_n$ are relatively compact. By considering  subsequences we can assume that 
    \[
\xi_n^{\bullet}\mathbb{P} \to \xi^{\bullet}\mathbb{P}
    \text{ and }\eta_n^{\bullet}\mathbb{P} \to \eta^{\bullet}\mathbb{P}
    \]
    vaguely in $\mathcal M(\IR^d\times \Omega)$ as $n\to \infty$, where $\xi^\bullet$ and $\eta^\bullet$ are some invariant random measures. 
    From the weak convergence $\P_n\to \P$ as $n\to \infty$ it follows that $\xi^{\bullet}\sim \P$ and similiarly it follows that $\eta^{\bullet}\sim \mathsf R$.
    Now \cite[Proposition 8.5]{H16} yields the existence of a coupling $\q\in \mathsf{cpl}_e(\xi^{\bullet},\eta^{\bullet})$ such that \[
    \IE\Big[\int_{ \Lambda_1\times \IR^d} \vartheta(x-y)\q^\bullet(dx,dy)\Big]
    \leq 
    \liminf_{n\to \infty }\IE\Big[\int_{ \Lambda_1\times \IR^d} \vartheta(x-y)\q_n^\bullet(dx,dy)\Big]\leq \liminf_{n\to \infty} \C(\P_n,\mathsf R_n).
    \]
    This proves the lower semicontinuity and the existence of a minimizing process coupling.
\end{proof}

In the case of point processes, the optimal process coupling will also be a point process, i.e.~a matching.

\begin{prop}\label{prop:matching}
Let $\P,\mathsf R\in \mathcal P_s(\Gamma)$ be two distributions of stationary point processes with finite cost $\C(\P,\R)<\infty$. Then, there exist jointly equivariant point processes $\xi^{\bullet}\sim \P$ and $\eta^{\bullet}\sim \mathsf R$ and we can choose the optimal process coupling  $\q^\bullet\in \mathsf{cpl}_e(\xi^{\bullet},\eta^{\bullet})$ of Proposition  \ref{prop:lsc_costfct+existence_min} as a matching.
       \[    \C(\P,\mathsf R)=\IE\Big[\int_{ \Lambda_1\times \IR^d} \vartheta(x-y)\q^\bullet(dx,dy)\Big].    \]
   \end{prop}
\begin{proof}

Similar to the the proof of Proposition \ref{prop:lsc_costfct+existence_min} we assume that all random variables are defined on a common compact probability space $\Omega$.

Let $\hat \q_n$ be an optimal partial matching between $\xi_{|\Lambda_n}$ and $\eta$ which exists by Birkhoff's theorem and measurable selection, e.g. \cite[Corollary 5.22]{Villani}. Observe that the normalised cost of this matching satisfies 
 \[\cost_n=n^{-d} \IE[\int_{\Lambda_n\times \IR^d} \vartheta(x-y)\hat \q_n(dx,dy)] \leq \inf_{\q\in \mathsf{cpl}_e(\xi,\eta)} \IE[\int_{\Lambda_1\times \IR^d}\vartheta(x-y)\q(dx,dy)]=\cost(\xi,\eta).\]
 Let $\xi_n=\xi_{|\Lambda_n}.$ For $(x,y)\in\Lambda_n\times\IR^d$ let
  \begin{align*}
      t_{(x,y)}&=\max\{t\in[0,1]:(1-t)x+ty\in\Lambda_n\}\\
      z&=z(x,y)=(1-t_{(x,y)})x+t_{(x,y)}y.
  \end{align*}
 Put $\eta_n=z_\#\hat \q_n$. Note, that $(\xi_n,\eta_n)$ are two point processes with the same number statistic on $\Lambda_n$, i.e. $\xi_n(\Lambda_n)$ and $\eta_n(\Lambda_n)$ have the same distribution,  and normalised transport cost bounded by $c_n$. Moreover, we have
 \begin{equation}\label{eq:etasub}
   (\eta_n)_{|(-\frac{n}{2},\frac{n}{2})^d} \subset \eta_{|\Lambda_n} ,
 \end{equation}
  where for two  configurations $\xi_1, \xi_2\in \Gamma$ we write $\xi_1\subset \xi_2$ iff 
  \[
\forall x\in \IR^d: \xi_1(\{x\})=1 \implies \xi_2(\{x\})=1.
\]
Let $ \q_n$ be an optimal matching of $\xi_n$ and $\eta_n$, $(\q_{n,z})_{z\in \IZ^d}$ be iid copies of $ \q_n$ and
  $U_n$ be independent and uniformly distributed on $\Lambda_n$. Define its stationarization by \[
\bar{\q}_{n}=\theta_{U_n}\left(\sum_{z\in \IZ^d}\theta_{nz}\q_{n,nz} \right).
    \]
Note that by a suitable extension of the underlying probability  space and the measurable flow, we may assume that $\bar \q_n$ is equivariant.
An easy calculation shows (see for example the proof of Proposition \ref{prop:cost_modified_processes}) that $\bar \q_n$  has transport cost uniformly bounded by $\cost(\xi,\eta)$, i.e. \begin{align}\label{eq:unifbound}
\IE\Big[\int_{ \Lambda_1\times \IR^d} \vartheta(x-y)\bar \q_n(dx,dy)\Big]\leq \cost_n\leq \cost(\xi,\eta).
\end{align}

Note that, by construction, the  marginals of $\bar \q_n$ have intensity equal to one.
It then follows, just as in Proposition \ref{prop:lsc_costfct+existence_min}, that the sequence $(\bar \q_n)_{n\in \IN}$ is tight w.r.t. the vague topology on $\mathcal M(\IR^d\times \IR^d\times \Omega)$. Denote by $\bar \q$ a limit point of a suitable  subsequence, which by abuse of notation we still denote by $(\bar \q_n)_{n\in \IN}$. Then 
  \[
    \IE\Big[\int_{ \Lambda_1\times \IR^d} \vartheta(x-y)\bar\q(dx,dy)\Big]
    \leq 
    \liminf_{n\to \infty }\IE\Big[\int_{ \Lambda_1\times \IR^d} \vartheta(x-y)\bar \q_n(dx,dy)\Big]\leq \cost(\xi,\eta).
    \]
Furthermore, the space of point processes is closed in the space of random measures and since for all $n\in \IN$ the random measure $\bar \q^n$ is a.s.\ a matching, also $\bar \q$ is a matching a.s.. It remains to show that the marginals of the equivariant random matching $\bar \q$, which we denote by $\bar \xi$ and $\bar \eta$, have the same distribution as $\xi$ and $\eta$, i.e. they  are distributed according to $\mathsf P$ and $\mathsf R$ respectively. We prove this for $\bar \eta$, since the other case is easier and can be proved analogously.

Let $\tilde U_n$ be an independent and uniformly distributed random variable on $\Lambda_{n-\sqrt{n}}$. Similar to the definition of $\bar \q^n$ set \begin{align*}
\tilde \q^n=\theta_{\tilde U_n}\left(\sum_{z\in \IZ^d}\theta_{nz}\q_{n,nz} \right).
\end{align*}
 It is immediate to check that also $\tilde \q^n$ converges to $\bar \q$ w.r.t. the vague topology on $\mathcal M(\IR^d\times \IR^d\times \Omega)$ as $n\to \infty$. 
 Denote by $\tilde \eta_n$ the second marginal of $\tilde \q^n$. 
 Note that combining the convergence $\tilde\q^n\xrightarrow{n\to \infty}\bar \q$ with the uniform bound on the costs \eqref{eq:unifbound}, it can be shown exactly as in the proof of \cite[Proposition 3.18]{H16} that  $\tilde \eta_n\xrightarrow{n\to \infty}\bar \eta$ w.r.t. the vague topology on $\mathcal M(\IR^d\times \Omega)$. 
In particular, this implies for a function $f\in \mathcal C_c^{\infty}(\IR^d)$ that \[
\IE\left[\int fd\tilde \eta_n\right]\xrightarrow{n\to\infty}\IE\left[\int fd\bar \eta\right],
\]
 which implies that the intensity of $\bar \eta$ is greater or equal than one.
 Furthermore, for fixed $r>0$ and $n$ large enough,  \eqref{eq:etasub} implies that  \begin{align}\label{eq:subset}
     (\tilde \eta_n)_{| \Lambda_r}=(\theta_{\tilde U_n}\eta_n)_{| \Lambda_r}\subset (\theta_{\tilde U_n}\eta)_{| \Lambda_r} .
 \end{align}
Note that for all $n\in \IN$ the equivariant random measures $\theta_{\tilde U_n}\eta$ have the same distribution as $\eta$ (by independence of $\tilde U_n$), i.e. their distribution is equal to $\mathsf R$. 
Set $\Q_n=\mathsf{law}(\tilde\eta_n,\theta_{\tilde U_n}\eta)\in \mathcal P_s(\Gamma\times \Gamma)$. The sequence $(\Q_n)_{n\geq 1}$ is relatively compact since the sequences of the marginals are relatively compact. Let $\Q$ denote a limit point of a suitable subsequence.
Then $\Q$ is a  coupling of the distribution of $\bar \eta$ and the distribution of $\eta$.
Note that \eqref{eq:subset} implies that $\Q$ is concentrated on the set $\{(\xi_1,\xi_2)\in \Gamma^2: \xi_1\subset \xi_2\}$.
Since the intensity of $\bar \eta$ is greater or equal than one the intensity of  $\eta$, it follows that $\Q$ is concentrated on the diagonal  $\{(\xi,\xi)\in \Gamma^2: \xi \in \Gamma\}$. Hence, the distribution of $\bar \eta$ is equal to the distribution of $\eta$, which is $\R\in \mathcal P_s(\Gamma)$.
\end{proof}

From now on, we will often work with a fixed optimal process coupling $\q$ between $\P_0$ and $\P_1$.
By Remark \ref{rem:q_implies_Q} we may assume the probability space, on which $\q$ is defined, to be the canonical one, i.e. $\q$ is defined on $\mathcal M(\IR^d)^2$ equipped with the optimal distribution coupling $\Q$ induced by $\q$.

\subsection{Geodesics}\label{sec:geodesics}
In this section we will show that $\C$ defines a geodesic metric. We start with the following gluing lemma adapted to our context.
\begin{lem}[Gluing]
Let $\xi\in\mathcal M$ and $\q_0,\q_1\in\mathcal M(\IR^d\times\IR^d)$ such that $\q_0(\IR^d\times A)=\xi(A)=\q_1(A\times \IR^d)$, then there exist $\q\in\mathcal M(\IR^d\times\IR^d\times\IR^d)$ such that $\q(\cdot\times \IR^d)=\q_0$ and $\q(\IR^d\times \cdot)=\q_1$. 
\end{lem}
\begin{proof}
We approximate $\q_i$ by $\q_{0,n}=\q_0(\cdot\cap (\IR^d\times \Lambda_n))$ and $\q_{1,n}=\q_1(\cdot\cap (\Lambda_n\times \IR ^d))$ for $\Lambda_n\nearrow\IR^d$. By disintegrating the finite measures for each $n\in\IN$, there exists a family of probability measures $(\q^y_{0,n})_{y\in\IR^d}$ such that $\q_{0,n}(dx,dy)=\q^y_{0,n}(dx)\eins_{\Lambda_n}(y)\xi(dy)$ (similar for $\q_1$).
Then, define $\q$ to be the vague limit of $\q^y_{0,n}(dx)\q^y_{1,n}(dz)\eins_{\Lambda_n}(y)\xi(dy)$ as $n\to\infty$.
\end{proof}

We consider the case of cost functionals $c$ given as in \eqref{eq:costMeasures} with $\vartheta(x)=|x|^p$ for some $p\ge 1$.
Then, following the classical notation for the Wasserstein metric, let us denote
\begin{align}\label{eq:Wp}
\gls{Wp}(\P_0,\P_1)=\glsdisp{C}{\C_p(\P_0,\P_1)}^{1/p}
\end{align}
for $\P_0, \P_1\in \mathcal P_s(\mathcal M(\IR^d))$.

\begin{lem}
 If $p\ge 1$, then $(\mathcal P_s(\mathcal M(\IR^d)),\W_p)$ is an extended  metric space.
\end{lem}

We will see later, in Lemma \ref{thm:convdist_convweak}, that convergence in $\W_2$ implies weak convergence.

\begin{proof}
Clearly, $\W_p$ is non-negative. To show the symmetry
let $\P_0, \P_1\in \mathcal P_s(\mathcal M(\IR^d))$ with $\C_p(\P_0,\P_1)<\infty$ and let $(\Q_0,\q_0)$ be an optimal pair.
Applying the mass-transport principle
to the function $f:\IZ^d\times \IZ^d\to [0,\infty]$ defined by $f(u,v)=\IE_{\Q} [\int_{u+\Lambda_1\times v+\Lambda_1}|x-y|^p\q_0(dx,dy)]$, see \cite[\S 8]{Lyons} or \cite[Lemma 3.8]{H16}, yields  \[
\C_p(\P_0,\P_1)=\sum_{v\in \IZ^d}f(0,v)=\sum_{u\in \IZ^d}f(u,0)=\IE_{\Q} 
\int_{\IR^d\times \Lambda_1}|x-y|^p\q_0(dx,dy)\geq \C_p(\P_1,\P_0).
\]
Hence $\W_p(\P_0,\P_1)= \W_p(\P_1,\P_0)$ with the same optimal coupling $\q_0$.
Assuming $\W_p(\P_0,\P_1)=0$, then by equivariance 
\[
0=\IE \int_{\Lambda_n\times \IR^d}|x-y|^p \q_0(dx,dy)=\IE \int_{ \IR^d\times \Lambda_n}|x-y| ^p\q_0(dx,dy)
\]  
for all $n\in \IN$. Hence for all $n\in \IN$, it holds $\Q_0$-a.s.~that $x=y$ $\q_0$-a.e., or with other words $(\P_0)_{\Lambda_n}=(\P_1)_{\Lambda_n}$. This implies $\P_0=\P_1$.

To show the triangle inequality, let furthermore $\P_2\in \mathcal P_s(\mathcal M(\IR^d))$ with $\C_p(\P_1,\P_2)<\infty$ and let $(\Q_1,\q_1)$ be an optimal pair of $\P_1$ and $\P_2$. Applying the gluing procedure first to the probability measures $\Q_0$ and $\Q_1$ and then to the realisations of $\q_0$ and $\q_1$ we obtain an equivariant coupling $(\Q,\q)$ of $\P_0$, $\P_1$ and  $\P_2$. Two applications of the Minkowski inequality imply
\begin{align*}
\W_p(\P_0,\P_2)&= \Bigg(\IE_{\Q}\int_{\Lambda_n\times \IR^d \times \IR^d}\abs{x-z}^p d\q\Bigg)^{1/p}\\
&\le \Bigg(\IE_{\Q}\int_{\Lambda_n\times \IR^d \times \IR^d}\abs{x-y}^p d\q\Bigg)^{1/p}+\Bigg(\IE_{\Q}\int_{\Lambda_n\times \IR^d \times \IR^d}\abs{y-z}^p d\q\Bigg)^{1/p},
\end{align*}
where we abbreviated $\q(dx,dy,dz)$ by $d\q$.
By the mass-transport principle, the integration can be rephrased as
\begin{align*}
\IE_{\Q}\int_{\Lambda_n\times \IR^d \times \IR^d}\abs{y-z}^p d\q=\IE_{\Q}\int_{\IR^d\times \Lambda_n \times \IR^d}\abs{y-z}^p d\q=\IE_{\Q}\int_{\Lambda_n\times \IR^d }\abs{y-z}^p d\q_1.
\end{align*}
Together, we obtain the triangle inequality
\begin{align*}
    \W_p(\P_0,\P_2)\le \Bigg(\IE_{\Q}\int_{\Lambda_n\times \IR^d }\abs{x-y}^p d\q_0\Bigg)^{1/p}+\Bigg(\IE_{\Q}\int_{\Lambda_n\times \IR^d }\abs{y-z}^p d\q_1\Bigg)^{1/p}=    \W_p(\P_0,\P_1)+    \W_p(\P_1,\P_2).
\end{align*}
\end{proof}

Now, we are ready to define geodesics on the space of distributions of random measures. 
Let  $\P_0, \P_1\in \mathcal P_s(\mathcal M(\IR^d))$ with $\C_p(\P_0,\P_1)<\infty$ and let $(\Q,\q)$ be an optimal pair. For $t\in[0,1]$ define the map $\mathsf{geo}_t:\IR^d\times \IR^d\to \IR^d$ by $\mathsf{geo}_t(x,y)=x+t(y-x)$. Then, the push-forward $(\mathsf{geo}_t)_\#\q\in\mathcal M(\IR^d)$ is the process interpolating the points of the support of $\q$ according to the Euclidean geodesics. We let  $\P_t$ be its distribution, i.e. 
$\gls{Pt}=\mathsf{law}\big((\mathsf{geo}_t)_\#\q\big)$.

\begin{lem}
The family $(\P_t)_{t\in[0,1]}$ is a constant speed geodesic from $\P_0$ to $\P_1$ with respect to $\W_p$.
\end{lem}

\begin{proof}
It is sufficient to show that $(\P_t)_{t\in [0,1]}$ satisfies
$$ \C_p(\P_s,\P_t) \leq (t-s)^p \C_p(\P_0, \P_1)$$
for any $0\leq s<t\leq 1$. We will show that $(\P_t)_{t\in [0,1]}$ defines a geodesic. Let $0\leq s<t\leq 1$. We define the random equivariant couplings by 
$\q^{s,t}:=(\mathsf{geo}_s,\mathsf{geo}_t)(\q)$. 
For $s=0$ it then follows that
\[
\C_p(\P_0,\P_s)\leq \IE_{\Q}\left[\int_{\Lambda_1\times \IR^d }\abs{x-y}^p\q^{0,t}(dx,dy)\right]=t^p\C_p(\P_0,\P_1).
\]
Moreover, by the transport principle we obtain
\begin{align*}
\C_p(\P_s,\P_t)&\leq \IE_{\Q}\left[\int_{\Lambda_1\times \IR^d }\abs{x-y}^p\q^{s,t}(dx,dy)\right]\\
& =\IE_{\Q}\left[\int_{\IR^d \times \Lambda_1 }\abs{x-y}^p\q^{s,t}(dx,dy)\right]\\
&= \frac{(t-s)^p}{t^p}\IE_{\Q}\left[\int_{\IR^d \times \Lambda_1 }\abs{x-y}^p\q^{0,t}(dx,dy)\right]\\
&=\frac{(t-s)^p}{t^p}\IE_{\Q}\left[\int_{\Lambda_1\times \IR^d }\abs{x-y}^p\q^{0,t}(dx,dy)\right]\\
&=(t-s)^p\C_p(\P_0,\P_1).
\end{align*}
\end{proof}

The set of stationary point processes is geodesically closed in the following sense.
\begin{prop}\label{prop:closed_sPP}
    For  $\P_i\in \mathcal P_s(\Gamma)$ there exists a constant speed geodesic $(\P_t)_{0\leq t\leq 1}$  from $\P_0$ to $\P_1$ with respect to $\W_p$ such that $\P_t\in \mathcal P_s(\Gamma)$ for $0\leq t\leq 1$.
\end{prop}
\begin{proof}
    Let $\q$ be the matching  given by Proposition \ref{prop:matching}. In particular, $\q$ is an optimal process coupling. Then $\P_t=((\mathsf{geo}_t)_\#\q)\in \mathcal P_s(\Gamma)$ defines a  constant speed geodesic.
\end{proof}

\begin{ex}\label{ex:grid}
Let us find the optimal pair of the stationarized grid and its i.i.d.~perturbation.
Consider two random variables $Y,Z$, both uniformly distributed on $\Lambda_1$ and an independent family of random variables $\{X_z\}_{z\in\IZ^d}$, also uniformly distributed on $\Lambda_1$.
Let $\P_0$ and $\P_1$ be the distributions of the point processes
\begin{align*}
    \xi^\bullet=\sum_{z\in\IZ^d} \delta_{z+Y}\ \text{and}\ \eta^\bullet=\sum_{z\in\IZ^d} \delta_{z+Z+X_z}.
\end{align*}
It is difficult not to see that the optimal distribution-coupling $\Q$ is given by the joint distribution of $(\xi^\bullet,\eta^\bullet)$ for the straightforward choice of $Z=Y$ and the optimal process coupling is given by
$$ \q ^\bullet=\sum_{z\in\IZ^d} \delta_{(z+Y,z+Y+X_z)}.$$
The geodesic is then given by the distribution of
\[(\mathsf{geo}_t)_\#\q^\bullet= \sum_{z\in\IZ^d} \delta_{z+Y+tX_z},\quad t\in[0,1]\]
hence the distribution of the displacement interpolation $\P_t=\mathsf{law}\big((\mathsf{geo}_t)_\#\q\big)
\in\mathcal P_s(\mathcal M(\IR^d))$ generates stationarized grids with independent perturbations $tX_z$, which are uniformly distributed on $\Lambda_a(z)$.
\end{ex}

In general, couplings between two random measures (or point processes) or their distributions need not be unique. However, if the distributions lie on a geodesic which can be extended in at least one direction, we have uniqueness.
\begin{lem}\label{lem:uniqueness_cpl}
Let  $0<t< 1$ and $\P_0,\P_t,\P_1$ be point processes s.t.\ $\P_t$ is an $t$-midpoint of $\P_0$ and $\P_1$, i.e. 
\begin{equation}\label{eq:tmidpoint}
\W_p(\P_0,\P_t)=t \cdot \W_p(\P_0,\P_1),\quad  \W_p(\P_t,\P_1)=(1-t) \cdot \W_p(\P_0,\P_1).
\end{equation}
Then there exists an unique optimal coupling of $\P_0$ and $\P_t$ and of $\P_t$ and $\P_1$.
\end{lem}

\begin{proof}
Let $(\Q_0,\q_0)$ be an optimal pair of $\P_0$ and $\P_t$ 
and $(\Q_1,\q_1)$ an optimal pair of $\P_t$ and $\P_1$. We will show that these couplings are unique by combining \eqref{eq:tmidpoint} with several triangle inequalities which allow us to lift the corresponding statements for points in $\IR^d$ to random measures.

To this end, define $ \Q$ and $\q$ by gluing first $ \Q_0$ and $ \Q_1$ and then $\q_0$ and $\q_1$.
The Minkowski triangle inequality reads
 \begin{align}
\Bigg(\int_{\Lambda_n\times \IR^d \times \IR^d}\abs{x-z}^pd\q\Bigg)^{1/p}
\leq \Bigg(\int_{\Lambda_n\times \IR^d \times \IR^d}\abs{x-y}^pd\q\Bigg)^{1/p}+
\Bigg(\int_{\Lambda_n\times \IR^d \times \IR^d}\abs{y-z}^pd\q\Bigg)^{1/p} \label{eq:triangle1}
\end{align}
for any $n>0$.
Furthermore 
\begin{align*}
&n^{d/p}\C_p(\P_0,\P_1)^{1/p}\\
&\leq \IE_{\Q}\left[\int_{\Lambda_n\times \IR^d \times \IR^d}\abs{x-z}^pd\q\right]^{1/p}\\
&\leq \IE_{\Q}\left[ \left(\left(\int_{\Lambda_n\times \IR^d \times \IR^d}\abs{x-y}^pd\q\right)^{1/p}+
\left(\int_{\Lambda_n\times \IR^d \times \IR^d}\abs{y-z}^pd\q\right)^{1/p}\right)^{p}\right]^{1/p} \numberthis \label{eq:triangle2}\\
&\leq \IE_{\Q}\left[\int_{\Lambda_n\times \IR^d \times \IR^d}\abs{x-y}^pd\q\right]^{1/p}
+\IE_{\Q}\left[\int_{\Lambda_n\times \IR^d \times \IR^d}\abs{y-z}^pd\q\right]^{1/p}\numberthis \label{eq:triangle3}\\
&= \IE_{\Q}\left[\int_{\Lambda_n\times \IR^d \times \IR^d}\abs{x-y}^pd\q\right]^{1/p}
+\IE_{\Q}\left[\int_{ \IR^d \times \Lambda_n\times\IR^d}\abs{y-z}^pd\q\right]^{1/p}\\
&=n^{d/p}\C_p(\P_0,\P_t)^{1/p}+n^{d/p}\C_p(\P_t,\P_1)^{1/p}\\
&=n^{d/p}\C_p(\P_0,\P_1)^{1/p},
\end{align*}
by the assumption \eqref{eq:tmidpoint}. Hence all above inequalities are equalities. 
By \cite[Chapter8, Problem 4]{zyg}, equality in the triangle inequalities implies that the two functions involved are collinear. For instance, from \eqref{eq:triangle3} we obtain the existence of $\kappa>0$ such that $ \Q$-a.s.
\begin{equation}\label{eq:equal1}
 \left(\int_{\Lambda_n\times \IR^d \times \IR^d}\abs{x-y}^pd\q\right)^{1/p}=
\kappa\left(\int_{\Lambda_n\times \IR^d \times \IR^d}\abs{y-z}^pd\q\right)^{1/p}.
\end{equation}
Combining \eqref{eq:triangle1} and \eqref{eq:triangle2} yields that $ \Q$-a.s.
\begin{align}\label{eq:equal2}
&\left(\int_{\Lambda_n\times \IR^d \times \IR^d}(\abs{x-y}+\abs{y-z})^pd\q\right)^{1/p}\\
&= \left(\int_{\Lambda_n\times \IR^d \times \IR^d}\abs{x-y}^pd\q\right)^{1/p}+
\left(\int_{\Lambda_n\times \IR^d \times \IR^d}\abs{y-z}^pd\q\right)^{1/p}.
\end{align}
Hence, by \cite[Chapter8, Problem 4]{zyg} there exists (a priori random) constant $\alpha>0$ such that 
\[
\alpha \abs{x-y}= \abs{y-z} \quad \forall (x,y,z)\in \mathsf{supp}(q)\cap \Lambda_n\times \IR^d \times \IR^d.
\] 
Since we have $\Q$ a.s.\ equality in the line of inequalities in \eqref{eq:triangle1} we obtain
\[
\left(\int_{\Lambda_n\times \IR^d \times \IR^d}\abs{x-z}^pd\q\right)^{1/p}=\left(\int_{\Lambda_n\times \IR^d \times \IR^d}(\abs{x-y}+\abs{y-z})^pd\q\right)^{1/p} \quad \Q a.s.
\] 
so that 
\[
y=(1-r)x+rz \quad \forall (x,y,z)\in \mathsf{supp}(q)\cap \Lambda_n\times \IR^d \times \IR^d,
\]
where $r\in [0,1]$ is random and potentially depends on the triple $(x,y,z)$. Our first goal is to show that $r$ is a deterministic constant. First note that 
\begin{align*}
\alpha \abs{x-y}&=\alpha\abs{x-(1-r)x-rz}=\alpha r\abs{x-z}\\
&= \alpha \frac{r}{1-r} \abs{y-z}=\alpha\frac{r}{1-r}\abs{x-y}.
\end{align*}
Hence $r$ is uniquely determined by $\alpha$. It remains to show that $\alpha$ is deterministic. To this end, observe that by \eqref{eq:equal1}
\begin{align*}
\left(\int_{\Lambda_n\times \IR^d \times \IR^d}\abs{x-y}^pd\q\right)^{1/p}
&= \frac{1}{\alpha} \left(\int_{\Lambda_n\times \IR^d \times \IR^d}\abs{y-z}^pd\q\right)^{1/p}\\
&=\frac{1}{\kappa\alpha} \left(\int_{\Lambda_n\times \IR^d \times \IR^d}\abs{x-y}^pd\q\right)^{1/p}.
\end{align*}
This shows that $ \Q$-a.s. $\frac{r}{1-r}=\frac{1}{\alpha}=\kappa$ is deterministic. Finally, we will show that in fact $r=t$ as expected.
By the mass-transport principle and our assumption \eqref{eq:tmidpoint} we have
\begin{align*}
&n^{d/p}(1-t)\C_p(\P_0,\P_1)^{1/p}=n^{d/p}\C_p(\P_t,\P_1)^{1/p}\\
&=\IE_{\Q}\left[\int_{\IR^d \times \Lambda_n\times  \IR^d}\abs{y-z}^2d\q\right]^{1/p}\\
&=\IE_{\Q}\left[\int_{ \Lambda_n\times \IR^d \times  \IR^d}\abs{y-z}^pd\q\right]^{1/p}\\
&=(1-r)\IE_{\Q}\left[\int_{\Lambda_n\times\IR^d \times   \IR^d}\abs{x-z}^pd\q\right]^{1/p}\\
&=n^{d/p}(1-r)\C_p(\P_0,\P_1)^{1/p}
\end{align*}
Hence $r=t$. Note that this argument is independent of $n>0$. Hence we have that $ \Q$-a.s. 
\[
y=(1-t)x+tz \quad \forall (x,y,z)\in \mathsf{supp}(q)\cap\IR^d \times \IR^d \times \IR^d.
\]
In particular $ \Q_1, \q_1$ is obtained by the pushforward of $ \Q_0, \q_0$ under the map $(x,y)\mapsto (y,\frac{y-(1-t)x}{t})$ and hence is uniquely determined. Switching the roles of $ \Q_0$ and $ \Q_1$ yields the uniqueness of $ \Q_0$.
\end{proof}

Finally, we show that convergence in $\W_2$ implies weak convergence.
\begin{lem}\label{thm:convdist_convweak}
Let $f\geq 0$ be a continuous function with compact support and $(\P_n)_n$ and $(\tilde{\P}_n)_n$ sequences of stationary point processes with bounded intensity such that \begin{align*}
\W_2(\mathsf{P}_n,\tilde{\P}_n)\xrightarrow{n\to \infty} 0. 
\end{align*}
Then \begin{align*}
\abs{\IE_{\P_n}\left[e^{-\int fd\xi}\right]-\IE_{\tilde{\P}_n}\left[e^{-\int fd\xi}\right]}\xrightarrow{n\to \infty}0.
\end{align*}
If in addition $\P_n\xrightarrow{\text{weakly}}\P$, for some stationary $\P$, then also $\tilde{\P}_n\xrightarrow{\text{weakly}}\P$.
\end{lem}
\begin{proof}
Let $ (\Q_n,\q_n)$ be an optimal pair of $\P_n$ and $\tilde{\P}_n$. Then \begin{align*}
&\IE_{\P_n}\left[e^{-\int fd\xi}\right]-\IE_{\tilde{\P}_n}\left[e^{-\int fd\xi}\right]\\
&=\IE_{ \Q_n}\left[e^{-\int f(x)d\q_n(dx,dy)}-e^{-\int f(y)d\q_n(dx,dy)}\right]\\
&\leq \IE_{ \Q_n}\left[ \int \abs{f(y)-f(x)}d\q_n(dx,dy)\right],
\end{align*}
where in the last line we applied the mean value theorem to the function $e^{-x}$, $x>0$, and used the fact that $f\geq 0$. Now let $g$ be a $L$-Lipschitz function with compact support such that $\norm{f-g}_{\infty}\leq \epsilon$. Furthermore assume that $\supp{f}\cup \supp{g}\subset \Lambda_N$. Then \begin{align*}
&\IE_{ \Q_n}\left[ \int \abs{f(y)-f(x)}d\q_n(dx,dy)\right]\\
&\leq L \IE_{ \Q_n}\left[ \int_{\Lambda_N \times \IR^d} \abs{x-y}d\q_n(dx,dy)\right]+ \epsilon \IE_{\P_n}\left[ \xi(\Lambda_N)\right]+
\epsilon \IE_{\tilde{\P}_n}\left[ \xi(\Lambda_N)\right].
\end{align*}
 Then by stationarity we can estimate the first term in the above line by a constant times \begin{align*}
&\IE_{ \Q_n}\left[ \int_{\Lambda_1\times \IR^d} \abs{x-y}d\q_n(dx,dy)\right]\\
&\leq \IE_{ \Q_n}\left[ \left(\int_{\Lambda_1\times \IR^d} \abs{x-y}^2 d\q_n(dx,dy)\right)^{1/2}\xi(\Lambda_1)^{1/2}\right]\\
&\leq \IE_{ \Q_n}\left[ \int_{\Lambda_1\times \IR^d} \abs{x-y}^2 d\q_n(dx,dy)\right]^{1/2}  \IE_{\P_n}\left[ \xi(\Lambda_1)\right]^{1/2}\\
&=\W_2^2(\mathsf{P}_n,\tilde{\P}_n)^{1/2} \IE_{\P_n}\left[ \xi(\Lambda_1)\right]^{1/2},
\end{align*} 
the last line converges to $0$ as $n\to \infty$. This proves the claim since $\varepsilon$ can be chosen arbitrarily small.

The second part of the statement is a consequence of the equivalence of weak convergence and convergence of the Laplace functionals, see \cite[Theorem 4.11]{Kallenberg2}.
\end{proof}

\section{Constructions and Approximations}
From now on, we will work with point processes. In this section, we first recall the definition and properties of the specific relative entropy.
Then, we introduce various constructions of point processes that play an important role in our theory and in the end of this section, we present an approximation result concerning the specific relative entropy via one particular construction of modified stationary point processes.

Recall that the set of all stationary point processes on $\IR^d$ with intensity one is
\[
\spp1=\{\P\in \gls{Ps}(\mathcal{M}(\IR^d))\mid \P(\Gamma)=1\},
\] 
where $\Gamma$ is the configuration space on $\IR^d$ equipped with the topology of vague convergence.
 For any set $A\subseteq\IR ^d$ define the projection $\mathrm{pr}_A:\mathcal M(\IR^d)\to \mathcal M (A)$ on the space of Radon measures supported on $A$ via $\gls{pr_A}(\xi)=\xi(A\cap \cdot)$. For $\P\in \spp1$ denote the restriction $\gls{P_A}:= (\mathrm{pr}_A)_\# \P$. Recall that $\Poi$ is the distribution of the Poisson point process.
\subsection{Specific entropy}
 Assume $\P_A \ll\Poi_A$ for all bounded $A\in\mathcal B(\IR^d)$ and hence the Radon-Nikodym-derivative $\frac{d\P_A}{d\Poi_A}$ exists.\footnote{Even though this is implied by $\P \ll\Poi$, this assumption is  too restrictive as the only stationary distribution with $\P \ll\Poi$ is $\P =\Poi$ itself.}
The \emph{specific relative entropy} $ \mathcal E( \P)$ with respect to the Poisson point process, see \cite{EHL21,Serfaty,RAS,Georgii}, is defined as
\begin{align}\label{eq:specific_rel_ent}
    \gls{specE}( \P):=\lim_{n\to\infty} \frac 1{n^d}\mathsf{Ent}(\P_{\Lambda_n}|\Poi_{\Lambda_n}):=\lim_{n\to\infty} \frac 1{n^d} \int \log\Big(\frac{d\P_{\Lambda_n}}{d\Poi_{\Lambda_n}}(\xi)\Big)d\mathsf{P}_{\Lambda_n}(\xi).
\end{align}
Here, $\gls{Ent}$ is the classical relative entropy (or Kullback-Leibler-divergence). It has the following basic properties, see \cite{Serfaty,RAS,Georgii}.
\begin{lem}\label{lem:Georgii} Let $\P\in \spp1$. Then
\begin{enumerate}[a)]
    \item The limit in \eqref{eq:specific_rel_ent} exists in $[0,\infty]$.
    \item The map $\P\mapsto\mathcal E(\P)$ is affine and lower semi-continuous.
    \item The specific relative entropy vanishes iff $\P=\Poi$.
    \item The limit in the definition of the specific entropy can be replaced by a supremum, i.e. \begin{align}\label{eq:def_entrop_sup}
        \mathcal E( \P)=\sup_{n\in \IN} \frac 1{n^d}\mathsf{Ent}(\P_{\Lambda_n}|\Poi_{\Lambda_n})
    \end{align}
\end{enumerate}
\end{lem}

\begin{rem}
Contrary to the Wasserstein distance $\W_p$, the specific relative entropy of two stationary distributions with different intensities can be finite. For instance, consider $\Poi_1,\Poi_2$ with intensities $\int\xi(\Lambda_1)d\Poi_1=1,\int\xi(\Lambda_1)d\Poi_2=2$, respectively. Then, $\W_p(\Poi_2,\Poi_1)=\infty$ since on each box $\Lambda_n$ there are $\approx n^d$ points that have to be transported over distance $\approx n$. On the other hand, $\frac{d\Poi_{2,\Lambda_n}}{d\Poi_{1,\Lambda_n}}(\xi)=2^{\xi(\Lambda_n)}e^{-n^d}$ and hence $\mathcal E (\Poi_2)=2\log 2 -1<\infty$.
\end{rem}

\subsection{Tiling and Stationarizing}\label{subsec:tiling}
Next, we present a two step construction to obtain a stationary point process on $\IR^d$ from a point process on $\Lambda_n$. These constructions were for example used in \cite{Lebl__2016} and \cite{EHL21}.

The first construction concerns gluing together independent copies of the same process to obtain a \textit{tiling} of $\IR^d$.
Let $\P$ be the distribution of a point process on $\Lambda_n$. 
Consider the map 
\begin{align}\label{eq:map_til}
    (\Gamma_{\Lambda_n})^{\IZ^d}\to \Gamma_{\IR^d},\quad (\xi_{z})_{z\in \IZ^d} \mapsto \sum_{z \in \IZ^d}\theta_{nz}\xi_z.
\end{align}
Define $\gls{Ptil}$ as the pushforward of  the product measure $\otimes_{z\in \IZ^d}\P$ under the map \eqref{eq:map_til}.
The process $\P^{til}$ is in general not stationary. To obtain a stationary process we have to additionally integrate the shifted versions of $\P^{til}$ over the box $\Lambda_n$. This is done in the second step, \textit{stationarization}:

For a distribution $\P$ of a point process on $\Lambda_n$ define \begin{align}\label{def:stationatized_process}
   \gls{Pstat}=n^{-d}\int_{\Lambda_n}(\theta_z)_{\#}\P^{til}dz.
\end{align}
In other words, if $\xi\sim\P^{til}$ and $U\sim\mathcal U(\Lambda_n)$ is an independent uniformly distributed random variable, then $\P^{stat}$ is the distribution of $\theta_U\xi$.
 We note that the process $\P^{stat}$ is stationary.  Furthermore, we can bound the specific entropy of $\P^{stat}$ by the entropy of $\P$ as follows.
 \begin{lem}
 For any distribution $\P\in\mathcal P(\Gamma_{\Lambda_n})$ of point processes it holds
 \begin{align}\label{eq:bound_entropy_by_cube}
      \mathcal{E}(\P^{stat})\le n^{-d}\mathsf{Ent}(\P|\Poi_{\Lambda_n}).
  \end{align}
 \end{lem}
\begin{proof}
    By convexity of the relative entropy  and the fact that $\mathcal E\left( \P^{til}\right)=n^{-d}\mathsf{Ent}(\P|\Poi_{\Lambda_n})$, see \cite[Chapter 5 and 6]{RAS}, we have
\begin{align*}
    \mathcal{E}(\P^{stat})&=\sup_{m\in \IN}m^{-d}\mathsf{Ent}(\P^{stat}_{\Lambda_m}|\Poi_{\Lambda_m})\\
    &=\sup_{m\in \IN}m^{-d}\mathsf{Ent}(n^{-d}\int_{\Lambda_n}(\theta_z)_{\#}\P^{til}_{\Lambda_m}dz|\Poi_{\Lambda_m})  \\
    &\leq \sup_{m\in \IN}m^{-d}n^{-d}\int_{\Lambda_n}\mathsf{Ent}\left((\theta_z)_{\#}\P^{til}_{\Lambda_m}\mid \mathsf{Poi}_{\Lambda_m}\right)dz  \\
    &\leq n^{-d}\int_{\Lambda_n} \sup_{m\in \IN}m^{-d}\mathsf{Ent}\left((\theta_z)_{\#}\P^{til}_{\Lambda_m}\mid \mathsf{Poi}_{\Lambda_m}\right)dz  \\
    &=  n^{-d}\int_{\Lambda_n} \mathcal{E}\left((\theta_z)_{\#}\P^{til}\right)dz  \\
      &=  n^{-d}\int_{\Lambda_n} \mathcal{E}\left(\P^{til}\right)dz  =\mathcal{E}\left(\P^{til}\right)  =n^{-d}\mathsf{Ent}(\P|\Poi_{\Lambda_n}).
\end{align*}
\end{proof}

\subsection{Semigroups}
We now define semigroups on the space of point processes, first on the space of stationary point processes on $\IR^d$, afterwards on boxes $\Lambda_n$.

 On the configuration space $\Gamma$ we fix a label map 
 \begin{align}\label{eq:def_label}
 \gls{l}:\Gamma\to (\IR^d)^{\IN}, \xi\mapsto (\ell(\xi)^i)_{i\in \IN}. 
 \end{align}
 That is, the map $\ell$  satisfies the following: For all $i\in \IN$ it holds that  $\ell(\xi)^i\in \xi$ and  for all  $x\in \supp(\xi)$ there exists a unique $i\in \IN$ s.t. $\ell (\xi)^i=x$.
Let $\gls{Wiener}$ be the standard Wiener measure on $\mathcal C\left([0,\infty\right),\IR^d)$ and define for $t\geq 0$ a map $F_t:\Gamma\times \mathcal C([0,\infty),\IR^d)^{\IN}\to\Gamma$ by 
\begin{align}\label{eq:F_t_einsvonvielen}
F_t(\xi,(\omega^i)_{i\in \IN})=\{\ell (\xi)^i+\omega^i_t\}_{i\in\IN}.
\end{align}
For a stationary point process $\P$ on  $\IR^d$ then define 
\begin{align}\label{def:semi_wholespace}
\gls{SemiP}=(F_t)_{\#}\left[\P \otimes (\otimes_{i=1}^{\infty}\mathbb{W})\right].
\end{align}

\begin{lem}\label{lem:semigroup}
$(\gls{Semi})_{t\geq 0}$ defines a $\mathcal C^0$-semigroup on $(\spp1,\W_2)$, in the sense that for all $\P\in \spp1$\begin{enumerate}[i)]
    \item $\Semi_0\P=\P$
    \item $\Semi_s\left(\Semi_t \P\right)=\Semi_{s+t}\P, \quad \forall s,t\geq 0$ 
    \item $\Semi_t\P\xrightarrow{t\to 0}\P$ in $\W_2$.
\end{enumerate} 
\end{lem}
\begin{proof} 
The first item follows from $\omega(0)=0$ $\mathbb W$-a.s..
The second item follows directly from the independence of the increments of Brownian motion. For the last item let $\P\in\spp1$,  fix a label map $\ell:\Gamma\to (\IR^d)^{\IN}, \xi\mapsto (\ell (\xi)^i)_{i\in \IN}$ and consider the random measure
\[
\q=\{\left(\ell (\xi)^1,\ell (\xi)^1+\omega^1_t\right),\left(\ell (\xi)^2,\ell (\xi)^2+\omega^2_t\right),\dots\}\in \Gamma_{\IR^d\times \IR^d},
\]
where $(\xi,(\omega^i)_{i\in \IN})$ have distribution $\P \otimes (\otimes_{i=1}^{\infty}\mathbb{W})$.
Using the formulation for the cost of Corollary \ref{cor:equiv_cost}, we can estimate
\begin{align*}
   \W_2^2(\P,\Semi_t\P)&\leq 
   \liminf_{n\to \infty}n^{-d}\IE\left[\int_{\Lambda_n\times \IR ^d}\abs{x-y}^2\q(dx,dy)\right]\\
   &=\liminf_{n\to \infty}n^{-d}\IE_{\P}\left[\int\int_{\Lambda_n}\abs{x-(x+\omega_t)}^2\xi(dx){\mathbb{W}}(d\omega)
   \right]\\
   &=t\liminf_{n\to \infty}n^{-d}\IE_{\P}\left[\int_{\Lambda_n}\xi(dx)\right]=t.
\end{align*}
This shows that $\W_2(\P,\Semi_t\P)\to 0$ as $t\to 0$.
\end{proof}

Now let us turn to a fixed box $\Lambda_n$, $n\in \IN$, where we need a rigorous definition of reflected paths. For a continuous path $\omega:[0,\infty)\to \IR^d$ with $\omega(0)\notin \bigcup_{u\in n\IZ^d}(u+\partial \Lambda_n)$ there exists a unique $z\in n\IZ^d$ s.t. $\omega(0)\in z+\Lambda_n$. Note that for processes of our consideration (e.g.~restrictions of stationary point processes), particles will almost surely not lie on the above lattice.

We want to define the new path $\refl_{\Lambda_n}(\omega)$, which is the path $\omega$ reflected at the boundary of $z+\Lambda_n$. The construction we use, can be found in \cite[Exercise 8.9]{Shreve}. Let $\varphi:\IR\to \IR$ be the function which satisfies $\varphi(n/2+2k)=n/2$, $\varphi(-n/2+2k)=-n/2$ for all $k\in n\IZ$ and is linear between these points and define $\Phi:\IR^d\to \Lambda_n$ via $\Phi(x)=(\varphi(x^1),\dots,\varphi(x^d))$. For a continuous path $\omega:[0,\infty)\to \IR^d$ with $\omega(0)\in z+\Lambda_n$, $z\in n\IZ^d$, we define the path $\refl_{\Lambda_n}(\omega):[0,\infty)\to z+\Lambda_n$ via
\begin{align}\label{def:reflectedpath}
\refl_{\Lambda_n}(\omega)_t := z+ \Phi(\omega_t-z)\;.
\end{align}
Note that for a standard Brownian motion $(B_t)$ and $x\in z+\Lambda_n$, $z\in n\IZ^d$, the process $(\refl_{\Lambda_n}(x+B_t))_{t\geq 0}$ is a Brownian motion reflected at the boundary of $z+\Lambda_n$ and started at $x$. In analogy to the above, define $F_t:\Gamma_{\Lambda_n} \times \mathcal C\big([0,\infty),\IR^d\big)^{\IN}\to \Gamma_{\Lambda_n}$ by 
\begin{align}\label{eq:def_GF_box}
F_t(\xi,(\omega^i)_{i\in \IN})=\{\refl_{\Lambda_n}(\ell (\xi)^i+\omega^i)_t\}_{i\in\IN}.
\end{align}
For the distribution of a point process $\P$ on $\Lambda_n$ define the semi-group 
\begin{align}\label{def_semi_box}
\Semi^{\Lambda_n}_t\P={F_t}_{\#}\left[\P \otimes  (\otimes_{i=1}^{\infty}\mathbb{W})\right].
\end{align}
Indeed, Lemma \ref{lem:semigroup} continues to hold for $\mathsf S_t^{\Lambda_n}$ for the classical unnormalized $L^2$-Wasserstein distance, we omit the proof.

\subsection{Integrability}
 In the following we will introduce  modifications of our point processes, which will enable us to lift results from the  Euclidean space $\IR^d$ to the space of configurations. First, let us record two technical lemmas.
\begin{lem}\label{lem:cost_estimate}
For an optimal pair $\Q,\q$ we have 
    \begin{align}\label{eq:cost1}
\IE_\Q[\int_{\Lambda_n\times \Lambda_n^c}\abs{x-y}^2 d\mathsf{q}(x,y)]=o(n^d)
    \end{align}
   as $n\to\infty$ and for any $0<\epsilon<1/2$
    \begin{align}\label{eq:cost2}
    \IE_{ \Q} \left[\mathsf \q(\Lambda_n\times \Lambda_n^c)\right] = o(n^{d-\epsilon}).
    \end{align}
\end{lem}
    \begin{proof}
We begin with
\begin{align*}
&\IE_{ \Q} \int_{\Lambda_n\times \Lambda_n^c}\abs{x-y}^2 d\q(x,y)\\
&\leq \IE_{ \Q} \int_{\Lambda_n\setminus \Lambda_{n-\sqrt n}\times \IR^d}\abs{x-y}^2 d\mathsf{q}(x,y)+\IE_{ \Q} \int_{\Lambda_{n-\sqrt n}\times \Lambda_n^c}\eins_{\abs{x-y}\geq \sqrt n}\abs{x-y}^2 d\mathsf{q}(x,y)\\
&\leq (n^d-(n-\sqrt n)^d)\IE_{ \Q} \int_{\Lambda_1\times \IR^d}\abs{x-y}^2 d\mathsf{q}(x,y)+n^d \IE_{ \Q} \int_{\Lambda_1\times \IR^d}\eins_{\abs{x-y}\geq \sqrt n}\abs{x-y}^2 d\mathsf{q}(x,y).
\end{align*}
Dividing by $n^d$ and letting $n\to \infty$ shows \eqref{eq:cost1}.
To show \eqref{eq:cost2}, we split $\Lambda_n$ again into
\begin{align}
\IE_{ \Q} \left[\mathsf \q(\Lambda_n\times \Lambda_n^c)\right]&=\IE_{ \Q} \left[\mathsf \q((\Lambda_{n}\setminus \Lambda_{n-n^{\epsilon}})\times \Lambda_n^c)\right]+\IE_{ \Q} \left[\mathsf \q(\Lambda_{n-n^{\epsilon}}\times \Lambda_n^c)\right]\nonumber\\
&\leq \IE_{ \Q} \left[\mathsf \q((\Lambda_{n}\setminus \Lambda_{n-n^{\epsilon}})\times \IR^d)\right]+n^{-2\epsilon}\IE_{ \Q}  \Big[ \int_{\Lambda_{n-n^{\epsilon}}\times \Lambda_n^c} \abs{x-y}^2 d\mathsf \q(x,y)\Big]\nonumber\\
&\le (n^d-(n-n^{\epsilon})^d)+n^{-2\epsilon}\IE_{ \Q}  \Big[ \int_{\Lambda_{n}\times \IR ^ d} \abs{x-y}^2 d\mathsf \q(x,y)\Big]\label{eq:Markov_ineq},
\end{align}
where we used the Markov inequality in the last step.
After multiplication with $n^{-d+\epsilon}$ and letting $n\to\infty$, this vanishes, since
\[
\limsup_{n\to \infty}n^{-d}\IE_{ \Q}\Big[ \int_{\Lambda_n\times \IR^d} \abs{x-y}^2 d \q(x,y)\Big]=\C_2(\P,\mathsf R)<\infty.
\]
\end{proof}

The following disintegration formula for the relative entropy for general probability measures will also be useful. Its proof is a direct consequence of disintegration of measures.

\begin{prop}\label{prop:disint_entr}
Let $X$,$Y$ be Polish spaces. Let $\beta$, $\gamma$ be two probability measures on $X$ and let $T:X\to Y$ be a measurable map. Let $\bar{\beta}=T_{\#}\beta$ and $\bar{\gamma}=T_{\#}\gamma$ and let $\beta(\cdot\mid T=y)$ and $\gamma(\cdot\mid T=y)$ denote the regular conditional probabilities. Then
\[
\ent(\beta\mid \gamma)=\ent(\bar{\beta}\mid \bar{\gamma})+\int_Y\ent(\beta(\cdot\mid T=y)\mid \gamma(\cdot\mid T=y))d\bar{\beta}(y).
\]
In particular, we have $\ent(\beta\mid \gamma)\geq\ent(\bar{\beta}\mid \bar{\gamma})$.
\end{prop} 

The last lemma shows that uniformly bounded specific entropies imply uniform integrability of the intensities.

\begin{lem}\label{lem:unif_int}
Let $(\P_i)_{i\in I}\subseteq \mathcal P_s(\Gamma)$ be a family of distributions of stationary point processes on $\IR^d$ with \[
\sup_{i\in I}\mathcal{E}(\P_i)<\infty \text{ and } \sup_{i\in I}\IE_{\P_i}\left[\xi(\Lambda_1)\right]<\infty
.\] 
If $(\xi_i)_{i\in\IN}$ is a family of random variables, defined on some common probability space, with $\xi_i \sim \P_i$ for all $i$, then 
\[
\sup_{i\in I}\IE\left[\xi_i(\Lambda_1)\log(\xi_i(\Lambda_1))\right]<\infty.
\] 
In particular, the family  $(\xi_i(\Lambda_1))_{i\in I}$ is uniformly integrable.
\end{lem}
\begin{proof}
 By assumption, the intensities of the processes 
  ${\P}_i$ are uniformly bounded, say by  $K$. Put $p_i(k)={\P}_i( \xi(\Lambda_1)=k)$. Then the Shannon-entropy of the number statistic of $({\P}_i)_{\Lambda_1}$ is uniformly bounded in $i$ because 
\begin{align}\label{eq:shannon}
 -\sum_{k\in\IN} p_i(k)\log\big(p_i(k)\big)\lesssim \sum_{k\in\IN}kp_i(k)+e^{-k/2}\lesssim K+\sum_{k\in\IN}e^{-k/2}<\infty.
 \end{align}
 The first inequality follows from considering two cases. If   $-p_i(k)\log\big(p_i(k)\big)\leq kp_i(k)$ holds, the inequality is clear. On the other hand, if this  does not hold, we have $-p_i(k)\log\big(p_i(k)\big)\geq kp_i(k)$, which is equivalent to $p_i(k)\leq e^{-k}$. Then we can bound \[
p_i(k)\log\big(1/p_i(k)\big)\lesssim \sqrt{p_i(k)}\lesssim e^{-k/2},
 \]
 where we used the fact that the function $-\sqrt{x}\log(x)$ is bounded on $[0,1]$.
 Since the specific entropy of a stationary point process $\P$ is given by 
 \[
\sup_{n\geq 1}n^{-d}\ent( \P_{\Lambda_n}\mid \Poi_{\Lambda_n}),
\]
 we conclude that the relative entropies of the restricted processes $({\P}_i)_{\Lambda_1}$ w.r.t.\ $\Poi_{\Lambda_1}$ are also uniformly bounded in $i$. 
 By the disintegration formula for the entropy also the relative entropy of the number statistics, i.e.\ of the random variabls $\xi(\Lambda_i)$ are uniformly bounded in $i$, say by the same constant $K>0$ as above.
 The relative entropies of the number statistics are given by
 \begin{align}\label{eq:entropy_count_dist}
  \sum_{k\in\IN} p_i(k)\log\Big( p_i(k)\Big)-\sum_{k\in\IN} p_i(k)\log\Big( \mathsf{poi}(k)\Big)=\sum_{k\in\IN} p_i(k)\log\Big( \frac{p_i(k)}{\mathsf{poi}(k)}\Big).
 \end{align}
 Combining \eqref{eq:shannon} and \eqref{eq:entropy_count_dist} yields that   the terms \[
-\sum_{k\in\IN} p_i(k)\log\Big( \mathsf{poi}(k)\Big)= 1+\sum_{k\in\IN} p_i(k)\log(k!)
 \]
 also are uniformly bounded in $i$. By Stirling's approximation, i.e. $k!>\sqrt{2\pi k}(k/e)^ke^{1/(12k+1)}$, we obtain the boundedness of \[
\sum_{k\in\IN} p_i(k)\log(k)k=\IE\left[{\xi}_i(\Lambda_1)\log\left({\xi}_i(\Lambda_1)\right)\right]. 
 \]
This implies uniform integrability of $({\xi}_i(\Lambda_1))_{i\in I}$.
\end{proof}

\subsection{Modification}
 
Finally, we present the modification procedure which generates point processes with matching point statistics in a box $\Lambda_n$, while keeping the distribution inside $\Lambda_{n-1}$ untouched and the cost $\C_2$ as well as the specific relative entropy $\mathcal E$ controlled.
\begin{thm}\label{thm:Modification}
Let $\P,\mathsf R\in\spp1$ have finite cost $\C_2(\P,\mathsf R)<\infty$ with optimal pair $(\Q,\q)$ and choose $\q$ to be a matching as in Proposition \ref{prop:matching}. Then, there exists a sequence of pairs $(\tilde{ \Q}_{\Lambda_n},\tilde{\q}_{\Lambda_n})_{n\in\IN}$, where $\tilde{ \Q}_{\Lambda_n}\in \mathcal P_s (\Gamma_{\Lambda_n}\times \Gamma_{\Lambda_n})$ is a coupling of its marginals $\gls{tildeP},\tilde{\mathsf R}_{\Lambda_n}$ and such that the following holds
\begin{enumerate}
\item \label{thm:mod_proc_same_number} The coupling $\tilde{ \Q}_{\Lambda_n}$ is concentrated on the set $
\{(\xi,\eta)\in \Gamma_{\Lambda_n}^2:\xi(\Lambda_n)=\eta(\Lambda_n)\}$ and moreover it holds $\tilde \P_{\Lambda_{n-1}}= \P_{\Lambda_{n-1}}$ as well as $\tilde \R_{\Lambda_{n-1}}=\R_{\Lambda_{n-1}}$. In particular, it follows that $\tilde\P_{\Lambda_n}$,$\tilde \R_{\Lambda_n}$ converge weakly to $\P,\R$, respectively, as $n\to\infty$.
\item \label{thm:mod_proc_entropy}
The relative entropies are well approximated in the sense that 
\begin{align}
    \mathsf{Ent}(\tilde{\P}_{\Lambda_n}\mid \Poi_{\Lambda_n})= \mathsf{Ent}(\P_{\Lambda_{n-1}}\mid \Poi_{\Lambda_{n-1}})+o(n^d),\\
     \mathsf{Ent}(\tilde{\R}_{\Lambda_n}\mid \Poi_{\Lambda_n})= \mathsf{Ent}(\R_{\Lambda_{n-1}}\mid \Poi_{\Lambda_{n-1}})+o(n^d).\nonumber
\end{align}
In particular it follows $\mathcal E(\P)=\lim_{n\to\infty}{n^{-d}}\mathsf{Ent}(\tilde{\P}_{\Lambda_n} \mid \Poi_{\Lambda_n})$ if it is finite.
\item \label{thm:mod_proc_cost}  The following cost estimate holds \begin{align}
\limsup_{n\to \infty}n^{-d}\IE_{\tilde{ \Q}_{\Lambda_n}}\left[\int_{\Lambda_n\times \Lambda_n}\abs{x-y}^2d\tilde{\q}_{\Lambda_n}(x,y)\right]
\leq \C_2(\P,\mathsf R).
\end{align}
\item \label{thm:mod_proc_intens}
The intensities of the stationarized processes converge: $
\lim_{n\to \infty}n^{-d}\IE_{(\tilde{\P}_{\Lambda_n})^{stat}}\left[\xi(\Lambda_n)\right]=1$.
\end{enumerate}
\end{thm}

\begin{proof}[Proof of part \eqref{thm:mod_proc_same_number}]
Fix $n\in \IN$ and for remainder of the proof we abbreviate $\P:=\P_{\Lambda_n}$. Define boxes $K_i:=z_i+\Lambda_{1/2}$ with disjoint interior such that $K:=\bigcup_{i=1}^N K_i=\overline{\Lambda_n\setminus \Lambda_{n-1}}$ (here, $\overline\cdot$ denotes the closure). Note that $N\le 4^dn^{d-1}$. We shall modify the configuration $\xi$ under $\P$ by removing all points in $K$ and then add $k_i$ many points to boxes $K_i$ such that $\xi$ and its optimally coupled $\eta$ both have exactly $l\in\IN$ points in $\Lambda_n$. 

To this end, let us keep track of those pairs of $\q$, which are matched from the outside $\Lambda_{n-1}^c$ to the inside $\Lambda_{n-1}$ (or vice versa) via a straight line through $K_i$, that is
\begin{align*}
V_i&:=\{(x,y)\in\supp \q\cap\Lambda_{n-1}^c\times \Lambda_{n-1} :\exists \tau\in [0,1 ]\text{ s.t. }\tau y+(1-\tau)x\in K_i\cap \Lambda_{n-1}\},\\
V_i'&:=\{(x,y)\in\supp \q\cap\Lambda_{n-1}\times \Lambda_{n-1}^c:\exists \tau\in [0,1 ]\text{ s.t. }\tau y+(1-\tau)x\in K_i\cap \Lambda_{n-1}\}.
\end{align*}
The number of points to be added to $\xi$ or $\eta$, respectively is
\begin{align*}
    k_i:=\q \big(V_i\big),\quad k_i':=\q \big(V_i'\big)
\end{align*}
and the total number of points after modification will be given by
\begin{align*}
    l:=\q (\Lambda_{n-1}\times \Lambda_{n-1})+\q (\Lambda_{n-1}\times\Lambda_{n-1}^c)+\q (\Lambda_{n-1}^c\times\Lambda_{n-1}),
\end{align*}
both of which are functions of $(\xi,\eta)$.
For fixed $k=(k_1,\dots,k_N),(k'_1,\dots,k'_N)\in\IN_0 ^{N}$ and $l\in\IN$, we partition the support of $\Q$ into the disjoint events 
\begin{align}\label{eq:A_lk}
    A_{l,k,k'}:=\{(\xi,\eta)\in \Gamma_{\IR^d}\times\Gamma_{\IR^d}: l \text{ and } k \text{ are defined as above}\}
\end{align}
and for the restriction we write $\Q_{l,k,k'}:=\Q(\cdot\cap A_{l,k,k'})$. Define the function that adds points by
\begin{align*}    F_{l,k,k'}&:A_{l,k,k'}\times \prod_{i=1}^N \Gamma^{k_i}_{K_i}\times \prod_{i=1}^N \Gamma^{k_i'}_{K_i}\to \Gamma_{\Lambda_n}^2,\\
    F_{l,k,k'}(\xi,\eta,\beta,\beta')&:=\Big(\xi|_{\Lambda_{n-1}}+\sum_{i=1}^N\beta_i,\eta|_{\Lambda_{n-1}}+\sum_{i=1}^N\beta_i'\Big).
\end{align*}
Recall that we write $\Gamma^{k_i'}_{K_i}$ for configurations in $K_i$ having exactly $k_i$ points. 
Then, adding $k_i$ independent points $Z_{i,j}$ with uniform distribution $\mathcal U _{K_i}$ corresponds to binomial point processes $\beta_i^\bullet=\sum_{j=1}^{k_i}\delta_{Z_{i,j}}\sim \mathsf {Bin}_{K_i,k_i}=\Poi_{K_i}(\cdot\mid \xi(K_i)=k_i)$ and the modified joint distribution of point processes is given by the decomposition
\begin{align}\label{eq:Q_Modification}
\tilde \Q:=\sum_{l,k,k'}\tilde \Q_{l,k,k'}:=\sum_{l,k,k'}(F_{l,k,k'})_\#\Big(\Q_{l,k,k'}\otimes\bigotimes_{i=1}^N\mathsf {Bin}_{K_i,k_i}\otimes\bigotimes_{i=1}^N\mathsf {Bin}_{K_i,k'_i}\Big).
\end{align}
In particular, the marginals $\tilde\P=(\proj_1)_\#\tilde\Q$ and $\tilde\R=(\proj_2)_\#\tilde\Q$ determine the modified distribution of the point processes. Note that the modified distribution inherits the disjoint decomposition onto $B_{l,k}:=\{\xi\in\Gamma_{\Lambda_n}:\xi(\Lambda_n)=l, \xi(K_i)=k_i\forall i\le N\}=\proj_1(\cup_{k'}\mathrm{image}(F_{l,k,k'}))$ via $$\tilde{\P}=\sum_{l,k,k'}{\proj_1}_\#\tilde{\Q}_{l,k,k'}=\sum_{l,k}\tilde{\P}(\cdot\cap B_{l,k}).$$
The modified matching $\tilde\q$ is defined as follows: It keeps the matching between points in $\Lambda_{n-1}^2$ of $\q$ and assigns new pairs according to
\begin{align*}
    \tilde \q|_{\Lambda_{n-1}^c\times\Lambda_{n-1}}=\sum_{i=1}^N\sum_{(x,y)\in V_i}\sum_{j=1}^{k_i}\delta_{(Z_{i,j},y)},\quad \tilde \q|_{\Lambda_{n-1}\times\Lambda_{n-1}^c}=\sum_{i=1}^N\sum_{(x,y)\in V'_i}\sum_{j=1}^{k'_i}\delta_{(x,Z_{i,j})}.
\end{align*}
By construction, we adjusted the number of points so that
$   l=\tilde \q (\Lambda_n^2)=\xi(\Lambda_n)=\eta(\Lambda_n)$, $\tilde \Q$-a.s.. Moreover, since the added points $Z_{i,j}\in K_i$ lie outside of $\Lambda_{n-1}$, we obtain $\tilde \P|_{\Lambda_{n-1}}= \P|_{\Lambda_{n-1}}$ as well as $\tilde \R|_{\Lambda_{n-1}}=\R|_{\Lambda_{n-1}}$ as claimed.
\end{proof}

\begin{proof}[Proof of part \eqref{thm:mod_proc_entropy}]
Apply Proposition \ref{prop:disint_entr} to $T(\xi)=\xi|_{\Lambda_{n-1}}$ to obtain
\begin{align*}
 \mathsf{Ent}(\tilde{\P}_{\Lambda_n}\mid \Poi_{\Lambda_n})= & \mathsf{Ent}\big(\tilde{\P}_{\Lambda_{n-1}}\big\vert \Poi_{\Lambda_{n-1}}\big) \\
 + \int &\mathsf{Ent}\big(\tilde{\P}_{\Lambda_n}(\cdot\mid \{\cdot|_{\Lambda_{n-1}}=\tilde\xi\})\big\vert \Poi_{\Lambda_n}(\cdot\mid \{\cdot|_{\Lambda_{n-1}}=\tilde\xi\})\big) d\tilde\P_{\Lambda_{n-1}}(\tilde\xi)
\end{align*}
First, note that the first term equals $ \mathsf{Ent}\big({\P}_{\Lambda_{n-1}}\big\vert \Poi_{\Lambda_{n-1}}\big) $ since $\tilde \P|_{\Lambda_{n-1}}= \P|_{\Lambda_{n-1}}$ by the previous part \eqref{thm:mod_proc_same_number}. We turn to the second term.
For an event $C\subset \Gamma_{K}$ and a configuration $\tilde \xi\in \Gamma_{\Lambda_{n-1}}$ we set $C+\tilde \xi=\{\gamma \cup \xi\mid \gamma\in C\}$.
Regarding the second term, the complete independence property of the Poisson point process yields for $\tilde \P_{\Lambda_{n-1}}$ almost all $\tilde\xi$ and any event $C=C|_K+\tilde \xi\subseteq\Gamma_{\Lambda_n}$ that
$$\Poi_{\Lambda_n}(\xi\in C\mid \{\xi|_{\Lambda_{n-1}}=\tilde\xi\}) =\Poi(\xi_K\in C|_K)=\prod_{i=1}^N \Poi(\xi_{K_i}\in C|_{K_i}).$$
We claim that for fixed $\tilde \xi\in\Gamma_{\Lambda_{n-1}}$ it holds 
\begin{align}\label{eq:modified_density}
    \frac{d \tilde{\P}_{\Lambda_n}(\cdot\mid \{\cdot|_{\Lambda_{n-1}}=\tilde\xi\})}{d\Poi(\cdot|_K)}(\xi)=\tilde \P_{\Lambda_n}(B_{l,k})e^N\prod_i k_i!
\end{align} for all $\xi\in B_{l,k}$ and $\xi|_{\Lambda_{n-1}}=\tilde \xi$, where $B_{l,k}=\{\xi\in\Gamma_{\Lambda_n}:\xi(\Lambda_n)=l, \xi(K_i)=k_i\forall i\le N\}$ as above and $l=\sum_i k_i+\tilde\xi(\Lambda_{n-1})$.
Indeed, take any event $C\subseteq B_{l,k}\subseteq\Gamma_{\Lambda_n}$, then by the construction of $\tilde \P$, the point process $\proj_1(F_{l,k,k'}(\xi,\eta,Z,Z'))$ with fixed $\xi|_{\Lambda_{n-1}}=\tilde\xi$ is given by the sum of $\tilde \xi$ and binomial point processes in $K_i$. More precisely the regular conditional probability is given by
\begin{align*}
&\tilde{\P}_{\Lambda_n}(\xi\in C\mid \{\xi|_{\Lambda_{n-1}}=\tilde\xi\})\\
=&\sum_{k'}\Q_{l,k,k'}\otimes\bigotimes_{i=1}^N\mathsf {Bin}_{K_i,k_i}\otimes\bigotimes_{i=1}^N\mathsf {Bin}_{K_i,k_i'}\Big(\proj_1(F_{l,k}(\xi,\eta,\beta,\beta'))\in C\Big\vert\proj_1(F_{l,k}(\xi,\eta,\beta,\beta'))|_{\Lambda_{n-1}}=\tilde\xi\Big)\\
=&\sum_{k'}\Q_{l,k,k'}\otimes\bigotimes_{i=1}^N\mathsf {Bin}_{K_i,k_i}\otimes\bigotimes_{i=1}^N\mathsf {Bin}_{K_i,k_i'}\Big(\tilde \xi+\sum_{i=1}^N\beta_i\in C\Big)\\
=&\Q(\cup_{k'}A_{l,k,k'})\prod_{i=1}^N \mathsf {Bin}_{K_i,k_i}\big(\beta_i\in C|_{K_i}\big)\\
=&\tilde \P_{\Lambda_n}(B_{l,k})\Poi(\xi_K\in C|_K)e^N\prod_{i=1}^Nk_i!,
\end{align*}
where in the second step we simply plugged in the condition.
Now, we split the entropy into $B_{l,k}$, that is
\begin{align*}
    &\mathsf{Ent}\big(\tilde{\P}_{\Lambda_n}(\cdot\mid \{\cdot|_{\Lambda_{n-1}}=\tilde\xi\})\big\vert \Poi_{\Lambda_n}(\cdot\mid \{\cdot|_{\Lambda_{n-1}}=\tilde\xi\})\big)\\
    &=\sum_{l,k}\int_{B_{l,k}}\log\left(\frac{d\tilde{\P}_{\Lambda_n}(\cdot\mid \{\cdot|_{\Lambda_{n-1}}=\tilde\xi\})}{d\Poi_{\Lambda_n}(\cdot\mid \{\cdot|_{\Lambda_{n-1}}=\tilde\xi\})}(\xi)\right)d \tilde{\P}_{\Lambda_n}(\xi\mid \{\xi|_{\Lambda_{n-1}}=\tilde\xi\}) \\
        &=\sum_{l,k}\int_{B_{l,k}}\left(N+\sum_{i=1}^N\log(k_i!)+\log(\tilde \P_{\Lambda_n}(B_{l,k})) \right)d \tilde{\P}_{\Lambda_n}(\xi\mid \{\xi|_{\Lambda_{n-1}}=\tilde\xi\}).
\end{align*}

Therefore, we conclude
\begin{align*}
 &\mathsf{Ent}(\tilde{\P}_{\Lambda_n}\mid \Poi_{\Lambda_n})\\
 &=  \mathsf{Ent}\big({\P}_{\Lambda_{n-1}}\big\vert \Poi_{\Lambda_{n-1}}\big) 
 + \sum_{l,k}\int\int_{B_{l,k}}\left(N+\sum_{i=1}^N\log(k_i!)+\log(\tilde \P_{\Lambda_n}(B_{l,k})) \right)d \tilde{\P}_{\Lambda_n}(\xi\mid \{\xi|_{\Lambda_{n-1}}=\tilde\xi\})d\tilde\P_{\Lambda_{n-1}}(\tilde\xi)\\
 &= \mathsf{Ent}\big({\P}_{\Lambda_{n-1}}\big\vert \Poi_{\Lambda_{n-1}}\big) 
 + \sum_{l,k}\tilde{\P}_{\Lambda_n}( B_{l,k})\left(N+\sum_{i=1}^N\log(k_i!)+\log(\tilde \P_{\Lambda_n}(B_{l,k})) \right)
 \\
 &= \mathsf{Ent}\big({\P}_{\Lambda_{n-1}}\big\vert \Poi_{\Lambda_{n-1}}\big) 
 + \sum_{l,k,k'}\Q(A_{l,k,k'})\sum_{i=1}^N\log(k_i!)+\sum_{l,k}\tilde\P_{\Lambda_n}(B_{l,k})\log(\tilde \P_{\Lambda_n}(B_{l,k}))+o(n^d).
 \end{align*}
In order to prove the upper bound of \eqref{thm:mod_proc_entropy} (and we shall see below that this is sufficient by semi-continuity), we drop the third term that is negative. It remains to control the logarithmic moment of the number of points $\q(\Lambda_{n-1}^c\times \Lambda_{n-1})$ which are transported from the outside to the inside. Using $k_i!\le k_i^{k_i}$ and $k_i\le \sum_j k_j$, it follows
 \begin{align}
     \sum_{l,k,k'}\Q(A_{l,k,k'})\sum_{i=1}^N\log(k_i!)&\le \sum_{l,k,k'}\Q(A_{l,k,k'})\sum_{i=1}^Nk_i\log\Big(\sum_{j=1}^N k_j\Big)\nonumber\\
     &=\IE_{\Q}\Big(\q(\Lambda_{n-1}^c\times \Lambda_{n-1})\log\big(\q(\Lambda_{n-1}^c\times \Lambda_{n-1})\big)\Big).\label{eq:EQqlogq}
 \end{align}
It follows from the Markov inequality, as we have done in \eqref{eq:Markov_ineq}, that $n^{-d}\q(\Lambda_{n-1}^c\times \Lambda_{n-1})\to 0$ in probability.
Then, for $\Phi(x)=x\log(x)$, also $\Phi(n^{-d}\q(\Lambda_{n-1}^c\times \Lambda_{n-1}))\to 0$ in probability and Lemma \ref{lem:cost_estimate} implies that $n^{-d}\Phi(\q(\Lambda_{n-1}^c\times \Lambda_{n-1}))\to 0$ in probability. On the other hand, using $\q(\Lambda_{n-1}^c\times \Lambda_{n-1})\le \eta(\Lambda_n)$, 
$$
0\le n^{-d}\Phi(\q(\Lambda_{n-1}^c\times \Lambda_{n-1}))\le n^{-d}\eta(\Lambda_n)\log(\eta(\Lambda_n))= \Phi(n^{-d}\eta(\Lambda_n))+n^{-d}\eta(\Lambda_n)\log (n^d)
$$
and the latter term again converges to $0$ in $L^1(\Q)$ by Lemma \ref{lem:cost_estimate}.
Moreover, it follows from convexity of $\Phi$ and the spatial ergodic theorem that
\[
\Phi(n^{-d}\eta(\Lambda_n))\leq n^{-d}\sum_{z\in\IZ^d\cap \Lambda_n}\eta(\Lambda_1(z))\log(\eta(\Lambda_1(z)))\xrightarrow{n\to \infty} Y_0
\]
 in $L^1(\Q)$, where $Y_0$ is an integrable random variable (being this deep into a technical proof, the reader might ask for a small pun to cheer up and we thought $Y_0$). This follows from the integrability $\IE_{\Q}[\eta(\Lambda_1)\log(\eta(\Lambda_1))]$, which is a consequence of Lemma \ref{lem:unif_int}. Thus, $n^{-d}\Phi(\q(\Lambda_{n-1}^c\times \Lambda_{n-1}))$ converges to $0$ in $L^1(\Q)$, which is precisely the claimed asymptotic of \eqref{eq:EQqlogq}. 

 It remains to show equality, that is $\mathcal E(\P)=\lim_{n\to\infty}{n^{-d}}\mathsf{Ent}(\tilde{\P}_{\Lambda_n} \mid \Poi_{\Lambda_n})$. Part (1) implies that $\tilde\P_{\Lambda_n}$ converges weakly to $\P$. By Lemma \ref{lem:Georgii}, $\mathcal E$ is lower semi-continuous, hence we combine our previous considerations to
\begin{align*}
    \mathcal E(\P)&\le \liminf_{n\to\infty}{n^{-d}}\mathsf{Ent}(\tilde{\P}_{\Lambda_n} \mid \Poi_{\Lambda_n})\\
    &=  \liminf_{n\to\infty}{n^{-d}}\mathsf{Ent}\big({\P}_{\Lambda_{n-1}}\big\vert \Poi_{\Lambda_{n-1}}\big) 
 + \liminf_{n\to\infty}{n^{-d}}\sum_{l,k}\tilde\P_{\Lambda_n}(B_{l,k})\log(\tilde \P_{\Lambda_n}(B_{l,k}))\le \mathcal E(\P),
\end{align*}
 which must be an equality.
\end{proof}

\begin{proof}[Proof of part \eqref{thm:mod_proc_cost}]
By construction of $\tilde{ \Q}_{\Lambda_n}$ from part (1) and $\tilde{\q}_{\Lambda_n}$ we have
\begin{align}
 &\IE_{\tilde{ \Q}_{\Lambda_n}}\int_{\Lambda_n\times \IR^d}\abs{x-y}^2d\tilde{\q}_{\Lambda_n}(x,y)\\
&\leq\IE_{ \Q} \int_{\Lambda_{n-1}^2}\abs{x-y}^2 d\mathsf{q}(x,y)+ \IE_{ \Q} \int_{\Lambda_{n-1}\times \Lambda_{n-1}^c}(\abs{x-y}+1)^2 d\mathsf{q}(x,y)+\IE_{ \Q} \int_{\Lambda_{n-1}^c\times \Lambda_{n-1}}(\abs{x-y}+1)^2 d\mathsf{q}(x,y).
\end{align}
Hence, it is sufficient to show that the last two terms are $o(n^d)$. By symmetry it is enough to treat one of them, that is
\[
 \IE_{ \Q} \int_{\Lambda_{n-1}\times \Lambda_{n-1}^c}(\abs{x-y}+1)^2 d\mathsf{q}(x,y) 
 \leq 2 \IE_{ \Q} \int_{\Lambda_{n-1}\times \Lambda_{n-1}^c}\abs{x-y}^2 d\mathsf{q}(x,y)+\IE_{ \Q}\mathsf{q}(\Lambda_{n-1}\times \Lambda_{n-1}^c).
\]
Both terms divided by $n^d$ vanish according to Lemma \ref{lem:cost_estimate}.
\end{proof}

\begin{proof}[Proof of part \eqref{thm:mod_proc_intens}]
This part follows directly from additivity and Lemma \ref{lem:cost_estimate}, more precisely
\begin{align*}
  \frac 1 {n^d}\IE_{(\tilde{\P}_{\Lambda_n})^{stat}}[\xi(\Lambda_n)] =\frac 1 {n^d}\IE_{\tilde{\P}_{\Lambda_n}}\left[\xi(\Lambda_n)\right]
=\frac 1 {n^d}\IE_{ \Q}\left[\mathsf \q(\Lambda_{n-1}\times \IR^d)+\mathsf \q(\Lambda_{n-1}^c\times \Lambda_{n-1}) \right]=1+o(1).
\end{align*}

\end{proof}

\section{Evolution Variational Inequality for stationary point processes}\label{sec:EVI}

The goal of this section is the following EVI.

\begin{thm}\label{thm:EVI}
Let $ \P,\mathsf R\in\spp1$  with $\W_2(\P,\mathsf R)<\infty$ and finite specific entropies, i.e. $\mathcal E(\P),\mathcal E(\mathsf R)<\infty$. Then the EVI holds
\begin{align}\label{eq:EVI_thm}
\W^2_2(\Semi_t \P,\mathsf R)-\W^2_2(\P,\mathsf R)\leq  2t(\mathcal{E}(\mathsf R)-\mathcal E(\Semi_t\P)),\quad \forall t\geq 0.
\end{align}
\end{thm}

Recall that the specific Wasserstein distance $\W_2$ has been defined in \eqref{eq:Wp} and the specific entropy has been defined in \eqref{eq:specific_rel_ent}. The proof is divided among the remaining subsections and organized as follows. First, we will leverage the EVI on finite dimensional Euclidean space first to random number of points. Then, it will be lifted to the stationarized and modified process, defined in the previous section, which we will show to approximate EVI of the original process as in \eqref{eq:EVI_thm}.

\subsection{Fixed number of points on a box}
We fix $k\in \IN$ and the box $\Lambda_n$. Denote the set of all point configurations on $\Lambda_n$ with $k$ points by $\Gamma^k_{\Lambda_n}$.
For two such configurations $\xi,\eta\in \Gamma^k_{\Lambda_n}$ we define the cost \[
 c_{\Lambda_n}^k(\xi,\eta)=\inf_{\q\in\cpl(\xi,\eta)}\int_{\Lambda_n\times \Lambda_n}\abs{x-y}^2d\q,
\] 
 For two point processes $\P_0,\P_1$
  on $\Lambda_n$ with deterministic  number of points $k$ we then define the cost \[
 \gls{Ckn}( \P_0,\P_1)=\inf_{ \Q\in\Cpl(\P_0,\P_1)}\IE_{ \Q}\left[c^k_{\Lambda_n}(\xi,\eta)\right].
\]
We want to identify $\Gamma_{\Lambda_n}^k$ with a portion of $(\IR^d)^k$ in order to compute entropies and transport costs. To this end, we consider the lexicographic order on $\IR^d$ and write $x\leq y$ if $x=(x_1,\dots, x_d)$, $y=(y_1,\dots,y_d)$ satisfy $x_i< y_i$ for the first index $i=1,\dots,d$ with different entries.  Consider the orthant 
\[
O=\{(x^1,\dots,x^k)\in (\IR^d)^k\mid x^i\leq x^{i+1}\}\;.
\]
Now, define $\ell$ to be the label-map, which orders the points of a configuration lexicographically 
\[
\gls{l}:\Gamma_{\Lambda_n}^k \to O \subset (\IR^d)^k\;, \{x^1,\dots, x^k\} \mapsto \ell (x^1,\dots,x^k)\;.
\]
By slight abuse of notation, we will also apply $\ell$ to vectors $(x^i)_{i\le k}\in(\IR^d)^k$.

\subsubsection{Entropy}

 Fix two distributions of point processes $\P_0,\P_1\in\mathcal P(\Gamma_{\Lambda_n}^k)$ on $\Lambda_n$ with deterministic  number of points $k$. Assume that $\C_{\Lambda_n}^k(\P_0,\P_1)<\infty$ and let $(\Q,\q=\q(\xi,\eta))$ be an optimal pair for the unnormalized cost  $\C_{\Lambda_n}^k$. Note that existence follows from classical results on $(\IR^d)^k$, see for instance \cite[Corollary 5.22]{Villani}.

 From the Birkhoff theorem for doubly stochastic matrices it follows that we can assume that it is a matching $\q(\xi,\eta)=(\id,T^{\xi,\eta})_{\#}\xi$ for some map $T^{\xi,\eta}$ s.t. $T^{\xi,\eta}_{\#}\xi=\eta$.
Denote by $\P_t$ the displacement interpolation at time $a$ induced by $ \Q$ and $\q$. Let  $\gls{pi}:\Gamma_{\Lambda_n}\to\IN_0, \xi \mapsto \xi(\Lambda_n)$. The aim of this subsection is to prove 
\begin{prop}\label{prop:conv_entropy_box}
The relative entropy $\ent(\P_t\mid \Poi_{\Lambda_n}(\cdot\mid \pi=k))$ is convex on $[0,1]$, where $\Poi_{\Lambda_n}(\cdot\mid \pi=k)$ denotes Poisson point process on $\Lambda_n$ conditioned to have $k$ points.
\end{prop}

Recall that, similar to Theorem \ref{thm:Modification}, the distribution $\Poi_{\Lambda_n}(\cdot\mid \pi=k)=\mathsf {Bin}_{\Lambda_n,k}$ is a binomial process. However, having in mind that the Poisson process will take the role of a reference process of our the specific relative entropy, it is instructive to think of a conditioned Poisson process instead of a binomial process.
 
 Define the map $F:\Gamma_{\Lambda_n}^k\times\Gamma_{\Lambda_n}^k\to (\IR^d)^k \times (\IR^d)^k$ by \begin{align}\label{eq:label_pairs}
(\xi,\eta)\mapsto (\ell (\xi),T^{\xi,\eta}(\ell (\xi))),
\end{align}
where $T^{\xi,\eta}$ is applied component wise.  Denote the displacement interpolation w.r.t. $F_{\#} \Q$ by $(\mu_t)_{t\in [0,1]}$.

 We claim that $F_{\#} \Q$ is an optimal coupling w.r.t. the Wasserstein $L^2$ distance $W_2$, since  its support is cyclically monotone. Indeed, let $(\xi,\eta), (\tilde{\xi},\tilde{\eta})\in \supp{ \Q}$. The proof is essentially the same for a general number of pairs in  $\supp{ \Q}$. Then\begin{align}\label{calculation:optimal}
&\sum_{i=1}^k\norm{x^i-T^{\tilde{\xi},\tilde{\eta}}(\tilde{x}^i)}^p+
\norm{\tilde{x}^i-T^{\xi,\eta}(x^i)}^p\\
&\geq c_{\Lambda_n}^k(\xi,\tilde{\eta})+c_{\Lambda_n}^k(\tilde{\xi},\eta)\nonumber\\ 
&\geq c_{\Lambda_n}^k(\xi,\eta)+c_{\Lambda_n}^k(\tilde{\xi},\tilde{\eta})\nonumber\\
&=\sum_{i=1}^k\norm{x^i-T^{\xi,\eta}(x^i)}^p+
\norm{\tilde{x}^i-T^{\tilde{\xi},\tilde{\eta}}(\tilde{x}^i)}^p,\nonumber
\end{align} 
where we used that the support of $ \Q$ is cyclically monotone. 
In particular, we obtain \begin{align}\label{eq:calc_wasser}
\C^k_{\Lambda_n}(\P_0,\P_1)&=\int c^k_{\Lambda_n}(\xi,\eta)\Q(d(\xi,\eta))\\
&=\int \sum_{x\in \xi}\norm{x-T^{\xi,\eta}(x)}^p\Q(d(\xi,\eta))\nonumber\\
&=\int \norm{\ell (\xi)-T^{\xi,\eta}(\ell (\xi))}^p\Q(d(\xi,\eta))\nonumber\\
&=\int \norm{x-y}^pF_{\#} \Q(dx,dy)\nonumber\\
&=\mathsf{W}_p^p(\mu_0,\mu_1).
\end{align}

Since $F_{\#} \Q$ is optimal, the relative entropy is convex along the geodesic $(\mu_t)_t$
induced by $F_{\#} \Q$. In the following we will consider permutations $\sigma:(\IR^d)^k\to (\IR^d)^k$, $(x^1,\dots,x^k)\mapsto (x^{\sigma_1},\dots,x^{\sigma_k})$. The set of these permutations will be called $\Sigma$. 
\begin{lem}\label{lem:permut}
Let $(\xi,\eta),(\tilde{\xi},\tilde{\eta})\in \supp{ \Q}$ and assume that there exists a permutation $\sigma$ and $0<t<1$ s.t. \[
\sigma\left((1-t)\ell (\xi)+tT^{\xi,\eta}(\ell (\xi))\right)
=(1-t)\ell (\tilde{\xi})+tT^{\tilde{\xi},\tilde{\eta}}(\ell (\tilde{\xi})).
\]Then $(\xi,\eta)=(\tilde{\xi},\tilde{\eta})$.
\end{lem}

\begin{proof}
We follow the lines of  \cite[Lemma 4.23]{Santa}. Denote the interpolation curve between the vectors $\sigma\left(\ell (\xi)\right)$ and $\sigma\left(T^{\xi,\eta}(\ell (\xi))\right)$ by $(e_s)_{s\in [0,1]}$.  For the pair $(\ell (\tilde{\xi}),T^{\tilde{\xi},\tilde{\eta}}(\ell (\tilde{\xi})))$ define similarly the curve $(\tilde{e}_s)$. By assumption $e_t=\tilde{e}_t$.  By convexity of $\norm{\cdot}^2$, the new pairs 
$(\sigma\left(\ell (\xi)\right),T^{\tilde{\xi},\tilde{\eta}}(\ell (\tilde{\xi})))$ and 
$(\ell (\tilde{\xi}),\sigma\left(T^{\xi,\eta}(\ell (\xi))\right))$
yield a lower cost. This however contradicts the optimality (or, cyclic monotonicity) of $ \Q$ since this would imply \begin{align*}
&c_{\Lambda_n}^k(\xi,\eta)+c_{\Lambda_n}^k(\tilde{\xi},\tilde{\eta})\\
&=\lVert\sigma\left(\ell (\xi)\right)-\sigma\left(T^{\xi,\eta}(\ell (\xi))\right)\rVert^2+\lVert \ell (\tilde{\xi})-T^{\tilde{\xi},\tilde{\eta}}(\ell (\tilde{\xi}))\rVert^2\\
&>\lVert\sigma\left(\ell (\xi)\right)-T^{\tilde{\xi},\tilde{\eta}}(\ell (\tilde{\xi}))\rVert^2+
\lVert \ell (\tilde{\xi})-\sigma\left(T^{\xi,\eta}(\ell (\xi))\right)\rVert^2\\
&\geq c_{\Lambda_n}^k(\xi,\tilde{\eta})+c_{\Lambda_n}^k(\tilde{\xi},\eta).
\end{align*}
\end{proof}

Next define for $\sigma\in \Sigma$ the sets
\[
A_{\sigma}=\{x\in (\IR^d)^k\mid \sigma(x)\in O\}
\]
and note they are disjoint except for their boundaries. Let us remove this ambiguity by setting
\[
B_1=A_{\sigma^1},\quad B_{i+1}=A_{\sigma^{i+1}}\setminus \cup_{k=1}^iB_k,
\]
where $\sigma^1,\dots,\sigma^m$ is some enumeration of $\Sigma$. Since the sets $B_k$ are disjoint we obtain 
\[
\ent(\mu_t| \Leb)=\sum_{i=1}^m \ent(\eins_{B_i}\mu_t| \Leb)=\sum_{i=1}^m\ent(\ell_{\#}(\eins_{B_i}\mu_t)|\Leb),
\]
where we used that $\ell$ coincides with $\sigma^i$ on $B_i\subset A_{\sigma^i}$ and that $\sigma^i:(\IR^d)^k\to (\IR^d)^k$ (and its inverse) preserves the Lebesgue measure. If we can show that the supports of the measures $\ell_{\#}(\eins_{B_k}\mu_t)$ are disjoint, then it would follow that 
\[
\ent(\mu_t|\Leb)=\sum_{i=1}^m\ent(\ell_{\#}(\eins_{B_i}\mu_t)|\Leb)=\ent(\ell_{\#}\mu_t|\Leb).
\] 
\begin{lem}\label{lem:finite_box_ent_iden}
The supports of the measures $\ell_{\#}(\eins_{B_k}\mu_t)$ are disjoint for $0<t<1$. Hence 
 \[
\ent(\mu_t|\Leb)=\ent(\ell_{\#}\mu_t|\Leb)=\ent( \P_t|\Poi_{\Lambda_n}(\cdot\mid \pi=k))-\log(\abs{O}).
\] 
\end{lem}
\begin{proof}
Assume the supports are not disjoint.
Then there exist $i<j$ and $(\xi,\eta),(\tilde{\xi},\tilde{\eta})\in \supp{ \Q}$ such that \[
z_t=(1-t)\ell (\xi)+tT^{\xi,\eta}(\ell (\xi))\in B_i \text{ and }
\tilde{z}_t=(1-t)\ell (\tilde{\xi})+tT^{\tilde{\xi},\tilde{\eta}}(\ell (\tilde{\xi}))\in B_j
\] and 
\[
\ell (z_t)
=\ell (\tilde{z}_t).
\]
This means there exists a permutation $\sigma$ such that 
\[
\sigma(z_t)=\tilde{z}_t.
\]
Hence by  Lemma \ref{lem:permut} we know that $(\xi,\eta)=(\tilde{\xi},\tilde{\eta})$ and thus $z_t=\tilde{z}_t$. Since the sets $(B_n)_{n\geq 1}$ are disjoint this implies $i=j$, which is a contradiction.
Finally we conclude 
\[
\ent(\mu_t|\Leb)=\sum_{i=1}^m\ent(\ell_{\#}(\eins_{B_i}\mu_t)|\Leb)=\ent(\ell_{\#}\mu_t|\Leb)=\ent( \P_t| \Poi_{\Lambda_n}(\cdot\mid \pi=k) )-\log(\abs{O}).\]
(The last $\log$ term appears because we consider the not normalised Lebesgue measure.)
\end{proof}
Note that we also have \[
\ent(\mu_0|\Leb)=\ent(\ell_{\#}\mu_0|\Leb)=\ent( \P_0|\Poi_{\Lambda_n}(\cdot\mid \pi=k))-\log(\abs{O}),
\]
 since $\mu_0=\ell_{\#}\mu_0$.
With the above results we can now prove Proposition \ref{prop:conv_entropy_box}.

\begin{proof}[Proof of Proposition \ref{prop:conv_entropy_box}]
 We start by  proving $\ent(\mu_1|\Leb)<\infty$.
It still holds  that \begin{align*}
\ent(\mu_1|\Leb)=\sum_{k=1}^m\ent(\ell_{\#}(\eins_{B_k}\mu_1)|\Leb).
\end{align*}
The entropy $\ent(\ell_{\#}(\eins_{B_i}\mu_1)|\Leb)$ can be bounded in the following way. Let $f_i$ be the density of $\ell_{\#}(\eins_{B_i}\mu_1)$ and $f$ be the density of 
$\ell_{\#}(\mu_1)$  w.r.t. Lebesgue. Write $\Phi(x)=x\log x$. Then $f_i\leq f$ and hence\begin{align*}
\int_{\Lambda_{n}^k}  \Phi(f_i(x))dx&\leq \int_{\Lambda_{n}^k}  \Phi(f(x))\eins_{f_i(x)\geq 1}dx+\int_{\Lambda_{n}^k}  \Phi(f_i(x))\eins_{f_i(x)\leq 1}dx\\
&\lesssim \int_{\Lambda_{n}^k}  \Phi(f(x))dx+\int_{\Lambda_{n}^k} dx,
\end{align*} 
where we used that the function $\Phi$ is monotone increasing on $[1,\infty)$ and bounded on $[0,1]$ with $\Phi(0)=0$. Note that the support of $f_i$ lies in $\Lambda_n^k$.  Since \[
\int  \Phi(f(x))dx=\ent(\ell_{\#}(\mu_1)|\Leb)=\ent(\P_1|\Poi_{\Lambda_n}(\cdot\mid \pi=k))-\log(\abs{O})<\infty,
\] it follows that $\ent(\mu_1|\Leb)<\infty$.
 Since $(\mu_t)_{t\in [0,1]}$ is the displacement interpoltion of an optimal coupling, it follows that $\ent(\mu_t|\Leb)$ is a  convex function for $t\in [0,1]$. Furthermore, we showed that $\ent(\mu_t|\Leb)$ is finite for all $t\in [0,1]$.

 For $0\leq t<1$ we know by Lemma \ref{lem:finite_box_ent_iden} that $\ent(\P_t|\Poi_{\Lambda_n}(\cdot\mid \pi=k))=\ent(\mu_t|\Leb)+\log(\abs{O})$. In order to prove that $\ent(\P_1|\Poi_{\Lambda_n}(\cdot\mid \pi=k))=\ent(\mu_1|\Leb)+\log(\abs{O})$ we  apply the exact same argument to the processes $\P_0'=\P_1$ and $\P_1'=\P_0$. We then obtain a convex function $\ent(\mu'_t|\Leb)$, $0\leq t\leq 1$ with $\ent(\mu'_t|\Leb)=\ent(\P'_t|\Poi_{\Lambda_n}(\cdot\mid \pi=k) )-\log(\abs{O})=\ent(\P_{1-t}|\Poi_{\Lambda_n}(\cdot\mid \pi=k))-\log(\abs{O})$ for $0\leq t<1$. Since $\ent(\mu_t|\Leb)=\ent(\mu'_{1-t}|\Leb)$ for $0<t<1$ they coincide also on the boundary. Hence \begin{align}\label{eq:identification_entropy_t1}
\ent(\P_1|\Poi_{\Lambda_n}(\cdot\mid \pi=k))=\ent(\mu'_0|\Leb)+\log(\abs{O})=\ent(\mu_1|\Leb)+\log(\abs{O}).
\end{align}
Finally we have that $\ent(\P_t|\Poi_{\Lambda_n}(\cdot\mid \pi=k))=\ent(\mu_t|\Leb)+\log(\abs{O})$ is convex on $[0,1]$.
\end{proof}
\subsubsection{EVI}

We now look at the identification of point processes and measures on $(\IR^d)^k$ for the EVI.

Let $\P_0$ and $\mathsf R$ be  two point processes on $\Lambda_n$ with a deterministic number of points $k\in \IN$ and finite cost. 
Let $(\Q,\q)$ be an optimal pair for $\P_0$ and $\mathsf R$. As before, write $\q(\xi,\eta)=(\mathsf{Id},T^{\xi,\eta})_{\#}\xi$. Combining \eqref{eq:def_GF_box} and \eqref{eq:label_pairs}, define the map  
\begin{align}\label{eq:bigF}
F_t:\Gamma_{\Lambda_n}^k\times\Gamma_{\Lambda_n}^k\times \mathcal C\big([0,\infty),(\IR^d)^k)\big)&\to (\IR^d)^k \times (\IR^d)^k \times (\IR^d)^k\nonumber\\
(\xi,\eta,(\omega^i)_{i=1,\dots,k})&\mapsto (\ell (\xi),\refl_{\Lambda_n}(\ell (\xi)+\omega)_t,T^{\xi,\eta}(\ell (\xi)))\;,
\end{align}
where $T^{\xi,\eta}$ and $\refl_{\Lambda_n}$ are applied component wise. 
See \eqref{def:reflectedpath} for the definition of the reflected path $\refl_{\Lambda_n}(\ell (\xi)+\omega)_t$.
On the space $\Gamma_{\Lambda_n}^k\times\Gamma_{\Lambda_n}^k\times \mathcal C([0,\infty),(\IR^d)^k)$ we consider the measure $  \Q\otimes(\otimes_{i=1}^k  \mathbb{W})$, where $\mathbb{W}$ is the standard Wiener measure.  We denote the marginals of $F_{\#} ( \Q\otimes(\otimes_{i=1}^k \mathbb{W}))$ by $\mu_0$, $\mu_t$ and $\nu$.
Note that the pushforward of $\mu_t$ under the delabeling map $(x^1,\dots,x^k)\in (\IR^d)^k\mapsto \{x^1,\dots,x^k\}\in \Gamma^k_{\Lambda_n}$
is equal to $\Semi_t^{\Lambda_n}\P_0$,  for the semigroup $\Semi_t^{\Lambda_n}$ defined in \eqref{def_semi_box}.
We get the following EVI
\begin{lem}\label{eq:EVI_fixed_everything}
For every $t\geq 0$ we have 
\begin{align*}
\C_{\Lambda_n}^k( \Semi_t^{\Lambda_n}\P,\mathsf R)-\C_{\Lambda_n}^k(\P_0,\mathsf R)\leq 2t\big[\ent(\mathsf R|\Poi_{\Lambda_n}(\cdot\mid \pi=k))-\ent(\Semi_t^{\Lambda_n}\P|\Poi_{\Lambda_n}(\cdot\mid \pi=k))\big]\;
\end{align*}
\end{lem}

\begin{proof}
 We saw in \eqref{calculation:optimal} that $(\pr_{1,3})_{\#}F_t{_{\#}} ( \Q\otimes(\otimes_{i=1}^k \mathbb{W}))$ is an optimal coupling of $\mu_0$ and $\nu$ w.r.t. the Wasserstein $L^2$ distance $W_2$ (not to be confused with the normalized distance $\mathsf W_2$ of stationary distributions in \eqref{eq:Wp}).
 By construction  $\mu_t$ is the solution to the heat equation on $(\Lambda_n)^k$ with Neumann boundary conditions started in $\mu_0$. Hence the following EVI holds by \cite[Proposition 8.10]{Santa}, i.e.
 \begin{align}\label{eq:EVI_classic}
{W}^2_2(\mu_t,\nu)-{W}^2_2(\mu_0,\nu)\leq 2t \big[\ent(\nu|\Leb)-\ent(\mu_t|\Leb)\big]\;.
\end{align}
The following properties have already been shown in \eqref{eq:calc_wasser} and \eqref{eq:identification_entropy_t1}:
\begin{itemize}
\item ${W}^2_2(\mu_0,\nu)=\C_{\Lambda_n}^k(\P_0,\mathsf R)$,
\item  $\ent(\nu|\Leb)=\ent(\mathsf R|\Poi_{\Lambda_n}(\cdot\mid \pi=k))-\log(\abs{O})$.
\end{itemize}
From Proposition \ref{prop:disint_entr} we infer that $\ent(\mu_t|\Leb)\geq\ent(\Semi_t^{\Lambda_n} \P|\Poi_{\Lambda_n}(\cdot\mid \pi=k))-\log(\abs{O})$, since $\Semi_t^{\Lambda_n} \P$ is the image of $\mu_t$ under the delabeling map. For the same reason,  every  coupling between $\mu_t$ and $\nu$ induces a coupling of $\Semi_t^{\Lambda_n} \P$ and $\mathsf R$ with the same cost it holds that $\C_{\Lambda_n}^k(\Semi_t^{\Lambda_n}\P,\mathsf R)\leq {W}^2_2(\mu_t,\nu)$.
 Plugging these equalities and inequalities into \eqref{eq:EVI_classic} yields the claim.
\end{proof}

\begin{cor}\label{cor:decr_entr}
Let $\P$ be a point process on $\Lambda_n$ with $k\in \IN$ points. Then the function\begin{align}
       [0,\infty) \ni t\mapsto \mathsf{Ent}\left(\Semi_t^{\Lambda_n}\P \mid \mathsf{\Poi_{\Lambda_n}}(\cdot \mid \pi=k) \right)
\end{align}
is decreasing.
\end{cor}
\begin{proof}
This follows imediately by applying Lemma \ref{eq:EVI_fixed_everything} with $\mathsf P_0$ and $\mathsf R$ both replaced by $\Semi_s^{\Lambda_n} P$ and noting that $(\Semi_s^{\Lambda_n}\P)_t= \Semi_{s+t}^{\Lambda_n} \P$ by Lemma \ref{lem:semigroup}.
See also \cite[Proposition 3.1]{Daneri_2008} for a general statement.
\end{proof}

\subsection{Random number of points in a box}
So far we treated point processes with a fixed number of points.  In the following, the number of points will be random.

Let $ (\Q,\q)$ be an optimal pair of two point processes $\P_0$ and $\mathsf R$ on $\Lambda_n$ in the following sense.
We assume \begin{enumerate}
\item For $(\xi,\eta)\in \supp( \Q)$ it holds $\xi(\Lambda_n)=\eta(\Lambda_n)$
\item Conditioned on the sets $\{(\xi,\eta):\xi(\Lambda_n)=\eta(\Lambda_n)=k\}$ the pair $ (\Q,\q)$ is optimal w.r.t. $\C_{\Lambda_n}^k$.
\end{enumerate}

In the subsequent sections, we aim to use the approach of this section for the modification $(\tilde{\mathsf Q}_{\Lambda_n},\tilde{\mathsf q}_{\Lambda_n})$, which was constructed in Theorem \ref{thm:Modification} precisely such that (1) holds. However, in general $(\tilde{\mathsf Q}_{\Lambda_n},\tilde{\mathsf q}_{\Lambda_n})$ is not optimal in the sense of assumption (2) (see also Theorem \ref{thm:Modification} (3)).

For fixed $k$ we write $\left(\Semi_t^{\Lambda_n}\P\right)^k=\left(\Semi_t^{\Lambda_n}\P\right)(\cdot\mid \pi=k)$, $\glsdisp{Pk}{\mathsf R^k}=\mathsf R(\cdot\mid \pi=k)$ and $p(k)=\P_0(\xi(\Lambda_n)=k)$.

\begin{lem}\label{lem:box_EVI}
For $t\in [0,1]$ we have \begin{align*}
\sum_{k=0}^{\infty}p(k)(\C_{\Lambda_n}^k(\left(\Semi_t^{\Lambda_n}\P\right)^k,\mathsf R^k)-\C_{\Lambda_n}^k(\P^k_0,\mathsf R^k))\leq
2t\big[\ent(\mathsf R\mid \Poi_{\Lambda_n})-\ent(\Semi_t^{\Lambda_n}\P\mid \Poi_{\Lambda_n})\big]\;.
\end{align*}
\end{lem}
\begin{proof}
First we treat the entropy terms and again we apply the disintegration formula w.r.t. the counting function $\pi$. Since the distribution of the number of points does not depend on $t$ we again obtain \[
\ent(\Semi_t^{\Lambda_n}\P\mid \Poi_{\Lambda_n})=\mathrm{const}+\sum_{k=0}^\infty\ent(\Semi_t^{\Lambda_n}\P(\cdot\mid \pi=k)\mid \Poi_{\Lambda_n}(\cdot\mid \pi=k))p(k)
\]
and a similiar expression for the process $\mathsf R$. Now fix $k$ and  consider $ \Q$ conditioned on the set $\{(\xi,\eta): \xi(\Lambda_n)=\eta(\Lambda_n)=k\}$. Lemma \ref{eq:EVI_fixed_everything} gives
\[
\C_{\Lambda_n}^k((\Semi_t^{\Lambda_n}\P)^k,\mathsf R^k)-\C_{\Lambda_n}^k(\P^k_0,\mathsf R^k)\leq 2t(\ent(\mathsf R^k\mid \Poi_{\Lambda_n}(\cdot\mid \pi=k))-\ent((\Semi_t^{\Lambda_n}\P)^k\mid \Poi_{\Lambda_n}(\cdot\mid \pi=k))),
\]
 and summing over $k$ yields
 \begin{align*}
&\sum_{k=0}^{\infty}p(k)(\C_{\Lambda_n}^k((\Semi_t^{\Lambda_n}\P)^k,\mathsf R^k)-\C_{\Lambda_n}^k(\P^k_0, \mathsf R^k))\\&\leq 2t \sum_{k=0}^{\infty}p(k)\left(\ent(\mathsf R^k\mid \Poi_{\Lambda_n}(\cdot\mid \pi=k))-\ent((\Semi_t^{\Lambda_n}\P)^k\mid \Poi_{\Lambda_n}(\cdot\mid \pi=k))\right)\\
&=2t\left(\ent(\mathsf R\mid \Poi_{\Lambda_n})-\ent(\Semi_t^{\Lambda_n}\P\mid \Poi_{\Lambda_n})\right).
\end{align*}
\end{proof}

\subsection{From point processes on a box to point processes on the whole space}\label{sec:from_box_to_space}

Let $\P,\mathsf R\in\spp1$  with $\W_2(\P_0,\mathsf R)<\infty$ and let $ (\Q,\q)$ be an optimal pair. For $n\in \IN$  apply Theorem \ref{thm:Modification} to obtain a (in general not optimal) pair $(\tilde{ \Q}_{\Lambda_n},\tilde{\q}_{\Lambda_n})$ on $\Lambda_n$ with marginals  $\gls{tildeP}$ and $\tilde{\mathsf R}_{\Lambda_n}$.

We now consider  processes defined on the whole space $\IR^d$.  Denote by $\stat{\Semi^{\Lambda_n}_t \tilde{\P}_{\Lambda_n}}$ the stationarized version of $\Semi^{\Lambda_n}_t \tilde{\P}_{\Lambda_n}$, defined in \eqref{def:stationatized_process}, and let  $\stat{\tilde{\mathsf R}_{\Lambda_n}}$ be the stationarized version of $\tilde{\mathsf R}_{\Lambda_n}$.

\begin{prop}\label{prop:cost_modified_processes}
The following inequality holds for all $t\geq 0$ \begin{align*}
\limsup_{n\to \infty}\W^2_2(\stat{\Semi^{\Lambda_n}_t \tilde{\P}_{\Lambda_n}},\stat{\tilde{\mathsf R}_{\Lambda_n}})-\W^2_2(\P_0,\mathsf R) \leq 2t\big(\mathcal{E}(\mathsf R)-\liminf_{n\to \infty}\mathcal{E}(\stat{\Semi^{\Lambda_n}_t \tilde{\P}_{\Lambda_n}})\big).
\end{align*}
\end{prop}
\begin{proof}
The coupling $\tilde{ \Q}_{\Lambda_n}$ is supported on the set $\{(\xi,\eta)\in \Gamma_{\Lambda_n}: \xi(\Lambda_n)=\eta(\Lambda_n)\}$. For every $k\in \IN$ consider $\tilde{ \Q}_{\Lambda_n}$, $\tilde{\q}_{\Lambda_n}$ conditioned on the set $A_k=\{(\xi,\eta)\in \Gamma_{\Lambda_n}: \xi(\Lambda_n)=\eta(\Lambda_n)=k\}$. This yields new ($k$ dependent) marginals. Now, pair these marginals in an optimal way (w.r.t. $\C_{\Lambda_n}^k$) and denote the optimal pairs by $(\hat{ \Q}_k,\hat{\q}_k)$. Define \[
\hat{ \Q}=\sum_{k=0}^{\infty}p(k)\hat{ \Q}_k \text{ and }
\hat{\q}=\sum_{k=0}^{\infty}\eins_{A_k}\hat{\q}_k,
\]
where $p(k)=\tilde{ \Q}_{\Lambda_n}(\xi(\Lambda_n)=\eta(\Lambda_n)=k)$. Applying Lemma \ref{lem:box_EVI} to the coupling $\hat{ \Q}$ yields
\begin{align}\label{eq:EVI_from}
\sum_{k=0}^{\infty}p(k)(\C_{\Lambda_n}^k((\Semi^{\Lambda_n}_t \tilde{\P}_{\Lambda_n})^k,\tilde{\mathsf R}^k_{\Lambda_n})-\C_{\Lambda_n}^k(\tilde{\P}^k_{\Lambda_n},\tilde{\mathsf R}^k_{\Lambda_n}))\leq 
2t\big[\ent(\tilde{\mathsf R}_{\Lambda_n}\mid \Poi_{\Lambda_n})-\ent(\Semi^{\Lambda_n}_t \tilde{\P}_{\Lambda_n}\mid \Poi_{\Lambda_n})\big],
\end{align}
where $(\Semi^{\Lambda_n}_t \tilde{\P}_{\Lambda_n})^k$ is $\tilde{\P}_{\Lambda_n}$ conditioned to have $k$ points and the same notation for $\tilde{\mathsf R}_{\Lambda_n}$. There exists an optimal process coupling $\q^{*}$ (on $\Lambda_n$) of the processes $\Semi^{\Lambda_n}_t \tilde{\P}_{\Lambda_n}$ and $\tilde{\mathsf R}_{\Lambda_n}$ with cost \[
\sum_{k=0}^{\infty}p(k)\C_{\Lambda_n}^k((\Semi^{\Lambda_n}_t \tilde{\P}_{\Lambda_n})^k,\tilde{\mathsf R}^k_{\Lambda_n}).
\]
Let $(\q^{*}_z)_{z\in \IZ^d}$ be iid copies of $\q^{*}$ and
 let $U$ be independent and uniformly distributed on $\Lambda_n$. Define  \[
\tilde{\q}_{t,n}=\theta_U\left(\sum_{z\in \IZ^d}\theta_{nz}\q^{*}_{nz} \right).
    \]
Let $\tilde{ \Q}_{t,n}$ be the distribution coupling induced by $\tilde{\q}_{t,n}$, see Remark \ref{rem:q_implies_Q}. Then  $\tilde{ \Q}_{t,n}$ is 
a coupling of the processes $\stat{\Semi^{\Lambda_n}_t \tilde{\P}_{\Lambda_n}}$ and $\stat{\tilde{\mathsf R}_{\Lambda_n}}$.
 Next we estimate the associated cost. Let $m_l=n+2ln$, then the box $\Lambda_{m_l}$ can be covered by exactly $(1+2l)^d$ of the $\Lambda_n$ boxes (the tiling of $\IR^d$ by these boxes with a box centered at $0$). If we shift the tiling by $z\in \Lambda_n$ then we need at most $(n+2(l+1)n)^d/n^d=(1+2l+2)^d$ boxes of this new tiling in order to cover $\Lambda_{m_l}$. Hence we can estimate \begin{align*}
\W^2_2(\stat{\Semi^{\Lambda_n}_t \tilde{\P}_{\Lambda_n}},\stat{\tilde{\mathsf R}_{\Lambda_n}})&\leq\liminf_{l\to \infty} m_l^{-d}\IE_{\tilde{ \Q}_{t,n}}\left[\int_{\Lambda_{m_l}\times\IR^d} \abs{x-y}^2d\tilde{\q}_{t,n}(x,y)\right]\\
&\leq \liminf_{l\to \infty} \frac{(n+2(l+1)n)^d}{(n+2ln)^dn^d}\sum_{k=0}^{\infty}p(k)\C^k_{\Lambda_n}((\Semi^{\Lambda_n}_t \tilde{\P}_{\Lambda_n})^k,\tilde{\mathsf R}^k_{\Lambda_n})\\
&=n^{-d}\sum_{k=0}^{\infty}p(k)\C_{\Lambda_n}^k((\Semi^{\Lambda_n}_t \tilde{\P}_{\Lambda_n})^k,\tilde{\mathsf R}^k_{\Lambda_n}).
\end{align*} 
Next we bound $\sum_{k=0}^{\infty}p(k)\C_{\Lambda_n}^k(\tilde{\P}^k_{\Lambda_n},\tilde{\mathsf R}^k_{\Lambda_n})$ from above using the  pair $(\tilde{\Q}_{\Lambda_n},\tilde{\q}_{\Lambda_n})$  of $\tilde{\P}_{\Lambda_n}$ and $\tilde{\mathsf R}_{\Lambda_n}$ mentioned at the very beginning, yielding 
\begin{align*}
&\sum_{k=0}^{\infty}p(k)\C_{\Lambda_n}^k(\tilde{\P}^k_{\Lambda_n},\tilde{\mathsf R}^k_{\Lambda_n})
\leq \IE_{\tilde{ \Q}_{\Lambda_n}}\left[\int_{\Lambda_n\times \Lambda_n}\abs{x-y}^2d\tilde{\q}_{\Lambda_n}(x,y)\right]\;.
\end{align*}
Taking the limes superior, the cost estimate \eqref{thm:mod_proc_cost} of Theorem \ref{thm:Modification} yields 
\[
\limsup_{n\to \infty}n^{-d}\sum_{k=0}^{\infty}p(k)\C_{\Lambda_n}^k(\tilde{\P}^k_{\Lambda_n},\tilde{\mathsf R}^k_{\Lambda_n})\leq \W_2^2(\P_0,\mathsf R).
\]

Combining our bounds on the left-hand side of the EVI \eqref{eq:EVI_from} yields the lower bound \begin{align*}
\limsup_{n\to \infty}n^{-d}\sum_{k=0}^{\infty}p(k)(\C_{\Lambda_n}^k((\Semi^{\Lambda_n}_t \tilde{\P}_{\Lambda_n})^k,\tilde{\mathsf R}^k_{\Lambda_n})-\C_{\Lambda_n}^k(\tilde{\P}^k_{\Lambda_n},\tilde{\mathsf R}^k_{\Lambda_n}))\\
\geq \limsup_{n\to \infty}\W_2^2(\stat{\Semi^{\Lambda_n}_t \tilde{\P}_{\Lambda_n}},\stat{\tilde{\mathsf R}_{\Lambda_n}})-\W_2^2(\P_0,\mathsf R)
\end{align*}
We turn to the right-hand side of the EVI. By construction of the modified process $\tilde{\mathsf R}_{\Lambda_n}$ \[
\limsup_{n\to\infty}n^{-d}\ent(\tilde{\mathsf R}_{\Lambda_n}\mid \Poi_{\Lambda_n})\leq \mathcal{E}(\mathsf R).
\]Furthermore $\mathcal{E}(\stat{\tilde{\P}_{\Lambda_n,t}}) \leq \ent(\tilde{\P}_{\Lambda_n,t}\mid \Poi_{\Lambda_n})$.
Combining this yields an upper bound on the right-hand side of the EVI\begin{align*}
\limsup_{n\to \infty}n^{-d}(\ent(\tilde{\mathsf R}_{\Lambda_n}\mid \Poi_{\Lambda_n})-\ent(\Semi^{\Lambda_n}_t \tilde{\P}_{\Lambda_n}\mid \Poi_{\Lambda_n}))
&\leq\mathcal{E}(\mathsf R)+\limsup_{n\to \infty}-\mathcal{E}(\stat{\Semi^{\Lambda_n}_t \tilde{\P}_{\Lambda_n}})\\
&=\mathcal{E}(\mathsf R)-\liminf_{n\to \infty}\mathcal{E}(\stat{\Semi^{\Lambda_n}_t \tilde{\P}_{\Lambda_n}}).
\end{align*}

\end{proof}

\subsection{From the modifications back to the original stationary point process}

Towards proving Theorem \ref{thm:EVI} and having Proposition \ref{prop:cost_modified_processes} in mind, it remains to compare the Gradient Flow Semigroup of the modified and stationarized process to the original one. First, we show weak convergence, which then implies that the EVI is well approximated  (that is, $\W_2$ and $\mathcal E$).

\begin{lem}\label{lem:weak_convergence}
We have weak convergence of the stationarized processes, both modified and original, i.e. for fixed $t\ge 0$ and $n\to\infty$
\begin{align*}
    \stat{\mathsf S_t^{\Lambda_n}\tilde{\P}_{\Lambda_n}}\to \Semi_t\P\\
    \stat{\mathsf S_t^{\Lambda_n}{\P}_{\Lambda_n}}\to \Semi_t\P.
\end{align*}
\end{lem}
\begin{proof}
    As in the proof of Lemma \ref{lem:charStat} it is sufficient to check for test functions $\phi\in\mathcal C_b(\mathcal M )$ of the form $\phi(\xi)=G\big[\big(\int f_i d\xi \big)_{i\le m}\big]$ for some fixed $m\in\IN$, $G\in\mathcal C_b(\IR^m)$ and $f_i\in\mathcal C_c(\IR^d)$. Therefore, we explicitly rephrase $\stat{\mathsf S_t^{\Lambda_n}\tilde{\P}_{\Lambda_n}}$ as a distribution of a tiled point process (see Section \ref{subsec:tiling}), uniformly shifted by $U\sim\mathcal U(\Lambda_n)$ (see \eqref{def:stationatized_process}), then pushed forward under the map $F_t$ from \eqref{eq:def_GF_box}, which sticks reflected Brownian motions to the particles, that is
\begin{align} 
&\int G\left[\Big(\int f_i(x)\xi(dx)\Big)_{i\le m}\right]  d\stat{\mathsf S_t^{\Lambda_n}\tilde{\P}_{\Lambda_n}}(\xi)\nonumber\\
&=\int\int\int G\left[\Big(\int f_i(\refl(x+\omega^x)_t+u)\xi(dx)\Big)_{i\le m}\right]d(\tilde{\P}_{\Lambda_n})^{til}(\xi)d\mathcal U(u)d\mathbb W^{\otimes \IN}(\omega),\label{eq:massiveintegral}
\end{align}
where we wrote $\omega^x$ for that coordinate $\omega^j$ of $\omega$ satisfying $\ell(\xi)_j=x$. 
We assume $\ell$ to be the map, which labels points on $\IR^d$ monotonously increasing, i.e.~if $i<j$, then $|\ell(\xi)_i|\le|\ell(\xi)_j|$, where we use the lexicographical order as a tie breaker.
Let $\epsilon<1/2$. We claim that, for $n\to\infty$, only $x\in \Lambda_{n-n^{\epsilon}}$ matter for the above integral and from which we will deduce that replacing $(\tilde{\P}_{\Lambda_n})^{til}$ by $({\P}_{\Lambda_n})^{til}$ or even $\P$ won't change the integral.
To this end, we first restrict ourselves to the high-probability event where $\xi$ has polynomially many points, most Brownian motions $\omega^i$ do not travel too far, and $U$ stays away from the boundary. Define
\begin{align*}
    A_1&=\{\xi\in\Gamma_{\IR^d}:\xi(\Lambda_n)\le n^{d+2}\}\;,\\
     A_2&=\{\omega\in\mathcal C([0,\infty],\IR^d)^{\IN}:|\omega_s^i|\le n^\epsilon \ \forall s\leq t\ \forall i\le n^{d+2}\}\;,\\
    A_3&=\{u\in\Lambda_{n-n^{1-\epsilon}}\}.
\end{align*}
Indeed, as $n\to\infty$, these are high-probability events. By Markov inequality we see that
\begin{align*}
    (\tilde{\P}_{\Lambda_n})^{til}(A_1^c)\le n^{-d-2}\IE(\xi(\Lambda_n))\to 0\;.
\end{align*}
Denoting by $\omega^i_{j,s}$ the $j$-th coordinate of $\omega^i_s$, a union bound, the distribution of the running maximum of 1D Brownian motion, and a Gaussian tail bound imply
\begin{align*}
\mathbb W^{\IN}(A_2^c)\le d n^{d+2}\mathbb W(\max_{s\leq t}|\omega_{1,s}^1|>n^{\epsilon}/\sqrt{d})\lesssim n^{d+2}e^{-n^{2\epsilon}/2d^\epsilon t}\to 0\;.
\end{align*}
Moreover, $\mathcal U(A_3)=(n-n^{1-\epsilon})^dn^{-d}\to 1$. Since $G$ is bounded, restricting the integral \eqref{eq:massiveintegral} to these events leaves the large $n$ limit unchanged. Note that on these events, we have for every $x\in\xi\cap \left(\Lambda_n\setminus\Lambda_{n-n^{\epsilon}}\right)$
\begin{align*}
    |\refl(x+\omega^x)_t+u|\gtrsim (n-n^\epsilon)-n^\epsilon -(n-n^{1-\epsilon})\to\infty
\end{align*}
and the same holds without reflecting the Brownian motions. 
In particular, since the $f_i$, $1\leq i\leq m$  have compact support, there exists an $N\in \IN$ such that for all $n\geq N$ \begin{align}\label{eq:xyz}
    \int f_i(\refl(x+\omega^x)_t+u)d\xi(x)=\int f_i(\refl(x+\omega^x)_t+u)d\xi|_{\Lambda_{n-n^\epsilon}}(x),\quad \forall (\xi,\omega,u)\in A_1\times A_2\times A_3. 
\end{align}
Therefore, we obtain
\begin{align*} 
&\int G\left[\Big(\int f_i(x)d\xi(x)\Big)_{i\le m}\right]  d\stat{\mathsf S_t^{\Lambda_n}\tilde{\P}_{\Lambda_n}}(\xi)\\
&=\int_{A_2}\int_{A_3}\int_{A_1} G\left[\Big(\int f_i(\refl(x+\omega^x)_t+u)d\xi(x)\Big)_{i\le m}\right]d(\tilde{\P}_{\Lambda_n})^{til}(\xi)d\mathcal U(u)d\mathbb W^{\otimes \IN}(\omega) +o(1)\\
&\stackrel{\eqref{eq:xyz}}{=}\int_{A_2}\int_{A_3}\int_{A_1} G\left[\Big(\int f_i(\refl(x+\omega^x)_t+u)d\xi|_{\Lambda_{n-n^\epsilon}}(x)\Big)_{i\le m}\right]d(\tilde{\P}_{\Lambda_n})^{til}(\xi)d\mathcal U(u)d\mathbb W^{\otimes \IN}(\omega) +o(1)\\
&=\int_{A_2}\int_{A_3}\int_{A_1} G\left[\Big(\int f_i(\refl(x+\omega^x)_t+u)d\xi|_{\Lambda_{n-n^\epsilon}}(x)\Big)_{i\le m}\right]d\P(\xi)d\mathcal U(u)d\mathbb W^{\otimes \IN}(\omega) +o(1)\\ 
&=\int_{A_2}\int_{A_3}\int_{A_1} G\left[\Big(\int f_i(x+\omega^x_t+u)d\xi|_{\Lambda_{n-n^\epsilon}}(x)\Big)_{i\le m}\right]d\P(\xi)d\mathcal U(u)d\mathbb W^{\otimes \IN}(\omega) +o(1),
\end{align*}
where the last line follows from the fact that 
\begin{align*}
     x+\omega^x_t \in \Lambda_n,\quad \forall x\in\xi|_{\Lambda_{n-n^\epsilon}},  (\xi,\omega,u)\in A_1\times A_2\times A_3.  
\end{align*}
It remains to argue that none of the infinitely many Brownian motions can enter $\Lambda_{n-n^\epsilon}$ from far away. 
By the Borel-Cantelli-Lemma, $A_1$ holds for all $n$ sufficiently large $\P$-a.s. Moreover, the event
\begin{align*}
    \tilde A_2=\{\forall k\in\IN,\forall i\le (nk)^{d+2}:|\omega_t^i|\le n(k-1)+n^\epsilon\}
\end{align*} 
also holds with high probability, since
\begin{align*}
    \mathbb W ^{\IN}(\tilde A_2^c)\le \sum_{k\in\IN} (nk)^{d+2} e^{-\tfrac{1}{2t}(n(k-1)+n^\epsilon)^2}\to 0
\end{align*}
as $n\to\infty$.
Hence by Borel-Cantelli, for fixed $\xi$ and $n$ sufficiently large (depending on $\xi$) we have (similar to the above): If $\omega\in \tilde A_2$ and $u\in A_3$, then for $x\in \xi$ with $x\notin \Lambda_{n-n^{\varepsilon}}$ it holds that $|x+\omega^x_t+u|\to\infty$ as $n\to\infty$ and hence, it is not contained in the compact support of the functions $f_i$, $1\le i\le m$. 
Ultimately by dominated convergence, it follows
\begin{align*}
    &\int_{A_2}\int_{A_3}\int_{A_1} G\left[\Big(\int f_i(x+\omega^x_t+u)d\xi|_{\Lambda_{n-n^\epsilon}}(x)\Big)_{i\le m}\right]d\P(\xi)d\mathcal U(u)d\mathbb W^{\otimes \IN}(\omega)\\
    &= \int_{\tilde A_2}\int_{A_3}\int_{A_1} G\left[\Big(\int f_i(x+\omega^x_t+u)d\xi(x)\Big)_{i\le m}\right]d\P(\xi)d\mathcal U(u)d\mathbb W^{\otimes \IN}(\omega)+o(1)\\
    &=\int G\left[\Big(\int f_i(x)d\xi(x)\Big)_{i\le m}\right]  d\mathsf S_t\P(\xi)+o(1),
\end{align*}
which proves the claim.
\end{proof}

Finally, we collected all ingredients to prove the EVI for stationary point processes.
\begin{proof}[Proof of Theorem \ref{thm:EVI}]
By Lemma \ref{lem:weak_convergence} we have $\stat{\Semi^{\Lambda_n}_t \tilde{\P}_{\Lambda_n}}\to \Semi_t\P$ weakly as $n\to \infty$ and by symmetry $\stat{\tilde{\R}_{\Lambda_n}}\to \R$. Lower semicontinuity of $\C_2=\W_2^2$ from Proposition \ref{prop:lsc_costfct+existence_min} implies
\begin{align*}
\W_2^2(\Semi_t\P,\mathsf R)&\leq  \liminf_{n\to \infty}\W_2^2(\stat{\Semi^{\Lambda_n}_t \tilde{\P}_{\Lambda_n}},\stat{\tilde{\mathsf R}_{\Lambda_n}})\\
&\leq  \limsup_{n\to \infty}\W_2^2(\stat{\Semi^{\Lambda_n}_t \tilde{\P}_{\Lambda_n}},\stat{\tilde{\mathsf R}_{\Lambda_n}}).
\end{align*} 
We apply Proposition \ref{prop:cost_modified_processes} and use the lower semi-continuity of the specific entropy of Lemma \ref{lem:Georgii}, to obtain the EVI
\begin{align*}
\W_2^2(\Semi_t\P,\mathsf R)-\W_2^2(\P_0,\mathsf R)&\leq \limsup_{n\to \infty}\W_2^2(\stat{\Semi^{\Lambda_n}_t \tilde{\P}_{\Lambda_n}},\stat{\tilde{\mathsf R}_{\Lambda_n}})-\W_2^2(\P_0,\mathsf R)\\&\leq 2t(\mathcal{E}(\mathsf R)-\liminf_{n\to \infty}\mathcal{E}(\stat{\Semi^{\Lambda_n}_t \tilde{\P}_{\Lambda_n}}))\\
&\leq 2t\left(\mathcal{E}(\mathsf R)-\mathcal E(\Semi_t\P)\right).
\end{align*}
\end{proof}

\section{Consequences}
\subsection{Convexity of the entropy}
We collect some consequences of Theorem \ref{thm:EVI}. 
Let $D(\mathcal E)=\{\P\in \spp1: \mathcal E(\P)<\infty\}$ be the proper domain of the specific entropy. Denote by $\overline{D(\mathcal E)}$ the closure of $D(\mathcal E)$ in $\spp1$ with respect to $\W_2$. The conclusion of Theorem \ref{thm:EVI} continues to hold for the extended domain of the entropy.

\begin{cor}\label{cor:EVI_closure}
For $\P,\mathsf R \in \overline{D(\mathcal E)}$ with $\W_2(\P,\mathsf R)<\infty$ the EVI holds, i.e. \begin{align}\label{eq:EVI_closure}
\W^2_2(\Semi_t\P,\mathsf R)-\W^2_2(\P,\mathsf R)\leq  2t(\mathcal{E}(\mathsf R)-\mathcal E(\Semi_t\P)),\quad \forall t\geq 0.
\end{align}
Furthermore, we have the contraction property \begin{align}
    \W_2(\Semi_t\P, \Semi_t\mathsf R)\leq \W_2(\P,\mathsf R).
\end{align}
\end{cor}
\begin{proof}
   Let $\P^{(n)}\xrightarrow{n\to \infty}\P$  in $\W_2$ with $\mathcal E(\P^{(n)})<\infty$ for all $n\in \IN$. Apply Theorem \ref{thm:EVI} to   $\P^{(n)}$ and $\P^{(1)}$ to obtain for $t>0$ 
  \[
  \W^2_2(\Semi_t \P^{(n)},\P^{(1)})-\W^2_2(\P^{(n)},\P^{(1)})\leq  2t(\mathcal{E}(\P^{(1)})-\mathcal E(\Semi_t \P^{(n)}).
  \]
  This is equivalent to \begin{align}\label{eq:consequences_EVI}
  \mathcal E(\Semi_t \P^{(n)}) \leq \mathcal{E}(\P^{(1)})
  +\frac{\W^2_2(\P^{(n)},\P^{(1)})- \W^2_2(\Semi_t \P^{(n)},\P^{(1)})}{2t}.
  \end{align}
  By letting $n\to \infty$ in \eqref{eq:consequences_EVI}, the lower semicontinuity of the specific entropy yields \[
  \mathcal E(\Semi_t\P)\leq 
   \mathcal{E}(\P^{(1)})
  +\frac{\W^2_2(\P,\P^{(1)})- \W^2_2(\Semi_t\P,\P^{(1)})}{2t}.
  \]
  Hence $\Semi_t\P$ has finite specific entropy. We can assume that $\mathcal{E}(\mathsf R)<\infty$, since otherwise the inequality \eqref{eq:EVI_closure} holds.
  Applying the EVI of Theorem \ref{thm:EVI} to $\Semi_s\P$ and $\mathsf R$ then yields for $t>0$ \[
  \W^2_2(\Semi_{s+t}\P,\mathsf R)-\W^2_2(\Semi_s\P,\mathsf R)\leq  2t(\mathcal{E}(\mathsf R)-\mathcal E(\Semi_{s+t}\P)).
  \]
  Letting $s\to 0$, the lower semicontinuity of the specific entropy and the continuity of $\Semi_t$ w.r.t. $\W_2$ from Lemma \ref{lem:semigroup} yield   
\[
  \W^2_2(\Semi_t\P,\mathsf R)-\W^2_2(\Semi_t\P,\mathsf R)\leq  2t(\mathcal{E}(\mathsf R)-\mathcal E(\Semi_t\P)).
  \]
  Now let $X$ be the set of all $\tilde{\mathsf R}\in \overline{D(\mathcal E)}$ with $W_2(\P,\tilde{\mathsf R})<\infty$ and let $d$ be the restriction of $\W_2$ to $X$. Since the EVI \eqref{eq:EVI_closure} holds, we can apply  \cite[Proposition 3.1]{Daneri_2008} to the metric space $(X,d)$ and the $\mathcal C^0$-semigroup $(\Semi_t)_{t\geq 0}$, which yields \[
  \W_2(\Semi_t\P,\Semi_t\mathsf R)\leq \W_2(\P,\mathsf R),\quad \forall t\geq 0.
  \]
\end{proof}

Secondly, the EVI in Theorem \ref{thm:EVI} implies convexity of the entropy along geodesics.
\begin{cor}\label{cor:convexity2}
    Let $\P_0,\P_1\in \spp1$ with $\W_2(\P_0,\P_1)<\infty$. For a geodesic $(\P_t)_{0\leq t\leq 1}$ we have
    \begin{align}\label{eq:convex}
 \mathcal{E}(\P_t)\leq (1-t)\mathcal E(\P_0)+t\mathcal E(\P_1),\quad \forall t\in [0,1].
    \end{align}In particular, the function $t \mapsto \mathcal E(\P_t)$ is convex on $[0,1]$.
\end{cor}
\begin{proof}
We can assume that $\P_0,\P_1 \in \overline{D(\mathcal E)}$, since otherwise the inequality holds.
   Let $X$ be the set of all  $\mathsf R \in \overline{D(\mathcal E)}$ such that $\W_2(\P_0,\mathsf R)<\infty$. Let $d$ be the restriction of $\mathsf W_2$ to $X$. 
   From Corollary \ref{cor:EVI_closure} it follows, that
    we can apply \cite[Proposition 3.1]{Daneri_2008} to the metric space $(X,d)$ and the $C^0$-semigroup  $(\mathsf S_t)_{t\geq 0}$. Hence, in the terminology of \cite{Daneri_2008},  the semigroup $\mathsf S_t$ is the $\lambda$-Flow, with $\lambda=0$, for the specific entropy $\mathcal{E}$. Finally,  \cite[Theorem 3.2]{Daneri_2008} proves \eqref{eq:convex}. 
    Since \eqref{eq:convex} holds for all geodesics, the second part of the corollary follows.
\end{proof}
Furthermore, we rediscover a natural statement: If we let the particles of an arbitrary stationary point process evolve like Brownian motions, then it will converge to a Poisson point process in the long run. Here, the convergence is stated in specific relative entropy, strengthening the classical result of \cite{St68} due to Pinsker's inequality.
\begin{cor}
    Let $\P\in \overline{D(\mathcal E)}$ with   $\W_2(\P,\mathsf{Poi})<\infty$. Then the function $(0,\infty)\ni t\mapsto \mathcal{E}(\Semi_t\P)$ is decreasing and 
    \[
    \mathcal{E}(\Semi_t\P)\xrightarrow{t\to\infty}0.
    \]
\end{cor}
\begin{proof}
    Let $X$ be the set of all  $\mathsf R\in \overline{D(\mathcal E)}$ with $\W_2(\mathsf R,\mathsf{Poi})<\infty$. Define $d$ as the restriction of $\mathsf W_2$ to $X$. Applying \cite[Proposition 3.1]{Daneri_2008} yields \[
    \mathcal{E}(\Semi_t\P)\leq \mathcal{E}(\mathsf{Poi})+\frac{1}{2t}d^2(\P,\mathsf R)=\frac{1}{2t}\mathsf W^2_2(\P,\mathsf R).
    \]Hence  \[
    \mathcal{E}(\Semi_t\P)\xrightarrow{t\to\infty}0.
    \] 
    Furthermore,  \cite[Proposition 3.1]{Daneri_2008} also implies 
   that the map \[
    t\mapsto \mathcal{E}(\Semi_t\P)
    \]is decreasing on $(0,\infty)$.
\end{proof}

\begin{ex}\label{ex:grid2}
    Let $d\geq 3$ and $(X_z)_{z\in \IZ^d}$ be iid random variables, which are uniformly distributed on $B_{\varepsilon}(0)$.  Let $Y$ be independent and uniformly distributed on $\Lambda_1$. Define the grid process by \[
    \mathsf{grid}=\theta_Y\left(\sum_{z\in \IZ^d}\delta_{z+X_z}\right).
    \]
       Since $d\geq 3$, it follows from \cite[Theorem 1.3]{HS13} that 
    the distance of the process $\Poi$ to $\Leb$ is finite, where $\Leb$ denotes the random measure a.s. equal to the Lebesgue measure.  
    Since an easy computation shows that $\W_2(\mathsf{grid},\Leb)<\infty$, the triangle inequality  yields
    $\W_2(\mathsf{grid},\mathsf{Poi})<\infty$. 
    Hence \[
    \mathcal{E}(\Semi_t(\mathsf{grid}))\xrightarrow{t\to \infty}0.
    \]
\end{ex}

\subsection{Specific Fisher information and HWI inequality}\label{sec:HWI}

Here we will introduce a notion of specific Fisher information for a stationary point process. In analogy to the celebrated HWI inequality relating entropy, Wasserstein distance and Fisher information, we will establish an inequality relating specific entropy, the transport cost $\W_2$, and the specific Fisher information.

Consider a function $V:\IR^d\to(-\infty,+\infty]$ which is lower semicontinuous and $\lambda$-convex and whose proper domain $\{x\in\IR^d:V(x)<+\infty\}$ has non-empty convex interior $\Omega$. An example is the convex indicator function $V=\eins_K$ of a convex set $K\subset\IR^d$, i.e. $V=0$ on $K$ and $V=+\infty$ on $\IR^d\setminus K$.

For a probability measures $\mu\in \mathcal{P}(\IR^d)$ we define the weighted Fisher information as follows. If $\mu=\rho\Leb$ s.t. $\rho\in W^{1,1}_{\sf loc}(\Omega)$ and there is $w\in L^2(\mu)$ such that $\rho w= \nabla \rho + \rho \nabla V$, we set  
\begin{equation}\label{eq:Fisher_Rd}
I(\mu|e^{-V}\Leb) := \int_{\IR^d}|w|^2d\mu\;.
\end{equation}
Otherwise, we set $I(\mu|e^{-V}\Leb)=+\infty$. 
In the case $V=i_K$, this can be rewritten as 
\[I(\mu|\Leb_K)=4\int_{K}|\nabla\sqrt{\rho}|^2d\Leb\;,\]
provided $\mu$ is supported in $K$ with $\mu=\rho\Leb_K$ and $\sqrt\rho\in W^{1,2}(K)$, and $I(\mu|\Leb_K)=+\infty$ otherwise.
Note that $I(\cdot|e^{-V}\Leb)$ is lower semicontinuous with respect to weak convergence, see \cite[Prop.~10.4.14]{AGS08}.
The classical HWI inequality now states:
\begin{thm}\label{thm:HWI-classic}
 Let $\mu_0,\mu_1\in P_2(\IR^d)$ be such that $\ent(\mu_0|e^{-V}\Leb)$, $\ent(\mu_1|e^{-V}\Leb)$, and $I(\mu_0|e^{-V}\Leb)$ are finite. Then
 \begin{equation}\label{eq:HWI-classic}
     \ent(\mu_0|e^{-V}\Leb) - \ent(\mu_1|e^{-V}\Leb)\leq W_2(\mu_0,\mu_1)\sqrt{I(\mu_0|e^{-V}\Leb)} - \frac{\lambda}{2}W_2(\mu_0,\mu_1)^2\;.
 \end{equation}
 \end{thm} 

 \begin{proof}
 We refer to \cite{OV00}. In the form stated here, the inequality follows from the characterisation of the subdifferential of geodesically convex functionals in the Wasserstein space in \cite[Sec.~10.1.1, (10.1.6)]{AGS08} together with the Cauchy-Schwartz inequality and the identification of the norm of the minimal subdifferential of the relative entropy in \cite[Thm.~10.4.9]{AGS08}. 
 \end{proof}

We recall the geometry on the configuration spaces introduced by Albeverio-Kondratiev-R\"ockner in \cite{AKR98}. 
For a cylinder function $\phi:\Gamma\to \IR$ 
of the form $\phi(\xi)=G\big[\big(\int f_i d\xi \big)_{i\le m}\big]$ for some fixed $m\in\IN$, $G\in\mathcal C_b(\IR^m)$ and $f_i\in\mathcal C_c(\IR^d)$
the gradient is defined as $\nabla \phi(\xi,x)=\sum_{i=1}^n \partial_i G\big[\big(\int f_i d\xi \big)_{i\le m}\big]\nabla f_i(x)$. The quadratic form
\[E(\phi):= \int_\Gamma\int_{\IR^d}|\nabla \phi|^2(\xi,x)d\xi(x) d\Poi(\xi)\]
defined on cylinder functions is closable and gives rise to a Dirichlet form $(E,D(E))$ on $L^2(\Gamma,\Poi)$.
For a function $\phi\in D(E)$ we denote the carré du champs operator of $E$ by $|\nabla\phi|^2$.
For a point process $\P$ and a subset $A\subset \IR^d$ we define its relative Fisher information on $A$ by
\[\gls{I}\big(\P_A\big\vert\Poi_A\big):= 4 E\big(\sqrt{\rho}\big)= \int_{\Gamma_A}|\nabla \sqrt{\rho}|^2d \Poi_A\;,\]
provided $\P_A = \rho \Poi_A$ is absolutely continuous w.r.t.\ $\Poi_A$ and $\sqrt\rho\in D(E)$.

\begin{defi}\label{def:I_alt}
The \emph{specific (relative) Fisher information} of a stationary point process $\P$  is defined by
\begin{align}
\gls{specI} (\P):=
\limsup_{n\to\infty}\frac 1 {n^d}
 I(\P_{\Lambda_n}|\Poi_{\Lambda_n})\;.
\end{align}
 \end{defi}

We note that the expression $I(\P_{\Lambda_n}|\Poi_{\Lambda_n})$ can be expressed more explicitly e.g. by identifying $\P_{\Lambda_n}$ with measures on Euclidean spaces. For instance, recall that we denote by $\P_{\Lambda_{n}}^k$ the distribution of $\P_{\Lambda_n}$ conditioned on having $k$ points in $\Lambda_n$. Denote by $p_n(k)$ the probability that $\P_{\Lambda_n}$ has $k$ points in $\Lambda_n$. As in Section \ref{sec:EVI}, we denote by $O_n^k\subset(\IR^d)^k$ the orthant and by $\mu_n^k$ the image of $\P_{\Lambda_n}^k$ under the lexicographic ordering map. Alternatively, we can identify $\P_{\Lambda_n}^k$ with a symmetric probability measure $\nu_n^k$ on $(\Lambda_n)^k$. Then, we have
\begin{equation}\label{eq:Fisher-repEuclidean}I(\P_{\Lambda_n}|\Poi_{\Lambda_n})
=
\sum_{k=0}^\infty p_n(k)\cdot I(\mu_n^k|\Leb_{O_n^k})
=
\sum_{k=0}^\infty p_n(k)\cdot I(\nu_n^k|\Leb_{\Lambda_n^k})
\;.\end{equation}

We now have the following result in analogy to the classical HWI inequality.

\begin{thm}\label{thm:HWI2}
Let $\P_0,\P_1\in\spp1$ have finite cost $\W_2(\P_0,\mathsf \P_1)<\infty$, finite entropies $\mathcal E(\P_0),\mathcal E(\P_1)<\infty$ and finite Fisher information $\mathcal I(\P_0)<\infty$. Then, the following HWI-inequality holds
\begin{align}\label{eq:HWI2}
    \mathcal E(\P_0)-\mathcal E(\P_1)\le \W_2(\P_0,\P_1)\sqrt{\mathcal I (\P_0)}.
\end{align}
\end{thm}

\begin{proof}
We consider the processes $\tilde\P_{0,n},\tilde\P_{1,_n}$ defined on $\Lambda_n$ given by Theorem \ref{thm:Modification} for the pair $\P_0,\P_1$. Let $\tilde \P_{0,n}^k,\tilde \P_{1,n}^k$ be their conditional distribution on having $k$ points, $\tilde p_n(k)$ the probability of having $k$ points (identical for $\tilde \P_{0,n}$ and $\tilde \P_{1,n}$ by construction) and $\tilde\mu_{0,n}^k,\tilde\mu_{1,n}^k\in\mathcal{P}(O_n^k)$ their images under the lexicographic ordering.
From Section \ref{sec:EVI}, in particular \eqref{eq:calc_wasser}, \eqref{eq:identification_entropy_t1}, we recall that for $i=0,1$ and a suitable constant $\mathrm{const}(n)$:
\begin{equation}\label{eq:split-ent-dist}
\begin{split}
   \ent(\tilde\P_{i,n}^k\vert \Poi(\cdot|\pi=k)) &= \ent(\tilde\mu_{i,n}^k|\Leb_{O_n^k}) +\log|O_n^k|\;, \qquad C_{\Lambda_n}^k(\tilde\P_{0,n}^k,\tilde\P_{1,n}^k)= W_2^2(\tilde\mu_{0,n}^k,\tilde\mu_{1,n}^k)\;,\\
\ent(\tilde\P_{i,n}\vert\Poi|_{\Lambda_n})&= \sum_{k=0}^\infty p_n(k) \ent(\tilde\P_{i,n}^k\vert \Poi(\cdot|\pi=k))+\mathrm{const}(n)\;.   
\end{split}
\end{equation}
We will show below that $I(\tilde\P_{0,n}|\Poi_{\Lambda_n})$ is finite and hence $I(\tilde\mu_{0,n}^k\vert\Leb_{O_n^k})$ is finite. Then, from the finite-dimensional HWI inequality Theorem \ref{thm:HWI-classic} we infer 
\begin{equation}\label{eq:HWI-nk}
    \ent(\tilde\mu_{0,n}^k|\Leb_{O^k_n}) - \ent(\tilde\mu_{1,n}^k|\Leb_{O^k_n})\leq W_2(\tilde\mu_{0,n}^k,\tilde\mu_{1,n}^k)\sqrt{I(\tilde\mu_{0,n}^k|\Leb_{O^k_n})}\;.
\end{equation}
Hence, by multiplying \eqref{eq:HWI-nk} by $p_n(k)$, summing over $k$, using Cauchy-Schwartz and \eqref{eq:split-ent-dist}, \eqref{eq:Fisher-repEuclidean} we obtain
\begin{equation*}
    \ent(\tilde\P_{0,n}\vert\Poi|_{\Lambda_n})-\ent(\tilde\P_{1,n}\vert\Poi|_{\Lambda_n}) \leq \sqrt{\sum_{k=0}^\infty p_n(k)\C_{\Lambda_n}^k(\tilde\P_{0,n}^k,\tilde\P_{1,n}^k)}\sqrt{I(\tilde\P_{0,_n}|\Poi_{\Lambda_n})}\;.
\end{equation*}
From Theorem \ref{thm:Modification} \eqref{thm:mod_proc_entropy} and \eqref{thm:mod_proc_cost} we get upon dividing by $n^{-d}$ and sending $n\to\infty$
\begin{equation}\label{eq:preHWIapprox}
    \mathcal{E}(\P_0)-\mathcal{E}(\P_1) \leq W_2(\P_0,\P_1)\left(\limsup_{n\to\infty}\frac{1}{n^d}I(\tilde\P_{0,_n}|\Poi_{\Lambda_n})\right)^{\frac12}\;.
\end{equation}
It thus remains to show that the term in brackets coincides with $\mathcal I(\P_0)$.
To this end, we  drop the subscript $0$, follow the route of Theorem \ref{thm:Modification} and use the family of  disjoint events 
$B_{l,k}=\{\xi\in\Gamma_{\Lambda_n}:\xi(\Lambda_n)=l, \xi(K_i)=k_i\forall i\le N\}$.
Let us disintegrate $\tilde \P_{\Lambda_n}$ with respect to $\xi\mapsto\xi|_{\Lambda_{n-1}}$, which means in terms of densities that
\begin{align*}
    \frac{d\tilde\P_{\Lambda_n}}{d\Poi_{\Lambda_n}}(\xi)=\frac{d\tilde\P_{\Lambda_n}\big(\cdot|\{\cdot|_{\Lambda_{n-1}}=\xi|_{\Lambda_{n-1}}\}\big) }{d\Poi_{\Lambda_n}\big(\cdot|\{\cdot|_{\Lambda_{n-1}}=\xi|_{\Lambda_{n-1}}\}\big)}(\xi)\cdot  \frac{d\tilde\P_{\Lambda_{n-1}}}{d\Poi_{\Lambda_{n-1}}}(\xi|_{\Lambda_{n-1}}).
\end{align*}
Recall that \eqref{eq:modified_density} states that for all $\xi\in B_{l,k}$ it holds
\begin{align*}
    \frac{d \tilde{\P}_{\Lambda_n}(\cdot\mid \{\cdot|_{\Lambda_{n-1}}=\xi_{\Lambda_{n-1}}\})}{d\Poi_{\Lambda_n}\big(\cdot|\{\cdot|_{\Lambda_{n-1}}=\xi_{\Lambda_{n-1}}\}\big)}(\xi)=\tilde \P_{\Lambda_n}(B_{l,k})e^N\prod_i k_i!.
\end{align*}
 Also recall that $\tilde\P_{\Lambda_{n-1}} = \P_{\Lambda_{n-1}}$ by construction.

Hence,
\begin{align*}
    I(\tilde\P_{n}|\Poi_{\Lambda_n})&=
    \int_{\Gamma_{\Lambda_n}} \left |\nabla \sqrt{\frac{d\tilde\P_{\Lambda_n}}{d\Poi_{\Lambda_n}}(\xi)}\right |^2d \Poi_{\Lambda_n}(\xi)\\
    &= \int_{\Gamma_{\Lambda_{n-1}}}\int_{\Gamma_{\Lambda_{n}}} \left |\nabla \sqrt{\frac{d\tilde\P_{\Lambda_n}}{d\Poi_{\Lambda_n}}\left(\xi \right)}\right |^2d \Poi_{\Lambda_n}(\xi \mid \xi|_{\Lambda_{n-1}}=\tilde{\xi}) d \Poi_{\Lambda_{n-1}}(\tilde{\xi} )\\
     &= \int_{\Gamma_{\Lambda_{n-1}}}\sum_{l,k}\int_{B_{l,k}} \left |\nabla \sqrt{\frac{d\tilde\P_{\Lambda_n}}{d\Poi_{\Lambda_n}}\left(\xi \right)}\right |^2d \Poi_{\Lambda_n}(\xi \mid \xi|_{\Lambda_{n-1}}=\tilde{\xi}) d \Poi_{\Lambda_{n-1}}(\tilde{\xi} )\\
     &= \int_{\Gamma_{\Lambda_{n-1}}}\sum_{l,k}\tilde \P_{\Lambda_n}(B_{l,k})e^N\prod_i k_i!\int_{B_{l,k}} \left |\nabla \sqrt{\frac{d\tilde\P_{\Lambda_{n-1}}}{d\Poi_{\Lambda_{n-1}}}(\xi|_{\Lambda_{n-1}})}\right |^2d \Poi_{\Lambda_n}(\xi \mid \xi|_{\Lambda_{n-1}}=\tilde{\xi}) d \Poi_{\Lambda_{n-1}}(\tilde{\xi} )\\
      &= \int_{\Gamma_{\Lambda_{n-1}}}\sum_{l,k}\tilde \P_{\Lambda_n}(B_{l,k})e^N\prod_i k_i!\int_{B_{l,k}} \left |\nabla \sqrt{\frac{d\P_{\Lambda_{n-1}}}{d\Poi_{\Lambda_{n-1}}}(\tilde{\xi})}\right |^2d \Poi_{\Lambda_n}(\xi \mid \xi|_{\Lambda_{n-1}}=\tilde{\xi}) d \Poi_{\Lambda_{n-1}}(\tilde{\xi} )\\
    &= \int_{\Gamma_{\Lambda_{n-1}}}\left |\nabla \sqrt{\frac{d\P_{\Lambda_{n-1}}}{d\Poi_{\Lambda_{n-1}}}(\tilde{\xi})}\right |^2 \sum_{l,k}\int_{B_{l,k}} \tilde \P_{\Lambda_n}(B_{l,k})e^N\prod_i k_i!d \Poi_{\Lambda_n}(\xi \mid \xi|_{\Lambda_{n-1}}=\tilde{\xi}) d \Poi_{\Lambda_{n-1}}(\tilde{\xi} )\\
    &= \int_{\Gamma_{\Lambda_{n-1}}}\left |\nabla \sqrt{\frac{d\P_{\Lambda_{n-1}}}{d\Poi_{\Lambda_{n-1}}}(\tilde{\xi})}\right |^2 \sum_{l,k}\int_{B_{l,k}} \frac{d \tilde{\P}_{\Lambda_n}(\cdot\mid \{\cdot|_{\Lambda_{n-1}}=\tilde{\xi}\})}{d\Poi_{\Lambda_n}\big(\cdot|\{\cdot|_{\Lambda_{n-1}}=\tilde{\xi}\}\big)}(\xi)d \Poi_{\Lambda_n}(\xi \mid \xi|_{\Lambda_{n-1}}=\tilde{\xi}) d \Poi_{\Lambda_{n-1}}(\tilde{\xi} )\\
     &= \int_{\Gamma_{\Lambda_{n-1}}}\left |\nabla \sqrt{\frac{d\P_{\Lambda_{n-1}}}{d\Poi_{\Lambda_{n-1}}}(\tilde{\xi})}\right |^2  d \Poi_{\Lambda_{n-1}}(\tilde{\xi} )\\
     &=I(\P_{\Lambda_{n-1}}|\Poi_{\Lambda_{n-1}}).
\end{align*}
 Since $\mathcal I(\P_0)$ was assumed finite, the last term is indeed finite for  all but finitely many $n$. Moreover,
this yields \[
\limsup_{n\to\infty}\frac{1}{n^d}I(\tilde\P_{n}|\Poi_{\Lambda_n})=\limsup_{n\to\infty}\frac{1}{n^d}I(\P_{\Lambda_n}|\Poi_{\Lambda_n})
\]
and  the claim is proved.

\end{proof}

\begin{ex}\label{ex:grid3}
The stationarized perturbed grid of Examples \ref{ex:grid} and \ref{ex:grid2} have finite specific relative Fisher information if the distribution of the perturbation is smooth. More precisely, define 
\[
    \mathsf{grid}=\theta_Y\left(\sum_{z\in \IZ^d}\delta_{z+X_z}\right),
    \]
    where the perturbations $(X_z)_{z\in \IZ^d}$  are i.i.d.~random variables on $\IR^d$, whose distribution have Lebesgue-density $f\in\mathcal C_c^\infty$ with $\supp f\subseteq \Lambda_1$, and the independent shift $Y$ is uniformly distributed on $\Lambda_1$. Then, the density of  the distribution of $\mathsf{grid}$ on a box $\Lambda_n$ with respect to the distribution of the Poisson process, evaluated at a configuration $\theta_Y(\xi)=\sum_{z\in\Lambda_n\cap \IZ^d}\delta_{z+Y+x_z}$ is given by
    $$\frac {d\P_{\mathsf{grid},\Lambda_n}}{d\Poi_{\Lambda_n}}(\theta_Y(\xi))=\prod_{z\in\Lambda_n\cap \IZ^d}e^1f(x_z)\eins_{\xi|_{\Lambda_1(z)}=\delta_{x_z}}.$$
A simple calculation then yields
$$\mathcal I (\P_{\mathsf{grid}})=\int_{\IR^d}\frac{|\nabla f(x)|^2}{f(x)}dx,$$
which coincides with the classical Fisher-Information of the perturbation $X_0$. We omit the details as an exercise to the  reader.
\end{ex}

\printnoidxglossary[type=symbols,style=long3col,title={List of Symbols}]\label{sec:Symbols}

\end{document}